\documentclass{amsart}
\usepackage{amsmath,amssymb}
\usepackage[all]{xy}
\usepackage{xcolor}
\usepackage{comment}

\excludecomment{confidential}

\newtheorem{theorem}{Theorem}[section]
\newtheorem{lemma}[theorem]{Lemma}
\newtheorem{corollary}[theorem]{Corollary}

\newtheorem{proposition}[theorem]{Proposition}
\newtheorem{claim}[theorem]{Claim}
\newtheorem{Sclaim}[theorem]{Subclaim}
\newtheorem{SSclaim}[theorem]{Subsubclaim}
\newtheorem{conjecture}[theorem]{Conjecture}

\theoremstyle{definition}
\newtheorem{example}[theorem]{Example}

\newtheorem{remark}[theorem]{Remark}
\newtheorem{definition}[theorem]{Definition}
\newtheorem{assumption}[theorem]{Assumption}
\newtheorem{notation}[theorem]{Notation}

\def \deg {\operatorname{deg}}

\def \supp {\operatorname{supp}}
\def \DD {\mathcal D}
\def \mm {\mathfrak m}

\def \dd {\partial}
\def \md {\mathfrak d}
\def \mn {\mathfrak n}

\def \bDD {\underline{\mathcal D}}
\def \gD {{\mathcal D}^\Gamma}
\def \gDD {\underline{\mathcal D}^\Gamma}
\def \be {\underline{e}}

\def \dcl{\operatorname{dcl}}
\def \acl{\operatorname{acl}}

\def \id{\operatorname{id}}

\def \spec{\operatorname{Spec}}
\def \Null{\operatorname{Null}}

\def \NN {\mathbb N}
\def \ine {\triangleleft}
\def \ineq {\trianglelefteq}
\def \s {\sigma}

\def\Ind#1#2{#1\setbox0=\hbox{$#1x$}\kern\wd0\hbox to 0pt{\hss$#1\mid$\hss}
\lower.9\ht0\hbox to 0pt{\hss$#1\smile$\hss}\kern\wd0}
\def\ind{\mathop{\mathpalette\Ind{}}}
\def\Notind#1#2{#1\setbox0=\hbox{$#1x$}\kern\wd0\hbox to 0pt{\mathchardef
\nn=12854\hss$#1\nn$\kern1.4\wd0\hss}\hbox to
0pt{\hss$#1\mid$\hss}\lower.9\ht0 \hbox to
0pt{\hss$#1\smile$\hss}\kern\wd0}

\title[]{Fields with Lie-commuting and iterative operators}

\author{Jan Dobrowolski}
\address{Jan Dobrowolski, Department of Mathematics, Xiamen University Malaysia, Jalan Sunsuria, Bandar Sunsuria, 43900 Sepang, Selangor Darul Ehsan, Malaysia and \newline
Instytut Matematyczny
Uniwersystetu Wroc\l{}awskiego, plac Grunwaldzki 2, 50-384 Wroc\l{}aw, Poland}
\email{jan.dobrowolski@math.uni.wroc.pl}

\author{Omar Le\'on S\'anchez}
\address{Omar Le\'on S\'anchez, Department of Mathematics, University of Manchester, Oxford Road, Manchester, United Kingdom M13 9PL}
\email{omar.sanchez@manchester.ac.uk}

\date{\today}
\thanks{{\em Acknowledgements}: Both authors were partially supported by EPSRC grant EP/V03619X/1}

\begin{document}

\maketitle

\begin{abstract}
We introduce a general framework for studying fields equipped with operators, given as co-ordinate functions of homomorphisms into a local algebra $\DD$, satisfying various compatibility conditions that we denote by $\Gamma$ and call such structures $\gD$-fields. These include  Lie-commutativity of derivations and $\mathfrak g$-iterativity of (truncated) Hasse-Schmidt derivations. Our main result is about the existence of principal realisations of $\gD$-kernels. As an application, we prove companionability of the theory of $\gD$-fields and denote the companion by $\gD$-CF. In characteristic zero, we prove that $\gD$-CF is a stable theory that satisfies the CBP and Zilber's dichotomy for finite-dimensional types. We also prove that there is a uniform companion for model-complete theories of large $\gD$-fields, which leads to the notion of $\gD$-large fields and we further use this to show that PAC substructures of $\gD$-DCF are elementary.
\end{abstract}

\tableofcontents

\section{Introduction}\label{sec1}

The study of additive operators (such as derivations, higher derivations, automorphisms and, more generally, endomorphisms) on fields has been a central theme of research in algebra ever since the foundational work of Ritt on differential algebra and difference algebra in the 1930's \cite{Ritt32,Ritt34,Ritt35,Ritt33,Ritt39}, which was motived by studying differential equations and difference equations with coefficients in function fields.
Ritt’s work on differential algebra was vigorously continued by numerous mathematicians—such as Raudenbush \cite{Raudenbush34}, Levi \cite{Levi42}, and most prominently Kolchin \cite{Kolchin42,Kolchin}—while major developments in difference algebra began in the 1950's with Cohn’s work \cite{Cohn51, Cohn52, Cohn53}. In addition, the notion of a higher derivation was introduced in the work of Hasse and Schmidt \cite{Hasse37}, with the aim of establishing  a positive-characteristic analogue of Taylor series of a smooth function.

\smallskip

In \cite{Buium97}, Buium introduced the notion of a jet operator, generalising both the differential operators and the difference operators. In \cite{MooScan2011,MooScan2014}, Moosa and Scanlon made a further generalisation, introducing the concept of a system of additive operators on a field, in which instead of the Leibniz rule (assumed in the case of a derivation) or the multiplicativity condition (assumed in the case of an endomorphism) one considers product rules determined by an arbitrary finite-dimensional algebra over a fixed base field $k$ (see Section~\ref{sec2} for a precise definition). For example, the algebras $k[x]/(x)^n$ with $n\geq 2$ yield the rules defining higher derivations. While most of~\cite{MooScan2014} focuses on the case of characteristic zero, the positive characteristic case was treated  more extensively  in \cite{BHKK2019} by Beyarslan, Hoffmann, Kamensky, and Kowalski.

%In the last two papers mentioned above, the authors focussed on model-theoretic properties of the corresponding theories of fields with operators, in particular on the question of existence of a \emph{model companion}, whose presence allows to treat all considered structures (of a particular kind) as substructures of a fixed universal domain having some desirable properties. 

\smallskip

The operators considered in \cite{BHKK2019} and \cite{MooScan2011} were not assumed to satisfy any compatibility with each other, but such requirements occur very naturally in many algebraic settings -- for example, in a differential field equipped with more than one derivation it is usually assumed that the derivations commute with each other (e.g. in \cite{Kolchin}), or at least to Lie-commute.
In the present paper, we introduce a fairly general framework for studying operators in the sense of Moosa-Scanlon~\cite{MooScan2011} satisfying additional compatibility conditions that we call $\Gamma$ (see Sections \ref{sec3} and~\ref{sec4}). We establish both algebraic and model-theoretic results about them: in the algebraic part, we study the notion of a \emph{$\gD$-kernel} for a local algebra $\DD$, which is a natural generalisation of the notion of a differential kernel~\cite{Lando1970}, and in the model-theoretic part, we are concerned with existence and the properties of model companions of the corresponding first-order theories. We discuss both aspects in more detail below.

\smallskip

The notion of a differential kernel was first studied by Lando \cite{Lando1970} and Cohn~\cite{Cohn}, and was utilised to make progress on the Jacobi bound and the Ritt problem, the former conjectures an upper bound for the order of zero-dimensional irreducible components of a differential variety, while the latter problem asks when a given irreducible differential variety is in the general solution of an algebraically irreducible differential polynomial (see e.g. \cite[p.190]{Kolchin}). While the above studies deal with the case of a field equipped with a single derivation, the concept of a differential kernel was later generalised  to the setting of a field equipped with several commuting derivations by Pierce \cite{Pierce2014}, who then used it to characterise when a system of differential equations over a field $K$ equipped with commuting derivations $\dd_1,\dots,\dd_n$ is consistent; that is, when it has a solution in some  differential field extension of $K$. Briefly, a differential kernel of length $r$ over the differential field $K$ corresponds to a tower of finitely generated field extensions $K=L_{-1}\subseteq L_0\subseteq L_1 \dots \subseteq L_r$ with each $L_i$ (where $i<r$) equipped with derivations $\dd_{i,1},\dots,\dd_{i,n}:L_i\to L_{i+1}$  such that $\dd_{i+1,j}$ extends $\dd_{i,j}$ and $\dd_{r-1,j} \dd_{r-2,k}=\dd_{r-1,k}\dd_{r-2,j}$ on $L_{r-2}$ for all $1\leq j,k\leq n$. We say that such a kernel has a \emph{regular realisation} if there is a differential field extension $K\leq L$ containing $L_r$ such that the differential structure on $L$ agrees with $\dd_{r-1,1},\dots, \dd_{r-1,n}$.

\smallskip

In \cite{Pierce2014}, Pierce proved a kernel-prolongation theorem, which states that for every $r$ and every differential kernel $L$ of length $r$, if $L$ has a generic prolongation of length $2r$ (that is, a differential kernel of length $2r$ extending $L$ in which the only algebraic relations %between the generators
 are the ones obtained by differentiating the algebraic relations holding  in $L$) then $L$ has a generic prolongation of arbitrary length, and hence it admits a \emph{principal realisation}; that is, a generic regular realisation. As a consequence, one can show that a system of partial differential equations  
 $$f_1(x_1,\dots,x_m)=\dots=f_\ell(x_1,\dots,x_m)=0$$ 
 is consistent if and only if differentiating it $C$-many times gives a consistent system of polynomial equations in the algebraic variables $\dd_1^{i_1}\dots \dd_n^{i_n}x_j$, where $C$ is a constant depending only on the complexity of the system of equations (see 
 \cite[Theorem 4.10]{Pierce2014} or \cite[Theorem 11]{GLS2018}).
 
 \smallskip
 
In this paper we extend the above results to a general context of fields equipped with operators, in the sense of Moosa-Scanlon, which  satisfy a \emph{Jacobi-associative} commutativity rule -- we define those in Sections \ref{sec3} and~\ref{sec4}; usual commutativity of the operators is  an example of such a rule, among many others. 

\smallskip

In the case of fields equipped with two commuting automorphisms, 
%which corresponds to the algebra  $k\times k\times k$, 
it is known that the above kernel realisation theorem does not hold - it would imply existence of a model companion, which is known to fail for this class by a result of Hrushovski (see e.g. \cite{Sjoergen} for a proof). We thus restrict ourselves to the case of operator systems coming from local algebras, excluding then nontrivial endomorphisms. We will consider two types of commutativity rules; namely, Jacobi rules (generalising Lie-commutativity of derivations) and associative rules (generalising iterativity of Hasse-Schmidt derivations), hence we will work with two sets of operators coming from local algebras $\DD_1$ and $\DD_2$ over the base field $k$. The operators associated to $\DD_1$ will be assumed to satisfy a Lie-commutativity rule (where the associated coefficients obey a Jacobi condition), the ones associated to $\DD_2$ will be assumed to satisfy an iterativity rule (where the associated coefficients obey an associativity condition), and the operators associated to $\DD_1$ will be assumed to commute with those associated to $\DD_2$. 
Altogether, we will call such a set of conditions a \emph{Jacobi-associative commutativity rule}, and we will usually fix such a rule and call it $\Gamma$. 
We will also write $\bDD=(\DD_1,\DD_2)$ and call a field equipped with such a structure a $\gDD$-field (see Section~\ref{sec4}).

\smallskip

Examples of classes falling into our framework of fields with local operator-systems satisfying a Jacobi-associative commutativity rule include: 
\begin{itemize} 
	\item fields with Lie-commuting derivations (studied in \cite{Yaffe2001}),
	\item fields with truncated iterative Hasse-Schmidt derivations (studied in \cite{Kow2005}), 
	\item fields with $\mathfrak g$-derivations for a finite group scheme $\mathfrak g$ (studied in \cite{HK}),
	\item fields with commuting operators associated to an arbitrary local algebra (recently studied in \cite{Burton})
	%over a field of characteristic $0$, and
	%\item fields with commuting operators associated to an arbitrary  local algebra $\DD$ over a field of characteristic $p$ in which the maximal ideal coincides with the kernel of the Frobenius morphism.
	\end{itemize}
	
We will see in Section~\ref{sec5} that the concept of a differential kernel has a natural generalisation to the notion of a $\bDD^\Gamma$-kernel for systems of $\bDD$-operators  satisfying a Jacobi-associative commutativity rule $\Gamma$. Our main algebraic result, Theorem \ref{thebigone},  states that if a $\bDD^\Gamma$-kernel of length $r$ has a generic prolongation of length $2r$, then it has a principal realisation, hence generalising the aforementioned result of Pierce. %This then allows to deduce the existence of constants $C^n_{r,m}$ ({\bf shall we add such a statement?})

\smallskip

From the model-theoretic side, the development of the model theory of fields with operators was initiated by Robinson's work on differentially closed fields \cite{Robinson}, and has rapidly accelerated in recent decades,  %(more specifically, model theory of fields with derivations and of fields with an automorphism) 
finding several remarkable applications in various branches of mathematics, such as diophantine geometry \cite{11, 13}, algebraic
dynamics \cite{8, 12}, Galois theory \cite{22, 27}, and representation theory of
algebras \cite{2, 3, 16}. 
In those applications, one of the fundamental features of the underlying first-order theory is its \emph{companionability} (i.e., the existence of a model companion). %, which allows to treat all considered structures as subsets of a universal domain; 
For instance, companionability of the theory of fields with an automorphism  was proved by Macintyre in \cite{ACFA}, companionability of the theory of fields equipped with a single derivation was proved by Robinson in  \cite{Robinson}, and companionability of the theory of fields (in arbitrary characteristic) equipped with several commuting derivations was proved by Pierce in \cite{Pierce2014}. For a comprehensive survey on the model theory of fields with operators we refer the reader to \cite{Chatoperators}.
%, and for the theory of fields with arbitrary $n$-many derivations. 
In Section \ref{sec6}, using our results on $\bDD^\Gamma$-kernels from Section~\ref{sec5}, we prove that the theory of $\gDD$-fields 
%with $\Gamma$-commuting $\bDD$-operators 
has a model companion (denoted $\bDD^\Gamma$-CF) in characteristic zero, and that in characteristic $p>0$ it has a model companion if either $\dim_k(\DD_1)=1$ or
 the maximal ideal of $\DD_u$ is equal to the kernel of the Frobenius homomorphism $\operatorname{Fr}:\DD_u\to \DD_u$ for $u=1,2$ (we note that this latter condition appears in \cite{BHKK2019} and is necessary for the results there). % and also when $\dim_k(\DD_1)=1$.
 
 \smallskip
 
 Furthermore, in characteristic zero, we prove the theory $\bDD^\Gamma$-CF has a number of desirable properties such as completeness, quantifier-elimination, $|k|$-stability (where $k$ is the base field), elimination of imaginaries, the Canonical Base Property, and (the expected form of) Zilber's Dichotomy for finite-dimensional types.

\smallskip

In Section \ref{sec7}, we refine our companionability result by proving that, in characteristic zero, for an arbitrary local system $\bDD$ and a Jacobi-associative commutativity rule $\Gamma$, there is a theory UC$_{\gDD}$ axiomatising those large $\gDD$-fields that are existentially closed in every extension in which they are existentially closed as a field. We then observe that this generalises Tressl's uniform companion result from \cite{Tressl2005}.
We also prove that, for a natural notion of $\gDD$-largeness (generalising differential largeness \cite{LSTr2023,LSTr2024}), the PAC-substructures in $\gDD$-CF are precisely those $\gDD$-fields that are PAC (as fields) and $\gDD$-large.

%We now briefly explain the concept of a model companion. Let $T$ be a theory axiomatising some class of fields equipped with extra algebraic structure (such as a derivation) which is closed under taking unions of increasing chains. Then a theory $T^c$ is called the \emph{model companion} of $T$, if it axiomatises the class of \emph{existentially closed} models of $T$, where a model $M$ of $T$ is called existentially closed if it contains a solution to every system of equations that can be solved in any extension of $M$. As fundamental examples, the model companion of the theory of fields is the theory of algebraically closed fields, and the model companion of the theory of formally real fields is the theory of real closed fields. 

%In this paper we introduce a framework for studying the model theory of fields with operators satisfying abstract Leibniz rules induced by a local algebra and a certain compatibility condition between them. Examples of classes falling into this framework include fields with Lie-commuting derivations (studied in \cite{Yaffe2001}), fields with iterative Hasse-Schmidt derivations (studied in \cite{Zi}), fields with finite group scheme actions (studied in \cite{HK}) and fields with arbitrary commuting operators coming from a local algebra. Our main result is Corollary \ref{cor_comp} where we show the existence of model companion in this setting. The main step towards that result is Theorem \ref{thebigone} about extending certain partially defined operators. (write more about kernels ??)

\smallskip

Let us mention that while our current setup does not include the case of automorphisms, based on results for differential-difference fields \cite{LeonSan,InoLS}, we expect that the theory $\gDD$-CFA does exist. Namely, that the theory of $\gDD$-fields equipped with an automorphism (commuting with the operators) has a model companion. We leave this for future work.

%{\bf Maybe we mention at the end of the introduction that in the current setup we do not include automorphisms, but that we will explore this, i.e. the existence $\gDD$-CFA, in a future paper??}

\smallskip

\noindent {\bf Conventions.} Throughout $k$ is a field. We assume rings are commutative and unital, and algebras are associative (unless stated otherwise). Also, for us $\NN=\{1,2,\dots\}$ while $\NN_0=\NN\cup\{0\}$.

\

\section{Preliminaries and notation}\label{sec2}

Let $\DD$ be a local finite dimensional $k$-algebra (recall that $k$ is a field). We let $m$ be such that $\dim_k\DD=m+1$. Recall that local means that $\DD$ has a unique maximal ideal $\mm$. As $\DD$ is finite dimensional, $\mm$ is nilpotent. We let $d$ be the smallest nonnegative integer such that $\mm^{d+1}=0$. Assume that the residue field $\DD/\mm$ is $k$ and denote the residue map by $\pi:\DD\to k$. Let $d_i=\dim_k(\mm^i/\mm^{i+1})$ for $i=0,\dots,d$. Note that then $d_0=1$. Let $D_{-1}=0$ and $D_j=\sum_{i=0}^j d_i$ for $0\leq j\leq d$.

\medskip

With $\DD$ as above, one can find a $k$-linear basis of $\DD$ of the form 
$$(\epsilon_0=1,\epsilon_1,\dots,\epsilon_m)$$
such that
$(\epsilon_{D_{i-1}},\dots,\epsilon_{D_i-1})$
yields a basis for $\mm^i/\mm^{i+1}$ for $i=0,\dots,d$. We call any such basis a \emph{ranked basis} for $\DD$. In this case, for any $1\leq p\leq m$, we define $\sigma(p)$ to be the unique positive integer such that 
$$\epsilon_p\in \mm^{\sigma(p)}\setminus \mm^{\sigma(p)+1}.$$
Note that with this notation we have $\epsilon_p\cdot \epsilon_q\in \mm^{\sigma(p)+\sigma(q)}$ for $1\leq p,q\leq m$.

\medskip

\begin{definition}\label{def_op}
By a local operator-system we mean a pair $(\DD,\bar\epsilon)$ where $\DD$ is a local finite dimensional $k$-algebra $\DD$ with residue field $k$ and $\bar\epsilon=(1,\epsilon_1,\dots,\epsilon_m)$ is a ranked basis for $\DD$. Note that for such a $\DD$, with structure map $\iota:k\to \DD$, there is a unique  $k$-algebra homomorphism $\pi:\DD\to k$ such that $\pi\circ \iota=\operatorname{Id}_k$; namely, $\pi$ is the residue map. 
\end{definition}

We now fix a local operator-system $(\DD,\bar\epsilon)$ with residue map $\pi:\DD\to k$. For each $k$-algebra $R$, we denote by $\DD(R)$ the base change of $\DD$ from $k$ to $R$. Namely, $\DD(R)=\DD\otimes_k R$. We will in fact think of $\DD$ as a functor on the category of $k$-algebras where a $k$-algebra homomorphism $\phi:R\to S$ is canonically lifted to $\DD(\phi):\DD(R)\to \DD(S)$ (i.e. $\DD(\phi)=\operatorname{Id}_\DD\otimes \, \phi$).

\begin{definition}
Let $R\xrightarrow{\phi} S$ be a homomorphism of $k$-algebras. A $k$-algebra homomorphism $e:R\to \DD(S)$ is said to be a $\DD$-operator from $R$ to $S$ with respect to $\phi$ if $\pi^{S}\circ e=\phi$. Here $\pi^S$ is the base change of $\pi$ from $k$ to $S$. Such a $\DD$-structure is commonly denoted by $(R\xrightarrow{\phi} S,e)$. When $R$ is a subring of $S$ and the inclusion $R\hookrightarrow S$ is a $k$-algebra map, a $\DD$-operator from $R$ to $S$ with respect to the inclusion is simply called a $\DD$-operator and we denote it by $e:R\to \DD(S)$. If in addition $R=S$, we say that $R$ is a $\DD$-ring and we denote it by $(R,e)$. 
\end{definition}

A $k$-algebra $R$ can always be equipped with the trivial $\DD$-ring structure. Namely, with the lifted map $\iota^R:R\to \DD(R)$. Note that $\iota^R$ is simply the canonical embedding $r\mapsto 1\otimes r$ from $R\hookrightarrow \DD\otimes R$. We denote this trivial structure by $(R,\iota)$.

\medskip

Let $(R,e)$ be a $\DD$-ring and $e':S\to \DD(T)$ a $\DD$-operator (in particular $S$ is a subring of $T$ and the inclusion $S\hookrightarrow T$ is a $k$-algebra map). A $k$-algebra homomorphism $\phi:R\to S$ is said to be a $\DD$-homomorphism if
$$\DD(\phi)\,\circ\, e= e'\,\circ\, \phi.$$
In the case when $S$ is an $R$-algebra and the structure map $\iota:R\to S$ is a $\DD$-homomorphism we say that $(S\hookrightarrow T,e')$ is an $(R,e)$-algebra or a $\DD$-algebra over $(R,e)$.

%\medskip

\begin{remark}\label{operatornotation}
Any $\DD$-operator $e:R\to \DD(S)$ can be written in the terms of the (fixed) ranked basis as
$$e(x)=1\otimes x+\epsilon_1\otimes \dd_1(x)+\cdots+\epsilon_m\otimes \dd_m(x)$$
where $\dd_i:R\to S$, for $i=1,\dots,m$, are additive operators which satisfy a certain form of Leibniz rule; namely, a multiplication rule of the form
$$\dd_i(xy)=\dd_i(x)\, y +x\, \dd_i(y)+\sum_{p,q=1}^m\alpha_{i}^{pq}\dd_p(x)\dd_q(y)$$
where $\alpha_i^{pq}\in k$ is the coefficient of $\epsilon_i$ in the product $\epsilon_p\cdot \epsilon_q$. In addition, as our basis is ranked, we have that $\alpha_i^{pq}=0$ whenever $\sigma(p)+\sigma(q)>\sigma(i)$; thus, 
%if we set
%$$\gamma(i)=\{(p,q): \; 1\leq p,q\leq m \text{ and } \s(p)+\s(q)\leq\s(i)\}$$
%for each $1\leq i\leq m$, 
the Leibniz rule has the (simplified) form
\begin{equation}\label{Leibrule}
\dd_i(xy)=\dd_i(x)\, y +x\, \dd_i(y)+\sum_{\sigma(p)+\sigma(q)\leq \sigma(i)}\alpha_{i}^{pq}\dd_p(x)\dd_q(y)
\end{equation}
In particular note that if $j\neq i$, with $\sigma(j)\geq \sigma(i)$, then $\dd_j$ does \emph{not} appear in the product rule of $\dd_i$.
\end{remark}

\begin{example}\label{firstexamples}
These are the basic examples.
\begin{enumerate}
\item (several derivations) Let $m\in \NN$ and 
$$\DD=k[\epsilon_1,\dots,\epsilon_m]/(\epsilon_1,\dots,\epsilon_m)^2$$
equipped with %$\pi:\DD\to k$ such that $\pi(\epsilon_i)=0$, for $i=1,\dots,m$, and with 
ranked basis $(1,\epsilon_1,\dots,\epsilon_m)$. Then, $\DD$-rings correspond to differential rings with $m$-many derivations (not necessarily commuting).\\
%***
\item (truncated Hasse-Schmidt derivations) Let $m, n\in \NN$ and 
$$\DD=k[\epsilon_1,\dots,\epsilon_m]/(\epsilon_1^{n+1}, \dots, \epsilon_m^{n+1}, (\epsilon_i\cdot\epsilon_j)_{i<j})$$
equipped with %$\pi:\DD\to k$ such that $\pi(\epsilon)=0$ and 
ranked basis $(1,\epsilon_1,\dots,\epsilon_m,\dots,\epsilon_1^{n},\dots,\epsilon_m^{n})$. Then, $\DD$-rings correspond to rings equipped with $m$-many $n$-truncated Hasse-Schmidt derivations (not necessarily iterative nor commuting). We recall that an $n$-truncated H-S derivation is a tuple $(\dd_i)_{i=1}^n$ of additive operators satisfying
$$\dd_i(xy)=\dd_i(x)\, y+ x\, \dd_i(y) + \sum_{p+q=i}\dd_p(x)\dd_q(y).$$
\end{enumerate}
\end{example}

\

From \cite[Lemma 2.7]{BHKK2019}, we know that if $L/K$ is a separable field extension and $e:K\to\DD(L)$  is a $\DD$-operator, then there exists an extension to $e':L\to \DD(L)$; moreover, if $L/K$ is separably algebraic, then the extension is unique. Later on we will need a more detailed description of how these extensions of $\DD$-structures can be constructed; this is given in the following lemma (which is more or less well known).
%Later on we will make use of the following (more or less well-known) result on extending $\DD$-structures on fields. 
Recall that, given a $\DD$-ring $(R,e)$, the ring of $\DD$-constants is 
$$C_R:=\{r\in R:\, e(r)=\iota^R(r)\}=\{r\in R:\, \partial_{i}(r)=0 \text{ for all }1\leq i\leq m\}.$$
Also, we denote by $\operatorname{Fr}$ the Frobenius endomorphism on $\DD$; namely, $\operatorname{Fr}(x)=x^p$ when $char(k)=p>0$.

\begin{lemma}\label{extendstructure}
Let $K\leq L$ a field extension and $e:K\to \DD(L)$ a $\DD$-operator. Assume $a\in L$ and $b\in \DD(L)$ with $\pi(b)=a$. Then, there is a $\DD$-operator $e':K(a)\to \DD(L)$ extending $e$ and mapping $a\mapsto b$ if and only if $f^e(b)=0$ for all $f\in K[x]$ vanishing at $a$ (here $f^e$ is the polynomial over $\DD(L)$ obtained by applying $e$ to the coefficients of $f$). As result, we have
\begin{enumerate}
\item[(i)] if $a$ is separably algebraic over $K$, then there is a unique $\DD$-operator $e':K(a)\to \DD(L)$ extending $e$. 
\item[(ii)] if $a$ is transcendental over $K$, then for any choice of $b\in \DD(L)$ with $\pi(b)=a$ there is a unique $\DD$-operator $e':K(a)\to \DD(L)$ extending $e$ and mapping $a\mapsto b$.
\item[(iii)] if $a$ is inseparably algebraic over $K$, $\mm=\ker(\operatorname{Fr})$, and there exists an extension $\hat e:K(a)\to \DD(L)$ of $e$, then for any choice of $b\in \DD(L)$ with $\pi(b)=a$ there is a unique $\DD$-operator $e':K(a)\to \DD(L)$ extending $e$ and mapping $a\mapsto b$. Furthermore, the existence of the extension $\hat e:K(a)\to \DD(L)$ is equivalent to the minimal polynomial of $a$ having coefficients in $C_K$.
\end{enumerate}
\end{lemma} 
\begin{proof}
(i) As $\DD(K(a))$ is Artinian, Hensel's lemma applies when $a$ is separably algebraic. 

\smallskip

(ii) When $a$ is transcendental over $K$, there are no nontrivial $f\in K[x]$ vanishing at $a$, and so the conditions to extend $e$ so that $a\mapsto b$ are trivially satisfied.

\smallskip

(iii) Let $p=char(k)$ and $f$ be the minimal polynomial of $a$ over $K$. Since $a$ is inseparable, there is $g\in K[x]$ such that $f(x)=g(x^{p})$. The extension $\hat e$ exists iff $f^e(b)=0$ iff $g^e(b^{p})=0$. The assumption $\mm=\ker(\operatorname{Fr})$ yields that $b^{p}=a^{p}$. 
Now, if $g^e(a^{p})=0$ but $g\notin C_K[x]$, then writing
$g(x)=x^{\deg g}+\sum_{j<\deg g}c_jx^j$
we have that  $\partial_i(c_j)\neq 0$ for some $1\leq i\leq m$ and some $j$, so, using $\dd_i(1)=0$, we get 
$\sum_{j<\deg g}\partial_i(c_j)a^{pj}=0$ which contradicts that $f$ is the minimal polynomial of $a$ over $K$. 
Thus $g^e(a^{p})=0$ iff the coefficients of $g$ are in $C_K$. In conclusion, $\hat e$ exists iff $f\in C_K[x]$, as claimed. Now, to complete the proof, simply note that $f$ having coefficients in $C_K$ implies that $f^e(x)=f(x)$ and thus for any choice of $b\in \DD(L)$ with $\pi(b)=a$ we get
$$f^e(b)=f(b)=g(b^{p})=g(a^{p})=f(a)=0$$
where the third equality uses $\mm=\ker(\operatorname{Fr})$. This guarantees the existence of the desired extension $e'$.
\end{proof}

We will need a more explicit description of the shape of each $\dd_i(a)$ when $a$ is separably algebraic. This makes use of our choice of ranked basis $\bar\epsilon=(1,\epsilon_1,\dots,\epsilon_m)$. Recall that $\alpha_i^{pq}\in k$ denotes the coefficient of $\epsilon_i$ in the product $\epsilon_p\cdot \epsilon_q$. For each $1\leq i\leq m$, we let 
$$\supp^1(i)=\{q: \text{there exists $p$ with $\alpha_i^{pq}\neq 0$}\}$$
and for $n\geq 1$
$$\supp^{n+1}(i)=\bigcup_{q\in \supp^n(i)}\; \supp^1(q)$$
Note that $\supp^1(i)=\emptyset$ when $\sigma(i)=1$. More generally, since $q\in \supp^1(i)$ implies $\sigma(q)<\sigma(i)$, we have that $\supp^{\sigma(i)}(i)=\emptyset$. We define the support of $i$ as
$$\supp(i):=\supp^1(i)\cup\cdots\cup \supp^{\sigma(i)}(i).$$
We use this notion of support in the following two lemmas.

\begin{lemma}\label{explicit}
Let $K\leq L$ be a field extension and $f\in K[x_1,\dots,x_{n}]$. Then, for each $1\leq i\leq m$, there exists a polynomial 
$$h_i\in K((x_{p})_{p\leq n},(y_{p,q})_{p\leq n,q\in\supp(i)},(z_{p})_{p\leq n-1})$$ such that for every
 $\bar a=(a_1,\dots,a_n)\in L^n$ and every $\DD$-operator $e:K(\bar a)\to \DD(L)$ with $e(K)\subseteq \DD(K)$,  if $f(a_1,\dots, a_{n})=0$ then
$$\frac{\partial f}{\partial x_n}(\bar a) \cdot \dd_i(a_{n})=h_i((a_{p})_{p\leq n},(\dd_q(a_p))_{p\leq n,q\in\supp(i)},(\dd_i(a_{p}))_{p\leq n-1}).$$
\end{lemma}
\begin{proof}
Recall that $f^e$ denotes the polynomial (over $\DD(K)$) obtained by applying $e$ to the coefficients of $f$. We then have
$$0=e(0)=e(f(\bar a))=f^{e}(e(a_1),\dots,e(a_n)).$$
On the other hand, using an order-one Taylor expansion at $a_n$, we may write
\begin{align*}
f^e(x_1,\dots,x_{n-1},e(a_n)) & =f^e(x_1,\dots,x_{n-1},a_n) \\
&\quad   +\frac{\partial f^e}{\partial x_n}(x_1,\dots,x_{n-1},a_n)\cdot(\epsilon_1\dd_1(a_n)+\cdots+\epsilon_m\dd_m(a_n)) \\
& \quad +R(x_1,\dots,x_{n-1},a_n)
\end{align*}
where $(\epsilon_1\dd_1(a_n)+\cdots+\epsilon_m\dd_m(a_n))^2$ is a factor of $R$. Putting the above equalities together, we obtain
\begin{align*}
\frac{\partial f^e}{\partial x_n}(e(a_1),\dots,e(a_{n-1}),a_n)\cdot(\epsilon_1\dd_1(a_n)+\cdots+\epsilon_m\dd_m(a_n)) &= -f^e(e(a_1),\dots,e(a_{n-1}),a_n) \\
& \quad - R(e(a_1),\dots,e(a_{n-1}),a_n).
\end{align*}
When computing the coefficient of $\epsilon_i$, in the left-hand-side we find $\frac{\partial f}{\partial x_n}(\bar a) \cdot \dd_i(a_{n})$; while, using the fact that $\epsilon_p\cdot \epsilon_q\in \mm^{\sigma(p)+\sigma(q)}$, we see that the rest of the terms (in this coefficient) form a polynomial $h_i$ of the desired form.
\end{proof}

%What does the next one tells us?

\begin{lemma}\label{partial_com}
Let $K\leq F\leq L$ be field extensions and let $b\in L$ be separably algebraic over $F$. Suppose $e$ and $f: F(b)\to \DD(L)$ are $\DD$-structures with corresponding operators $(\dd_i)_{i=1}^m$ and $(\dd'_i)_{i=1}^m$ such that $e|_K=f|_K$. Let $1\leq i\leq m$ and suppose that $\dd_q=\dd'_q$ for all $q\in \supp(i)$. If  $\dd_{i}|_F=\dd'_{i}|_F$, then $\dd_{i}(b)=\dd'_{i}(b)$.
\end{lemma}
\begin{proof}
	As $b$ is separably algebraic over $F$, there are $a_1,\dots, a_{n-1}\in F$ and a polynomial $f$ over $K$ with $f(a_1,\dots,a_{n-1}, b)=0$ and $\frac{\dd f}{\dd x_n}(a_1,\dots,a_{n-1},b)\neq 0$.

By Lemma \ref{explicit} there is a polynomial $h$ over $K$ such that, writing $a_n:=b$ and $\bar a:=(a_1,\dots,a_n)$, we have 
$$\frac{\dd f}{\dd x_n}(\bar a) \cdot \dd_{i}(b)=h((a_{p})_{p\leq n},(\dd_q(a_p))_{p\leq n,q\in \supp(i)},(\dd_{i}(a_{p}))_{p\leq n-1})$$
and 
$$\frac{\dd f}{\dd x_n}(\bar a)\cdot \dd'_{i}(b)=h((a_{p})_{p\leq n},(\dd'_q(a_p))_{p\leq n,q\in \supp(i)},(\dd'_{i}(a_{p}))_{p\leq n-1}).$$
As, by the assumption, 
$$((\dd_q(a_p))_{p\leq n,q\in \supp(i)},(\dd_{i}(a_{p}))_{p\leq n-1})=((\dd'_q(a_p))_{p\leq n,q\in \supp(i)},(\dd'_{i_0}(a_{p})_{p\leq n-1})),$$ 
we conclude that $\dd_{i}(b)=\dd'_{i}(b)$.
\end{proof}

\iffalse
By Remark \ref{operatornotation} and Lemma \ref{explicit} one also easily obtains:
\begin{corollary}\label{comm_sep_cl}
	Suppose $K$ is a field generated by a subset $A\subseteq K$ and  $b\in L\supseteq K$ is separably algebraic over $K$. Suppose $e: K(a)\to \DD_1(L)$ and  $f: K(a)\to \DD_1(L)$ are homomorphisms with corresponding $\DD$-operators $(D_i)_{i\leq m}$ and $(D'_i)_{i\leq m'}$, $i_0\leq m$ and $i'_0\leq m'$. If
	$D_i(a)D'_{i'}(a)=D'_{i'}(a)D_i(a)$ for all $a\in A$, $i\in \supp(i_0)$, and $i'\in \supp(i'_0)$, then $D_i(b)D'_{i'}(b)=D'_{i'}(b)D_{i}$.
\end{corollary}
\fi

\

%\medskip

In order to introduce our notion of commutativity in the next section, we will need to consider pairs of local operator-systems. Namely, let 
$$\bDD=\{(\DD_1,\bar\epsilon_1),(\DD_2,\bar\epsilon_2)\}$$
where $(\DD_u,\bar\epsilon_u)$ is a local operator-system for $u \in \{1,2\}$. In this case, given a homomorphism of $k$-algebras $R\xrightarrow{\phi} S$, by a $\bDD$-operator from $R$ to $S$ with respect to $\phi$ we mean a pair $\be=(e_1, e_2)$ where each $e_u:R\to \DD_u(S)$ is a $\DD_u$-operator with respect to $\phi$. We denote this by $(R\xrightarrow{\phi} S,\be)$. As before, when $R$ is a subring of $S$ and the inclusion $R\hookrightarrow S$ is a $k$-algebra map, a $\bDD$-operator from $R$ to $S$ with respect to the inclusion is simply called a $\bDD$-operator and we denote it by $\be:R\to \bDD(S)$. If in addition $R=S$, we say that $R$ is a $\bDD$-ring and we denote it by $(R,\be)$. The notions of $\bDD$-homomorphism and $\bDD$-algebra are the obvious ones.

The notation above will be adjusted to the case of pairs of operator-systems by simply adding an index. For instance, $\dim_k \DD_u=m_u+1$ and $\mm_u$ denotes the maximal ideal of $\DD_u$. Similarly, we will denote the operators associated to the ranked bases by $\dd_{u,i}$ where $u \in \{1,2\}$ and $1\leq i\leq m_u$. Then, as in Remark~\ref{operatornotation}, all these operators are additive and satisfy
$$\dd_{u,i}(xy)=\dd_{u,i}(x)\, y +x\, \dd_{u,i}(y)+\sum_{p,q=1}^{m_u}\alpha_{u,i}^{pq}\dd_{u,p}(x)\dd_{u,q}(y)$$
where $\alpha_{u,i}^{pq}\in k$ is the coefficient of $\epsilon_{u,i}$ in the product $\epsilon_{u,p}\cdot \epsilon_{u,q}$ in $\DD_u$. Recall that, due to the choice of ranked basis, $\alpha^{p,q}_{u,i}=0$ whenever $\sigma_u(p)+\sigma_u(q)>\sigma_u(i)$.

\

\section{A notion of commutativity}\label{sec3}

In this section we introduce a notion of commutativity. Let 
$$\bDD=\{(\DD_1,\bar\epsilon_1), (\DD_2,\bar\epsilon_2)\}$$ 
be two local operator-systems (over the field $k$).

\begin{assumption}
We fix a $\bDD$-field $(F,\be=(e_1,e_2))$. For the remainder of this section we assume that all rings under consideration are $F$-algebras and all $\bDD$-rings are $(F,\be)$-algebras. 
\end{assumption}

The reason for this assumption is that it will allow us to recover the case of Lie-commuting derivations as treated by Yaffe in \cite{Yaffe2001}.

\subsection{Commutativity and examples}

We fix (for the remainder of this section) a $k$-algebra homomorphism $r:\DD_2\to \DD_1\otimes_k\DD_2(F)$. 
%Suppose $(R,\be=(e_1,e_2))$ is a $\DD$-ring (and an $(F,\be)$-algebra). 
Let $R\leq S$ be an extension of $F$-algebras and $\be=(e_1,e_2):R\to \bDD(S)$ a $\bDD$-operator.
We note that there are two ways of lifting $r$ to an $R$-algebra homomorphism:
$$r^\iota:\DD_2(R)\to \DD_1(\DD_2(S))$$
and 
$$r^{e_1}:\DD_2(R)\to \DD_1(\DD_2(S))$$
where the lift $r^\iota$ is with respect to the standard $R$-algebra structure on $\DD_2(R)$ and on $\DD_1(\DD_2(S))$, while the lift $r^{e_1}$ is with respect to the standard $R$-algebra structure on $\DD_2(R)$ but with respect to the $e_1$-structure on $\DD_1(\DD_2(S))$; namely, the latter structure is
$$R\xrightarrow[]{e_1} \DD_1(S)  \xrightarrow[]{\DD_1(\iota)}\DD_1(\DD_2(S))$$
where $\iota:S\to\DD_2(S)$ is the standard $S$-algebra structure on $\DD_2(S)=\DD_2\otimes_k S$.
Whenever the lift is with respect to one of these we denote it by $r^*$ (i.e., $*\in\{\iota,e_1\}$). 

%\begin{notation}
%We will denote by $\chi:\DD_1\to\DD_1\otimes_k\DD_2(F)$ the canonical embedding $a\mapsto a\otimes 1$.
%\end{notation}

\bigskip

The following is our notion of commutativity with respect to $r^*$.

\begin{definition}\label{def_comm}
Let $A\leq R\leq S$ be an extension of rings and $\be:R\to \bDD(S)$ a $\bDD$-operator such that $\be(A)\subseteq\bDD(R)$ (i.e. $e_i(A)\subseteq \DD_i(R)$ for $i=1,2$). Also, let $*\in \{\iota,e_1\}$. We say that $(e_1,e_2)$ commute on $(A,R,S)$ with respect to $r^*$ if the following diagram commutes

\begin{equation}\label{commdiagram}
\xymatrix{
 A \ar[rr]^{e_1} \ar[d]^{e_2} && \DD_1(R) \ar[d]^{\DD_1(e_2)} \\
 \DD_2(R) \ar[rr]^{r^*}   && \DD_1(\DD_2(S))
}
\end{equation}
 In case $\DD_1=\DD_2$ and $e_1=e_2$ we simply say that $e_1$ commutes on $(A,R,S)$ (meaning $(e_1,e_1)$ does), and if in addition $A=R=S$ we say $e_1$ commutes on $A$ (with respect to $r^*$).
\end{definition}

\begin{remark}\label{observe}
Let $A\leq R\leq S$ be as in Definition~\ref{def_comm} and assume that $(e_1,e_2)$ commute on $(A,R,S)$ with respect to $r^*$. The following are immediate from the fact that in diagram~\eqref{commdiagram} all maps are $k$-algebra homomorphisms:
\begin{enumerate}
\item Suppose $B\subseteq R$ is such that $\be(B)\subseteq \bDD(R)$. If diagram~\eqref{commdiagram} commutes on $B$, then $(e_1,e_2)$ commute on $(A[B],R,S)$ w.r.t. $r^*$. Here $A[B]$ denotes the ring generated by $B$ over $A$.
\item Suppose $R$ is a field and $\be(\operatorname{Frac}A)\subseteq \bDD(R)$, then $(e_1,e_2)$ commutes on $(\operatorname{Frac}A,R,S)$ w.r.t. $r^*$. 
%Here $A_\Sigma$ denotes the localisation of $A$ by $\Sigma$ (by \cite[Lemma 2.7]{BHKK2019} there is a unique $\bDD$-structure on $A_\Sigma$ extending that on $A$). 
\item Suppose $A$ and $R$ are fields. From (1) and (2) it follows that if $B\subseteq R$ is such that $\be(B)\subseteq \bDD(R)$ and diagram~\eqref{commdiagram} commutes on $B$, then $(e_1,e_2)$ commute on $(A(B),R,S)$ w.r.t. $r^*$. Here $A(B)$ denotes the field generated by $B$ over $A$.
\end{enumerate}
\end{remark}

We now observe that when $r$ is the canonical embedding $\DD_2\hookrightarrow\DD_1\otimes_k\DD_2(F)$ (i.e. $a\mapsto 1\otimes a$) and is lifted via $e_1$, then we recover trivial commutativity of the operators from $\DD_1$ with the ones from $\DD_2$ (i.e. the condition $\dd_{1,i}\dd_{2,j}(a)=\dd_{2,j}\dd_{1,i}$ for all $i$ and $j$).

\begin{lemma}
Let $(R,\be)$ be a $\bDD$-ring and $r$ be the canonical embedding $\DD_2\hookrightarrow\DD_1\otimes_k\DD_2(F)$. Then, $(e_1,e_2)$ commute on $R$ with respect to $r^{e_1}$ if and only if for all $1\leq i\leq m_1$ and $1\leq j\leq m_2$ we have
$$\dd_{1,i}\dd_{2,j}(a)=\dd_{2,j}\dd_{1,i}(a) \quad \text{ for all } a\in R.$$
\end{lemma}
\begin{proof}
Let  $\bar\epsilon_1=(1,\epsilon_{1,1},\dots,\epsilon_{1,m_1})$ and $\bar\epsilon_2=(1,\epsilon_{2,1},
\dots,\epsilon_{2,m_2})$ be the ranked bases of $\DD_1$ and $\DD_2$, respectively. Let $a\in R$. Then, the top and right arrows of diagram \eqref{commdiagram} yield
\begin{align*}
\DD_1(e_2)\circ e_1(a) & =  \DD_1(e_2)\, (a+\epsilon_{1,1}\dd_{1,1}(a)+\cdots+\epsilon_{1,m_1}\dd_{1,m_1}(a)) \\
& = 1\otimes e_2(a) +\epsilon_{1,1}\otimes e_2(\dd_{1,1}(a))+\cdots + \epsilon_{1,m_1}\otimes  e_2(\dd_{1,m_1}(a)) \\
& = 1\otimes 1 \otimes a+1\otimes \epsilon_{2,1}\otimes \dd_{2,1}(a)+\cdots +1\otimes \epsilon_{2,m_2}\otimes \dd_{2,m_2}(a)\; + \\ 
&\hspace{.5cm}  \epsilon_{1,1}\otimes 1\otimes \dd_{1,1}(a) +\epsilon_{1,1}\otimes\epsilon_{2,1}\otimes \dd_{2,1}\dd_{1,1}(a)+\cdots+ \epsilon_{1,1}\otimes \epsilon_{2,m_2}\otimes \dd_{2,m_2}\dd_{1,1}(a) \; + \\
& \hspace{.5cm} \vdots \\
& \hspace{.5cm} \epsilon_{1,m_1}\otimes 1\otimes \dd_{1,m_1}(a)+ \epsilon_{1,m_1}\otimes \epsilon_{2,1}\otimes \dd_{2,1}\dd_{1,m_1}(a)+\cdots + \epsilon_{1,m_1}\otimes \epsilon_{2,m_2}\otimes \dd_{2,m_2}\dd_{1,m_1}(a)
\end{align*}

On the other hand, using that $r$ is the canonical embedding, we get that $r^{e_1}=\DD_2(e_1)$ and thus a similar computation (on the bottom and left arrows of diagram~\eqref{commdiagram}) yields
\begin{align*}
r^{e_1}\circ e_2(a) & = 1\otimes 1 \otimes a+ \epsilon_{1,1}\otimes 1\otimes \dd_{1,1}(a)+\cdots + \epsilon_{1,m_1}\otimes 1\otimes \dd_{1,m_1}(a)\; + \\ 
&\hspace{.5cm}  1\otimes \epsilon_{2,1}\otimes \dd_{2,1}(a) +\epsilon_{1,1}\otimes \epsilon_{2,1}\otimes \dd_{1,1}\dd_{2,1}(a)+\cdots+ \epsilon_{1,m_1}\otimes \epsilon_{2,1}\otimes\dd_{1,m_1}\dd_{2,1}(a) \; + \\
& \hspace{.5cm} \vdots \\
& \hspace{.5cm} 1 \otimes \epsilon_{2,m_2}\otimes \dd_{2,m_2}(a)+\epsilon_{1,1}\otimes \epsilon_{2,m_2}\otimes\dd_{1,1}\dd_{2,m_2}(a)+\cdots +  \epsilon_{1,m_1}\otimes \epsilon_{2,m_2}\otimes\dd_{1,m_1}\dd_{2,m_2}(a)
\end{align*}
As $(\epsilon_{1,i}\otimes \epsilon_{2,j}: 0\leq i\leq m_1,0\leq j\leq m_2)$ is an $R$-linear basis of $\DD_1(\DD_2(R))$, recalling that $\epsilon_{1,0}=1$ and $\epsilon_{2,0}=1$, it follows that 
$$ \DD_1(e_2)\circ e_1(a)=r^{e_1}\circ e_2(a)$$ 
if and only if 
$$\dd_{1,i}\dd_{2,j}(a)=\dd_{2,j}\dd_{1,i}(a)\quad \text{for all } 1\leq i\leq m_1, 1\leq j\leq m_2.$$
\end{proof}

\begin{remark}
Let $\DD=\DD_1=\DD_2$ and $e=e_1=e_2$. Suppose $(R,\be)$ is a $\bDD$-ring and $r$ is the canonical embedding $\DD\hookrightarrow \DD\otimes_k\DD$ (i.e., $x\mapsto 1\otimes x$).
\begin{enumerate}
\item if $r$ lifted by $e$, then $e$ commutes on $R$ if and only if the operators from $\DD$ commute with each other (i.e. $\dd_i\dd_j(a)=\dd_j\dd_i(a)$ for all $a\in R$ and $1\leq i,j\leq m$).
\item if $r$ is lifted by $\iota$, then $e$ commutes on $R$ if and only if the operators from $\DD$ are all trivially zero (i.e. $\dd_i(a)=0$ for all $a\in R$ and $1\leq r\leq m$).
\item if $r':\DD\to \DD\otimes_k\DD$ is the composition of $\pi:\DD\to k$ and $k\hookrightarrow \DD\otimes_k \DD$ (where the latter is the canonical $k$-algebra structure), then $e$ commutes on $R$ (with respect to any lifting of $r'$) if and only if the operators from $\DD$ are all trivially zero.
\end{enumerate}
\end{remark}

We now spell out how to recover the motivating examples.

\begin{example}\label{Liecommexample} (Lie commuting derivations)
Let $m\in \NN$ and 
$$\DD=k[\epsilon_1,\dots,\epsilon_m]/(\epsilon_1,\dots,\epsilon_m)^2$$
with %$\pi(\epsilon_i)=0$ and 
ranked basis $(1,\epsilon_1,\dots,\epsilon_m)$. Recall that this recovers differential rings with $m$-many derivations (see Example \ref{firstexamples}(1)). Let $(F,e)$ be a $\DD$-field; in other words, $F$ is a field extension of $k$ equipped with $k$-linear derivations $\dd_1,\dots,\dd_m$. Let $(c_\ell^{ij})_{i,j,\ell=1}^m$ be a tuple from $F$ such that for each $\ell$ the $m\times m$-matrix $(c_\ell^{ij})_{i,j=1}^m$ is skew-symmetric. Consider the $k$-algebra homomorphism $r:\DD\to \DD(\DD(F))$ determined by
$$r(\epsilon_\ell)=1 \otimes \epsilon_\ell+\sum_{i,j=1}^{m}  \epsilon_i\otimes \epsilon_j\otimes c_{\ell}^{ji}$$
for $\ell=1,\dots,m$. Then, on any $\DD$-ring $(R,e)$, $e$ commutes on $R$ with respect to $r^e$ if and only if
$$[\dd_i,\dd_j]=c_1^{ij}\dd_1+\cdots+c_m^{ij}\dd_m$$
for $1\leq i, j\leq m$.
\end{example}

\begin{example}\label{exampleHS} (iterative truncated H-S derivations in positive characteristic) Assume $char(k)=p>0$. Let $n\in \NN$ and $\DD=k[\epsilon]/(\epsilon)^{p^n}$ with %$\pi(\epsilon)=0$ and 
	ranked basis $(1,\epsilon,\dots,\epsilon^{p^n-1})$. Recall that this recovers rings equipped with a $(p^n-1)$-truncated Hasse-Schmidt derivation $(\dd_i)_{i=1}^{p^n-1}$ (see Example~\ref{firstexamples}(2)). Consider the $k$-algebra homomorphism $r:\DD\to \DD\otimes_k\DD$ determined by 
$$r(\epsilon)=\epsilon\otimes 1+1\otimes \epsilon$$
(the assumption that $char(k)=p$ yields that $r$ is indeed a homomorphism). Then, on any $\DD$-ring $(R,e)$, $e$ commutes on $R$ with respect to $r^{\iota}$ if and only if for $1\leq i,j\leq n$ we have
$$\dd_j\dd_i=
\left\{
\begin{array}{cc}
\binom{i+j}{i}\,\dd_{i+j} &\hspace{.6cm}  i+j\leq p^n -1 \\
0  & i+j\geq p^n
\end{array}
\right.
$$
in other words, $(\dd_i)_{i=1}^{p^n-1}$ is iterative.
\end{example}
%{\bf Again, shall we modify or postpone the two examples below?}

\begin{example}\label{g-der}($\mathfrak{g}$-derivations)
Assume $char(k)=p>0$. Let $\ell,n\in \NN$ and 
$$\DD=k[\epsilon_1,\dots,\epsilon_\ell]/(\epsilon_1^{p^n},\dots,\epsilon_\ell^{p^n}).$$
Let us explain how the notion of a $\mathfrak{g}$-derivation studied by Hoffmann and Kowalski in \cite{HK} falls into our framework. Let $\mathfrak{g}$ be a finite group scheme over $k$ whose underlying scheme is $\spec(\DD)$. Let $r$ be the co-multiplication in the corresponding Hopf algebra. A \emph{$\mathfrak{g}$-derivation} on a $k$ algebra $R$ is a $k$-group scheme action of $\mathfrak g$ on $\spec R$ (see \cite[Definition 3.8]{HK}). 
By \cite[Remark 3.9]{HK}, a $\mathfrak{g}$-derivation on a $k$-algebra $R$ is the same as a $\DD$-operator on $R$ that commutes w.r.t. $r^\iota$ (in the sense of Definition \ref{def_comm}).

%Perhaps we want to add the case of Daniel-Piotr with a formal associative law. Need to add details...
\end{example}

\begin{example}\label{examplesevHS} (several iterative truncated H-S derivations that commute) Assume $char(k)=p>0$. Let $\ell,n\in \NN$ and 
$$\DD=k[\epsilon_1,\dots,\epsilon_\ell]/(\epsilon_1^{p^n},\dots,\epsilon_\ell^{p^n}).$$ 
We recover fields equipped with $\ell$-many iterative $(p^n-1)$-truncated H-S derivations that commute (with each other) by setting $\mathfrak{g}$ to be the truncated $k$-group scheme $\mathbb G_a^\ell[n]$ and considering fields equip with $\mathfrak{g}$-derivations in the sense of Example~\ref{g-der} (see \cite[Example 3.12(1)]{HK}). Alternatively, we can recover this setup as follows: let $\DD_i=k[\epsilon_i]/(\epsilon_i)^{p^{n}}$, for $i=1,\dots,\ell$, and $r_i$ as in Example~\ref{exampleHS}; then taking $\DD$ and $r$ to be the tensor products of the $\DD_i$'s and the $r_i$'s, respectively, yields that $\DD$-operators that commute w.r.t. $r^\iota$ correspond to $\ell$-many commuting iterative $(p^n-1)$-truncated derivations. We provide further details of this latter construction in Appendix~\ref{app_single_algebra}.
\end{example}

\begin{example}\label{allcombined}(Lie commutation and iterativity) Again assume $char(k)=p>0$. Let $m,\ell,n\in \NN$ and 
$$\DD_1=k[\epsilon_1,\dots,\epsilon_m]/(\epsilon_1,\dots,\epsilon_m)^2 \quad \text{ and } \quad \DD_2=k[\epsilon_1,\dots,\epsilon_\ell]/(\epsilon_1^{p^n},\dots,\epsilon_\ell^{p^n}).$$
Let $(F,e)$ be a $\DD_1$-field and $r_1:\DD_1\to \DD_1(\DD_1(F))$ a $k$-algebra homomorphism as in Example~\ref{Liecommexample}. Also  let $\mathfrak g$ be a finite $k$-group scheme with underlying scheme $\spec(\DD_2)$ and $r_2$ be co-multiplication in the Hopf algebra dual to $\mathfrak g$ (as in Example~\ref{g-der}). Furthermore, let $r_{12}:\DD_2\to \DD_1\otimes \DD_2$ be the $k$-algebra homomorphism $r(x)=1\otimes x$.

Then, in any $\bDD$-ring $(R,\be=(e_1,e_2))$, $e_1$ commutes w.r.t. $r_1^{e_1}$, $e_2$ commutes w.r.t. $r_2^\iota$, and $(e_1,e_2)$ commutes w.r.t. $r_{12}^{e_1}$ if and only if
\begin{itemize}
\item the operators $(\dd_{1,i})$ associated to $e_1$ are derivations that Lie-commute,
\item the operators $(\dd_{2,j})$ associated to $e_2$ are $\mathfrak g$-derivations, and 
\item each pair of operators $\dd_{1,i}$ and $\dd_{2,j}$ commute.
\end{itemize}
\end{example}

%%%%%%%%%%%commentoutbegin
\begin{confidential}

\begin{example}\label{examplesevHS} (iterative truncated H-S derivations that commute) Assume $char(k)=p>0$. Let $n_1,\dots,n_s$ be positive integers and $\DD_u=k[\epsilon]/(\epsilon)^{p^{n_u}}$, for $u=1,\dots,s$, with the natural choice of %$\pi_u$ and 
	ranked bases. In this case $\bDD$-rings correspond to rings equipped with $(\dd_{1,i})_{i=1}^{p^{n_1}-1},\dots,(\dd_{s,i})_{i=1}^{p^{n_s}-1}$ where each $(\dd_{u,i})$ is a $(p^{n_u}-1)$-truncated Hasse-Schmidt derivation. For each $1\leq i\leq j \leq s$, let $r_{i,j}:\DD_j\to \DD_j\otimes \DD_i$ be the $k$-algebra homomorphism defined by
$$
r_{i,j}(\epsilon)=
\left\{
\begin{array}{cc}
\epsilon\otimes 1 +1\otimes \epsilon & i=j \\
\epsilon\otimes 1 &    i\neq j
\end{array}
\right.
$$
Then, in any $\bDD$-ring $(R,\be=(e_1,\dots,e_s))$, $e_i$ commutes on $R$ with respect to $r_{i,i}^{\iota}$ and, for $i<j$, $(e_i,e_j)$ commute on $R$ with respect to $r_{i,j}^{e_i}$ if and only if the operators $(\dd_{u,i})$ are \emph{iterative} truncated Hasse-Schmidt derivations that \emph{commute} with each other.
\end{example}

\begin{example}\label{allcombined}(Lie commutation and iterativity) Again assume $char(k)=p>0$. As in Example~\ref{firstexamples}(4), we can consider $\bDD=\{\DD_0,\DD_1,\dots,\DD_s\}$ such that $\bDD$-rings correspond to rings equipped with 
$$(\dd_{0,i})_{i=1}^{m},(\dd_{1,i})_{i=1}^{p^{n_1}-1},\dots,(\dd_{s,i})_{i=1}^{p^{n_s}-1}$$
where $(\dd_{0,i})_{i=1}^{m}$ are $m$-many derivations and, for $u=1,\dots,s$, each $(\dd_{u,i})_{i=1}^{p^{n_u}-1}$ is an $(p^{n_u}-1)$-truncated Hasse-Schmidt derivation. Combining the previous examples one can find $k$-algebra homomorphisms $(r_{i,j}:0\leq i\leq j\leq s)$ such that our notion of commutativity corresponds precisely to:
\begin{itemize}
\item the derivations $(\dd_{0,1},\dots,\dd_{0,m})$ Lie-commute,
\item each truncated Hasse-Schmidt derivation $(\dd_{u,1},\dots,\dd_{u,p^{n_u}-1})$ is iterative, and
\item for $0\leq u<v\leq s$, the operators $\dd_{u,i}$ and $\dd_{v,j}$ commute.
\end{itemize}
\end{example}

\end{confidential}
 %%%%%%%%commentoutend

 \

\subsection{Commutativity in separable extensions and compositums} 

We carry on the notation from the previous subsection. In particular, $\bDD=\{\DD_1,\DD_2\}$ are two local operator-systems and $(F,\be=(e_1,e_2))$ is a fixed $\bDD$-field. All rings under consideration are $F$-algebras and $\bDD$-rings are $(F,\be)$-algebras. In this subsection we prove that our notion of commutativity is preserved when passing to separably algebraic field extensions and also on compositums.

\medskip

Recall that $\mm_i$ denotes the maximal ideal of $\DD_i$, for $i=1,2$. Due to a theorem of Sweedler \cite{Swee1975}, the algebra $\DD_{i,j}:=\DD_i\otimes_k\DD_j$ is local with maximal ideal
$$\tilde \mm:=\mm_i\otimes_k \DD_j+\DD_i\otimes_k \mm_j.$$
Hence we can view $\DD_{i,j}$ as a local operator-system (one can also attach a ranked basis of course, but in this section such basis will not be used). So now we may talk about $\DD_{i,j}$-structures. Note that the residue map is given by 
$$\pi_{i,j}:=\pi_i\otimes\pi_j:\DD_{i,j}\to k.$$

Let $A\leq R\leq S$ be an extension of rings and $\be:R\to \bDD(S)$ be a $\bDD$-operator such that $\be(A)\subseteq \bDD(R)$. Let $r:\DD_i\to \DD_{j,i}(F)$ be a $k$-algebra homomorphism and $*\in\{\iota,e_j\}$. Then, the map
$$A\xrightarrow{e_i}\DD_i(R)\xrightarrow{r^{*}} \DD_{j,i}(S)$$
yields a $\DD_{j,i}$-operator $A\to\DD_{j,i}(S)$. Indeed, $r^{*}\circ e_i$ is clearly a $k$-algebra homomorphism and a straightforward computation yields $\pi_{i,j}^R\circ r^{*}\circ e_i=\operatorname{Id}_A$ (using the fact that $r(\mm_i)\subseteq \tilde \mm\otimes_k F$ and $\pi_{i,j}(\tilde \mm)=0$).

\medskip

For the following proof we use the above observations when $i=2$ and $j=1$. In particular, $r:\DD_2\to \DD_{1,2}(F)$ and $*\in \{\iota,e_1\}$.

\begin{theorem}\label{commutingext}
Suppose $K\leq L\leq E$ is an extension of fields and $\be:L\to \bDD(E)$ is a $\bDD$-operator such that $\be(K)\subseteq \bDD(L)$. Assume $(e_1,e_2)$ commute on $(K,L,E)$ with respect to $r^*$. Then, for any $K\leq K'\leq L$ with $K'/K$ separably algebraic, we have that $\be(K')\subseteq \bDD(L)$ and $(e_1,e_2)$ commute on $(K',L,E)$ with respect to $r^*$.
\end{theorem}
\begin{proof}
By Lemma~\ref{extendstructure}(i), applied to each $\DD_i$, we get $\be(K')\subseteq \bDD(L)$.  We now show that commutativity is preserved. By the preceding observations, the homomorphisms
\begin{equation}\label{keytoextend}
\DD_1(e_2)\circ e_1\quad \text{ and }\quad r^{*}\circ e_2\quad \text{ from }K'\to \DD_{2,1}(E)
\end{equation}
are both $\DD_{1,2}$-operators. Moreover, the former extends the $\DD_{1,2}$-operator
$$ \DD_1(e_2)\circ e_1:K\to \DD_{1,2}(E)$$
while the latter extends
$$ r^{*}\circ e_2:K\to \DD_{1,2}(E).$$
Since $(e_1,e_2)$ commute on $(K,L,E)$ with respect to $r^*$, the previous two $\DD_{1,2}$-operators $K\to\DD_{1,2}(E)$ coincide. Now, as $K'/K$ is separably algebraic, Lemma~\ref{extendstructure}(i) yields that $\DD_{1,2}$-structures extend uniquely from $K$ to $K'$, and thus the two $\DD_{1,2}$-structures on $K'\to \DD_{1,2}$ (displayed in \eqref{keytoextend} above) must coincide. In other words, $(e_1,e_2)$ commute on $(K',L,E)$ with respect to $r^*$.
\end{proof}

As a consequence, we obtain that $\bDD$-operators extend to separable closures in a commuting manner. 

\begin{corollary}\label{commuting_sep}
Let $(K,\be)$ be a $\bDD$-field. If $(e_1,e_2)$ commute on $K$ with respect to $r^*$, then the unique $\bDD$-extension to $(K^{\operatorname{sep}},\be)$ (given by Lemma~\ref{extendstructure}(i)) also commutes with respect to $r^*$.
\end{corollary}
\begin{proof}
Apply Theorem~\ref{commutingext} with $E=L=K'=K^{\operatorname{sep}}$.
\end{proof}

We conclude with the observation that $\bDD$-operators also extend to compositums in a commuting manner.

\begin{lemma}\label{commcomp}
Let $(E,\be)$ be a $\bDD$-field and let $L_1$ and $L_2$ be $\bDD$-subfields. If $(e_1,e_2)$ commute on $L_i$ with respect to $r^*$ for $i=1,2$, then they also commute on $L_1\cdot L_2$ with respect to $r^*$.
\end{lemma}
\begin{proof}
First note that, since $e_1$ and $e_2$ are ring homomorphisms, the compositum $L_1\cdot L_2$ is a $\bDD$-subfield of $E$. Furthermore, all the maps involved in diagram~\eqref{commdiagram} are also ring homomorphisms; thus, commutativity of the diagram on $L_1\cdot L_2$ follows from commutativity in each $L_i$. Alternatively, one can simply invoke Remark~\ref{observe}(3).
\end{proof}

\

\section{Lie-Hasse-Schmidt commutativity}\label{sec4}

In this section we introduce the notion of Lie-Hasse-Schmidt commutation systems (or LHS-systems for short). We also discuss when these systems satisfy a property that we call Jacobi-associativity which will be used to prove the existence of principal realisations of $\gDD$-kernels (in particular, the existence of $\gDD$-polynomial rings) in Section \ref{kernels}.  We achieve this by building up to the general case in the next three subsections.

\subsection{Lie commutativity}\label{Lietype}

Throughout this section we let $(\DD,\bar\epsilon)$ be a local operator-system (with $\dim_k(\DD)=m+1$) and $(F,e)$ a $\DD$-field. We carry on our assumption that all rings are $F$-algebras and $\DD$-rings are $\DD$-algebras over $(F,e)$.

\medskip 

To define our notion of Lie commutation system, we will make use of the following terminology: recall that $\mm$ denotes the maximal ideal of $\DD$, the null of $\DD$ is defined as
$$\Null(\DD)=\{1\leq q\leq m: \epsilon_q\cdot \mm=(0)\}.$$
Note $\Null(\DD)$ is not empty as $q\in \Null(\DD)$ for those $q$'s with $\sigma(q)=d$. Recalling that $\alpha_{i}^{pq}$ denotes the coefficient of $\epsilon_i$ in the product $\epsilon_p\cdot\epsilon_q$, one easily checks that $q\in \Null(\DD)$ if and only if $\alpha_{i}^{pq}=0$ for all $i$ and $p$.

%\medskip

%Denote by 
%$$\operatorname{res}_1:\DD\otimes_k \DD(F)\to \DD(F)$$ 
%the residue map by the ideal $\DD\otimes_k \mm(F)$ of $\DD\otimes_k \DD(F)$. 

\begin{definition}
Let $r:\DD\to \DD\otimes _k\DD(F)$ be a $k$-algebra homomorphism. We say that $r$ is of Lie-commutation type if there exists a tuple  $(c_\ell^{ij})_{i,j,\ell=1}^m$ from $F$ such that
$$r(\epsilon_\ell)=1 \otimes \epsilon_\ell+ \sum_{i,j=1}^m \epsilon_i\otimes \epsilon_j\otimes c_\ell^{ji} $$
and  
$$c_{\ell}^{ji}=0 \quad \text{ unless }\quad i,j\in \Null(\DD).$$
%with $\gamma(d)$ as in Remark~\ref{operatornotation}. 
We call the tuple $(c_\ell^{ji})$ the Lie-coefficients of $r$ (with respect to the ranked basis $\bar \epsilon=(1,\epsilon_1,\dots,\epsilon_m)$). %Note that $r(\epsilon_i)= 1 \otimes \epsilon_i$ whenever $\sigma(i)\geq 2$.
\end{definition}

\begin{lemma}\label{Lie_sum}
Let $r$ be of Lie-commutation type. Then, on any $\DD$-ring $(R,e)$, $e$ commutes on $R$ with respect to $r^e$ if and only if
$$[\dd_i,\dd_j]= \sum_{\ell} c_{\ell}^{ij}\dd_\ell $$ %= c_1^{ij}\dd_1+\cdots +c_m^{ij}\dd_m$$
where $(c_\ell^{ij})$ is the tuple of Lie coefficients of $r$. In particular, 
$$[\dd_i,\dd_j]=0 \quad \text{ whenever one of $i,j$ is not in $\Null(\DD)$}.$$
\end{lemma}
\begin{proof}
This is a straightforward computation.
\end{proof}

From now on, if $(R,e)$ is a $\DD$-ring on which $e$ commutes with respect to $r^e$, we will simply say that $(R,e)$ is a $\DD^r$-ring.

\begin{remark}\label{skewind}
Let $r$ be of Lie-commutation type with Lie coefficients $(c_\ell^{ij})$. We note that if there exists a $\DD^r$-ring $(R,e)$ in which the operators $(\dd_1,\dots,\dd_m)$ are $F$-linearly independent (as functions from $R\to R$), then for each $\ell$ the $m\times m$-matrix $(c_\ell^{ij})_{i,j=1}^m$ is skew-symmetric. This follows from the previous lemma noting that $[\dd_i,\dd_j]+[\dd_j,\dd_i]=0$.
\end{remark}

We now introduce another notion that will be used in our discussion of realisations of $\gDD$-kernels in the next section.

\smallskip

%In the following lemma we characterise being Jacobi in terms of the Lie coefficients.

\begin{definition}\label{Liecoe}
	Let $r$ be of Lie-commutation type and let $(c_\ell^{ij})$ be its Lie-coefficients. We say that $r$ is Jacobi if
	\begin{enumerate}
		\item for each $\ell$, the $m\times m$ matrix $(c_{\ell}^{ij})_{i,j=1}^m$ is skew-symmetric,
		\item for each $1\leq i,j,k,r\leq m$
		$$\sum_{\ell=1}^m \left( c_\ell^{ij}c_r^{\ell k}+ c_{\ell}^{ki}c_r^{\ell j}  +c_\ell^{jk}c_r^{\ell i} \right)=\dd_i(c_r^{jk})+\dd_k(c_r^{ij})+\dd_j(c_r^{ki})$$
		(this is a form of the Jacobi identity),
		\item for each $1\leq i,j,k,r\leq m$
		$$\sum_{p=1}^m \left(\alpha_i^{pr}\dd_p(c_r^{jk})+\alpha_k^{pr} \dd_p(c_r^{ij}) +\alpha_j^{pr}\dd_p(c_r^{ki}) \right)=0.$$
		and in addition for $1\leq q<r\leq m$
		$$\sum_{p=1}^m \left(\alpha_i^{pq}\dd_p(c_r^{jk})+\alpha_k^{pq} \dd_p(c_r^{ij}) +\alpha_j^{pq}\dd_p(c_r^{ki}) + \alpha_i^{pr}\dd_p(c_q^{jk})+\alpha_k^{pr} \dd_p(c_q^{ij}) +\alpha_j^{pr}\dd_p(c_q^{ki}) \right)=0.$$
		where recall that $\alpha_i^{pq}\in k$ is the coefficient of $\epsilon_i$ in the product $\epsilon_p\cdot \epsilon_q$.
	\end{enumerate}
\end{definition}

\begin{remark}
We note that when the Lie coefficients $(c_\ell^{ij})$ are all zero, then $r$ is the canonical embedding $\DD\to \DD(\DD(F))$ (i.e., $x\to 1\otimes x$) and clearly this $r$ is Jacobi.
\end{remark}

The following explains why the Jacobi property is a natural condition. Recall that, for $r$ of Lie-commutation type, by a $\DD^r$-ring $(R,e)$ we mean a $\DD$-ring on which $e$ commutes with respect to $r^e$.

\begin{proposition}\label{just1}
Let $r:\DD\to\DD\otimes_k \DD(F)$ be of Lie-commutation type. If there exists a $\DD^r$-ring $(R,e)$ in which the operators
$$(\dd_1,\dots,\dd_m)\quad \text{ and } \quad (\dd_i\dd_j: 1\leq i\leq j\leq m)$$
are $F$-linearly independent (as functions $R\to R$), then $r$ is Jacobi.
\end{proposition}
\begin{proof}
It follows from Remark~\ref{skewind} that for each $\ell$ the matrix $(c_\ell^{ij})_{i,j=1}^m$ is skew-symmetric. Thus, it only remains to check (2) and (3) of Definition ~\ref{Liecoe}.

\

On the one hand, by Lemma~\ref{Lie_sum}, we have
$$\dd_i\dd_j\dd_k=\dd_j\dd_i\dd_k+\sum_\ell c_\ell^{ij}\dd_\ell\dd_k$$
On the other hand
\begin{align*}
\dd_i\dd_j\dd_k &=\dd_i\dd_k\dd_j+\sum_\ell \dd_i(c_\ell^{jk} \dd_\ell) \\
& = \dd_k\dd_i\dd_j +\sum_\ell c_\ell^{ik}\dd_\ell\dd_j+\sum_\ell \dd_i(c_\ell^{jk} \dd_\ell) \\
& = \dd_k\dd_j\dd_i +\sum_\ell \dd_k(c_\ell^{ij}\dd_\ell)+ \sum_\ell c_\ell^{ik}\dd_\ell\dd_j+\sum_\ell \dd_i(c_\ell^{jk} \dd_\ell) \\
& = \dd_j\dd_k\dd_i +\sum_\ell c_\ell^{kj}\dd_\ell\dd_i + \sum_\ell \dd_k(c_\ell^{ij}\dd_\ell)+ \sum_\ell c_\ell^{ik}\dd_\ell\dd_j+\sum_\ell \dd_i(c_\ell^{jk} \dd_\ell) \\
& = \dd_j\dd_i\dd_k+ \sum_\ell \dd_j(c_\ell^{ki}\dd_\ell)+\sum_\ell c_\ell^{kj}\dd_\ell\dd_i + \sum_\ell \dd_k(c_\ell^{ij}\dd_\ell)+ \sum_\ell c_\ell^{ik}\dd_\ell\dd_j+\sum_\ell \dd_i(c_\ell^{jk} \dd_\ell)
\end{align*}

So 
\[\dd_j\dd_i\dd_k+\sum_\ell c_\ell^{ij}\dd_\ell\dd_k =\dd_j\dd_i\dd_k+ \sum_\ell \dd_j(c_\ell^{ki}\dd_\ell)+\sum_\ell c_\ell^{kj}\dd_\ell\dd_i + \sum_\ell \dd_k(c_\ell^{ij}\dd_\ell)+ \sum_\ell c_\ell^{ik}\dd_\ell\dd_j+\sum_\ell \dd_i(c_\ell^{jk} \dd_\ell)\] 
and hence, using Remark~\ref{skewind} and Lemma~\ref{Lie_sum} in the third equality below, we get

\[0=\sum_\ell c_\ell^{kj}\dd_\ell\dd_i+\sum_\ell \dd_i(c_\ell^{jk} \dd_\ell) + \sum_\ell c_\ell^{ji}\dd_\ell\dd_k+ \sum_\ell \dd_k(c_\ell^{ij}\dd_\ell)+ \sum_\ell c_\ell^{ik}\dd_\ell\dd_j+ \sum_\ell \dd_j(c_\ell^{ki}\dd_\ell)\]
\[=\sum_\ell c_\ell^{kj}\dd_\ell\dd_i+\sum_\ell c_\ell^{jk}\dd_i\dd_\ell+\sum_\ell \dd_i(c_\ell^{jk})\dd_\ell+\sum_{\ell,p,q}\alpha_i^{pq}\dd_p(c_\ell^{jk})\dd_q\dd_\ell+\]
\[+\sum_\ell c_\ell^{ji}\dd_\ell\dd_k+\sum_\ell c_\ell^{ij}\dd_k\dd_\ell+\sum_\ell \dd_k(c_\ell^{ij})\dd_\ell+\sum_{\ell,p,q}\alpha_k^{pq}\dd_p(c_\ell^{ij})\dd_q\dd_\ell+ \]
\[+\sum_\ell c_\ell^{ik}\dd_\ell\dd_j+\sum_\ell c_\ell^{ki}\dd_j\dd_\ell+\sum_\ell \dd_j(c_\ell^{ki})\dd_\ell+\sum_{\ell,p,q}\alpha_j^{pq}\dd_p(c_\ell^{ki})\dd_q\dd_\ell \]
\[=\sum_\ell c_\ell^{jk}\sum_{p} c^{i\ell}_p\dd_p+\sum_\ell \dd_i(c_\ell^{jk})\dd_\ell+\sum_{l,p,q}\alpha_i^{pq}\dd_p(c_\ell^{jk})\dd_q\dd_\ell+\]
\[+\sum_\ell c_\ell^{ij}\sum_{p} c^{k\ell}_p\dd_p+\sum_\ell \dd_k(c_\ell^{ij})\dd_\ell+\sum_{l,p,q}\alpha_k^{pq}\dd_p(c_\ell^{ij})\dd_q\dd_\ell+\]
\[+\sum_\ell c_\ell^{ki}\sum_{p} c^{j\ell}_p\dd_p+\sum_\ell \dd_j(c_\ell^{ki})\dd_\ell+\sum_{l,p,q}\alpha_j^{pq}\dd_p(c_\ell^{ki})\dd_q\dd_\ell \]
\[= \sum_p(
\sum_{\ell} c_\ell^{jk}c_p^{i\ell}+\dd_i(c^{jk}_p)+
\sum_{\ell} c_\ell^{ij}c_p^{k\ell}+\dd_k(c^{ij}_p)+
\sum_{\ell} c_\ell^{ki}c_p^{j\ell}+\dd_j(c^{ij}_p))\dd_p+\]
\[+ \sum_{\ell,p,q}(
\alpha_i^{pq}\dd_p(c_\ell^{jk})+
\alpha_j^{pq}\dd_p(c_\ell^{ki})+
\alpha_k^{pq}\dd_p(c_\ell^{ij})
)\dd_q\dd_\ell\]
%{Now the coefficients by $\dd_p$ and $\dd_q\dd_\ell$ are exactly what we want to be zero, but not all $\dd_p$, $\dd_q\dd_\ell$ are linearly independent. If we want to eliminate the remaining $\dd_q\dd_\ell$ with $q>\ell$ we get:}
\[ =\sum_p(
\sum_{\ell} c_\ell^{jk}c_p^{i\ell}+\dd_i(c^{jk}_p)+
\sum_{\ell} c_\ell^{ij}c_p^{k\ell}+\dd_k(c^{ij}_p)+
\sum_{\ell} c_\ell^{ki}c_p^{j\ell}+\dd_j(c^{ij}_p))\dd_p+\]
\[+ \sum_{p}\sum_{q\leq \ell}(
\alpha_i^{pq}\dd_p(c_\ell^{jk})+
\alpha_j^{pq}\dd_p(c_\ell^{ki})+
\alpha_k^{pq}\dd_p(c_\ell^{ij}))\dd_q\dd_\ell+\] 
\[ +\sum_{p}\sum_{\ell>q}(
\alpha_i^{pq}\dd_p(c_\ell^{jk})+
\alpha_j^{pq}\dd_p(c_\ell^{ki})+
\alpha_k^{pq}\dd_p(c_\ell^{ij}))(\dd_\ell\dd_q+\sum_{s}c_s^{q\ell}\partial_s)
)\]
\[=
\sum_p(\sum_{\ell} c_\ell^{jk}c_p^{i\ell}+\dd_i(c^{jk}_p)+
\sum_{\ell} c_\ell^{ij}c_p^{k\ell}+\dd_k(c^{ij}_p)+
\sum_{\ell} c_\ell^{ki}c_p^{j\ell}+\dd_j(c^{ij}_p))\dd_p+\]
\[+ \sum_{p}\sum_{q< \ell}(
\alpha_i^{pq}\dd_p(c_\ell^{jk})+
\alpha_j^{pq}\dd_p(c_\ell^{ki})+
\alpha_k^{pq}\dd_p(c_\ell^{ij})+
\alpha_i^{pl}\dd_p(c_q^{jk})+
\alpha_j^{pl}\dd_p(c_q^{ki})+
\alpha_k^{pl}\dd_p(c_q^{ij})
)\dd_q\dd_\ell+\]
\[+\sum_{p}\sum_{\ell}(
\alpha_i^{p\ell}\dd_p(c_\ell^{jk})+
\alpha_j^{p\ell}\dd_p(c_\ell^{ki})+
\alpha_k^{p\ell}\dd_p(c_\ell^{ij}))\dd_\ell\dd_\ell+\]
\[+  \sum_{p}\sum_{\ell>q}(
\alpha_i^{pq}\dd_p(c_\ell^{jk})+
\alpha_j^{pq}\dd_p(c_\ell^{ki})+
\alpha_k^{pq}\dd_p(c_\ell^{ij}))(\sum_{s}c_s^{q\ell}\partial_s)
)
\]

Now note that 
\[ \sum_{p}\sum_{\ell>q}(
\alpha_i^{pq}\dd_p(c_\ell^{jk})+
\alpha_j^{pq}\dd_p(c_\ell^{ki})+
\alpha_k^{pq}\dd_p(c_\ell^{ij}))(\sum_{s}c_s^{q\ell}\partial_s)=0\]
as
%, by Remark \ref{zero_coefficients}, 
$c_s^{q\ell}=0$ unless $q\in \Null(\DD)$ in which case $\alpha_i^{pq}=\alpha_j^{pq}=\alpha_k^{pq}=0$.
So we get  \[
\sum_p(\sum_{\ell} c_\ell^{jk}c_p^{i\ell}+\dd_i(c^{jk}_p)+
\sum_{\ell} c_\ell^{ij}c_p^{k\ell}+\dd_k(c^{ij}_p)+
\sum_{\ell} c_\ell^{ki}c_p^{j\ell}+\dd_j(c^{ij}_p))\dd_p+\]
\[+ \sum_{p}\sum_{q< \ell}(
\alpha_i^{pq}\dd_p(c_\ell^{jk})+
\alpha_j^{pq}\dd_p(c_\ell^{ki})+
\alpha_k^{pq}\dd_p(c_\ell^{ij})+
\alpha_i^{pl}\dd_p(c_q^{jk})+
\alpha_j^{pl}\dd_p(c_q^{ki})+
\alpha_k^{pl}\dd_p(c_q^{ij})
)\dd_q\dd_\ell+\]
\[+\sum_{p}\sum_{\ell}(
\alpha_i^{p\ell}\dd_p(c_\ell^{jk})+
\alpha_j^{p\ell}\dd_p(c_\ell^{ki})+
\alpha_k^{p\ell}\dd_p(c_\ell^{ij}))\dd_\ell\dd_\ell=0\]

As all $\dd_p$ and $\dd_q\dd_\ell$ are linearly independent, the coefficients by $\partial_p$ must be $0$, giving us item (2) of Definition \ref{Liecoe},
and the coefficients by $\dd_q\dd_\ell$ and by $\dd_\ell\dd_\ell$ must be $0$ as well, giving us  item (3) of Definition \ref{Liecoe}.
%\[\sum_{p}(\alpha_i^{pq}\dd_p(c_\ell^{jk})+\alpha_j^{pq}\dd_p(c_\ell^{ki})+\alpha_k^{pq}\dd_p(c_\ell^{ij})+\alpha_i^{pl}\dd_p(c_q^{jk})+\alpha_j^{pl}\dd_p(c_q^{ki})+\alpha_k^{pl}\dd_p(c_q^{ij}))=0\]
%{which is not quite what we want in (3) in Lemma \ref{Liecoe} (in particular we get seemingly less equations than there, as we the one we got are indexed with only $q<l$ and $q=l$ but not $q>l$)}

\end{proof}

As we will see in Corollary~\ref{functionfield} below, the converse of this proposition also holds. Namely, the Jacobi property guarantees the existence of a $\DD^r$-ring where the operators and their compositions (in a fixed order) are linearly independent.

\begin{remark}
Suppose we are in the case $\DD={\mathbb Q}[\epsilon_1,\dots,\epsilon_m]/(\epsilon_1,\dots,\epsilon_m)^2$ with %$\pi(\epsilon_i)=0$ and 
ranked basis $(1,\epsilon_1,\dots,\epsilon_m)$. In this instance, $F$ is a field of characteristic zero equipped with derivations $\dd_1,\dots,\dd_m$. Let $r$ be of Lie-commutation type with Lie coefficients $(c_\ell^{ij})$, see Example~\ref{Liecommexample}. We note that in this case condition (3) of Definition~\ref{Liecoe} is trivially satisfied (as the $\alpha_{i}^{pq}$ are zero). Furthermore,
\begin{enumerate}
\item [(i)] conditions (1) and (2) of Definition~\ref{Liecoe} are equivalent to the existence of a Lie action (by derivations) on $F$ of an $m$-dimensional $F$-vector space $\mathfrak L$ equipped with a Lie-ring structure (where a fixed basis of $\mathfrak L$ maps to $(\dd_1,\dots,\dd_m)$). In other words, in the vector space $F^m$ with standard basis $b_1,\dots,b_m$, if we set 
$$[\alpha b_i,\beta b_j]=\alpha\beta\,(c_1^{ij}b_1+\cdots +c_m^{ij}b_m)+(\alpha\,\dd_i(\beta)b_j-\beta\,\dd_j(\alpha)b_i)$$
where $\alpha,\beta\in F$ and extend bi-additively, then this product yields a Lie-ring structure on $F^m$ (i.e. skew-symmetry and Jacobi identity) if and only if the Lie-coefficients satisfy conditions (1) and (2) of Definition~\ref{Liecoe}. Thus, in this instance, when $r$ is Jacobi we precisely recover the theory of Lie differential fields (in characteristic zero) as treated by Yaffe in \cite{Yaffe2001};

\item[(ii)] in \cite{Hub2005}, Hubert constructed the Lie-commuting analogue of the classical differential polynomial ring. It is noted there that such a construction requires conditions (1) and (2) of Definition~\ref{Liecoe}. In fact, these conditions appear explicitly in \S5.2 of \cite{Hub2005} (top of p.180).
\end{enumerate}
\end{remark}

\

\subsection{Hasse-Schmidt iterativity}\label{HStype} As in the previous section, $\DD$ denotes a local operator-system and we carry over our assumptions about the $\DD$-field $(F,e)$. 

%Denote by $\operatorname{res}_1$ and $\operatorname{res}_2$ the quotient maps $\DD\otimes_k \DD(F)\to \DD(F)$ by the ideals $$\DD\otimes_k \mm(F) \quad \text{ and } \quad \mm\otimes_k\DD(F)$$ of $\DD\otimes_k \DD(F)$, respectively. 

\begin{definition}
Let $r:\DD\to \DD\otimes _k\DD(F)$ be a $k$-algebra homomorphism. We say that $r$ is of Hasse-Schmidt-iteration type (or just HS-iteration type) if there exists a tuple $(c_\ell^{ij})_{i,j,\ell=1}^m$ from $F$ such that
$$r(\epsilon_\ell)=\epsilon_\ell\otimes 1+ 1\otimes \epsilon_\ell+ \sum_{i,j=1}^m \epsilon_i\otimes \epsilon_j\otimes c_\ell^{ij} .$$
We call the tuple $(c_\ell^{ij})$ the HS-coefficients of $r$ (with respect to the ranked basis $(1,\epsilon_1,\dots,\epsilon_m)$ of $\DD$).
\end{definition}

\interfootnotelinepenalty=10000
\begin{remark}\label{impliespositive} If $r$ is of HS-iteration type, then $char(k)>0$. Indeed, towards a contradiction, suppose $char(k)=0$. Then there is $\epsilon\in \mm$ such that $\epsilon^d\neq 0$ \footnote{This must be folklore, but we include an argument for completeness. Suppose there is no such $\epsilon$; we show that then $\mm^d=0$, which will be a contradiction (since $d$ was chosen smallest such that $\mm^{d+1}=0$). Let $m_0,\dots,m_{d-1}\in \mm$. Let $(\bar x^{\bar i_0}, \dots, \bar x^{\bar i_{\ell-1}})$ be the list of all monomials of degree $d$ in variables $\bar x=(x_0,\dots,x_{d -1})$.  
Consider new variables $\bar{ \bar x}=(\bar x_j)_{j<\ell}$ where each $\bar x_j=(x_{j,i})_{i<d}$ is a $d$-tuple.
%$\bar{ \bar x}=(x_{j,i})_{j<\ell, i<d}=(\bar x_j)_{j<\ell}$. 
Note that $P(\bar {\bar x}):=\det (\bar x_j^{\bar i_k})_{j,k<\ell}$ is a non-zero polynomial in $\mathbb Q[\bar{\bar x}]$, as the monomials $\prod_{j<\ell} \bar x_j^{\bar i_{\sigma(j)}}$ with $\sigma\in S_\ell$ are pairwise distinct. Thus there is a tuple $\bar{\bar q}=(\bar q_0,\dots,\bar q_{\ell -1})\in \mathbb{Q}^{\ell d}$ with $P(\bar {\bar q})\neq 0$ so the matrix $A:=(\bar q_j^{\bar i_k})_{j,k<\ell }$ is invertible. By assumption, for every $j$ we have $0=(q_{j,0}m_0+\dots +q_{j,d-1}m_{d-1})^d=\sum_{k<\ell} \bar q_j^{\bar i_k} {d \choose \bar i_k}\bar m^{\bar i_k}$ where $\bar m=(m_0,\dots,m_{d-1})$. Thus $A\cdot ( {d \choose \bar i_k} \bar m^{\bar i_k})_{k<\ell}=0$ so $( {d \choose \bar i_k}\bar m^{\bar i_k})_{k<\ell}=0$ as $A$ is invertible. In particular, $0= {d \choose 1,\dots,1}m_0m_1\dots m_{d-1}=d!m_0m_1\dots m_{d-1}$ so $m_0m_1\dots m_{d-1}=0$.}.
In particular, $\epsilon \notin \mm^2$. As $r$ is a homomorphism, 
\begin{align*}
0 & =(r(\epsilon))^{2d} \\
& =(\epsilon\otimes 1+ 1\otimes \epsilon+ \sum_{i,j=1}^m \epsilon_i\otimes \epsilon_j\otimes c_\ell^{ij})^{2d} \\
& =(\epsilon\otimes 1+ 1\otimes \epsilon)^{2d} \\
& ={2d \choose d}\epsilon^d\otimes \epsilon^d
\end{align*} 
%as all other products  give $0$ in one of the co-ordinates by the pigeonhole principle. 
Hence $\epsilon^d\otimes \epsilon^d=0$, a contradiction.
\end{remark}

\begin{lemma}\label{hscommuting}
Let $r$ be of HS-iteration type. Then, on any $\DD$-ring $(R,e)$, $e$ commutes on $R$ with respect to $r^\iota$ if and only if
$$\dd_i\dd_j=c_1^{ij}\dd_1+\cdots+c_m^{ij}\dd_m, \quad \text{ for all } 1\leq i,j\leq m,$$
where $(c_\ell^{ij})$ is the tuple of HS-coefficients of $r$. 
\end{lemma}
\begin{proof}
This is a straightforward computation.
\end{proof}

%As in the previous section, if $(R,\be)$ is a $\bDD$-ring on which $\be$ commutes with respect to $r^\iota$, we will simply say that $(R,\be)$ is a $\bDD^r$-ring. 

In analogy to the Jacobi notion (introduced in the previous section), we now introduce the notion of associativity (which will also be used when $\gDD$-kernels are discussed in Section \ref{kernels}). 

\begin{definition}\label{HScoe}
 Let $r$ be of HS-iteration type with coefficients $(c_\ell^{ij})$. We say $r$ is associative if for each $1\leq i,j,k,r \leq m$
 $$\sum_\ell \left( c_\ell^{ij} c_r^{\ell k} - c_\ell^{jk} c_r^{i\ell}-\sum_{p,q=1}^m\alpha_i^{pq}\dd_p(c_\ell^{jk})c_r^{q\ell}\right)=\dd_i(c_r^{jk})$$
where again recall that $\alpha_i^{pq}\in k$ is the coefficient of $\epsilon_i$ in the product $\epsilon_p\cdot \epsilon_q$.
 \end{definition}

We now give a justification of this notion of associativity. For $r$ of HS-iteration type, we say that a $\DD$-ring $(R,e)$ is a $\DD^r$-ring when $e$ commutes with respect to $e^\iota$.

\begin{proposition}\label{just2}
Let $r$ be of HS-iteration type. If there exists a $\DD^r$-ring $(R,e)$ where the operators $(\dd_1,\dots,\dd_m)$ are $F$-linearly independent (as functions $R\to R$), then $r$ is associative.
\end{proposition}
\begin{proof}
On the one hand, by Lemma~\ref{hscommuting}, we have
$$\dd_i\dd_j\dd_k=\sum_\ell c_\ell^{ij}\dd_\ell\dd_k=\sum_r\left(\sum_{\ell} c_\ell^{ij}c_r^{\ell k}\right)\dd_r$$
On the other hand,
\begin{align*}
\dd_i\dd_j\dd_k &=\sum_\ell \dd_i(c_\ell^{jk} \dd_\ell) \\
& = \sum_\ell \dd_i(c_\ell^{jk})\dd_\ell +\sum_\ell c_\ell^{jk}\dd_i\dd_\ell+\sum_\ell \sum_{p,q\geq 1}\alpha_i^{pq}\dd_p(c_\ell^{jk})\dd_q\dd_\ell \\
& = \sum_r \dd_i(c_r^{jk})\dd_\ell + \sum_{\ell,r}c_\ell^{jk}c_r^{i\ell}\dd_r+\sum_{\ell,r}\sum_{p,q\geq 1}\alpha_i^{pq}\dd_p(c_\ell^{jk})c_{r}^{q\ell}\dd_r \\
& =\sum_r\left( \dd_i(c_r^{jk}) + \sum_{\ell}c_\ell^{jk}c_r^{i\ell}+\sum_{\ell}\sum_{p,q\geq 1}\alpha_i^{pq}\dd_p(c_\ell^{jk})c_{r}^{q\ell}\right) \dd_r.
\end{align*}
As $(\dd_1,\dots,\dd_m)$ are $F$-linearly independent, comparing coefficients on both sides yields the equality in Definition~\ref{HScoe}.
\end{proof}

The content of Corollary~\ref{functionfield} below says that the converse of this proposition also holds. That is, being associative implies the existence of a $\DD^r$-ring where the operators are linearly independent.

\begin{example}\label{assHS}
Assume $char(k)=p>0$ and $F=k$. Suppose we are in the case $\DD=k[\epsilon]/(\epsilon)^{p^n}$ with basis $(1,\epsilon,\dots,\epsilon^{p^n-1})$. Let $r:\DD\to \DD\otimes_k\DD$ be defined by 
$r(\epsilon)=\epsilon\otimes 1 +1\otimes \epsilon.$
Recall that in this case commuting (with respect to $r^\iota$) $\DD$-rings correspond to rings equipped with an iterative Hasse-Schmidt derivation (see Example~\ref{exampleHS}). Furthermore, we observe that in this case $r$ is associative. Indeed, first note that in this case the HS-coefficients are 
$$
c_{\ell}^{ij}=
\left\{
\begin{array}{cc}
\binom{i+j}{i} & \ell=i+j \\
0 & o.w.
\end{array}
\right.
$$
And thus, the condition in Lemma~\ref{HScoe} simplifies to
$$c_{i+j}^{ij}\,c_{i+j+k}^{(i+j)k}=c_{j+k}^{jk}\, c_{i+j+k}^{i(j+k)}.$$
The latter is just
$$\binom{i+j}{i}\cdot \binom{i+j+k}{i+j}=\binom{j+k}{j}\cdot\binom{i+j+k}{i}$$
which is a well known binomial identity.
\end{example}

\begin{example}
Assume $char(k)=p>0$ and $F=k$. Let $\ell,n\in \NN$ and 
$$\DD=k[\epsilon_1,\dots,\epsilon_\ell]/(\epsilon_1^{p^n},\dots,\epsilon_\ell^{p^n}).$$
Let $\mathfrak{g}$ be a finite group scheme over $k$ whose underlying scheme is $\spec(\DD)$. Let $r$ be the co-multiplication in the corresponding Hopf algebra. Recall from Example~\ref{g-der} that a $\mathfrak{g}$-derivation on a $k$ algebra $R$ (in the sense of \cite[Definition 3.8]{HK}) is the same as a $\DD$-operator on $R$ that commutes w.r.t. $r^\iota$. We claim that this $r$ is of HS-iteration type and associative. Indeed, as $\DD$ is local, the co-unit map $E:\DD\to k$ must be the residue map, so the equalities 
$$(E\otimes \id_\DD) \circ r=\id_\DD=(\id_\DD\otimes E) \circ r$$ 
from the definition of Hopf algebra give us precisely that $r$ is of HS-iteration type. 
Now for any $1\leq p\leq \ell$ we have 
\[((r\otimes \id_\DD)\circ r)(X_p)=(r\otimes \id_\DD)(\sum_{q,k} c^{q k}_p X_q\otimes X_k)=\sum_{q,i,j,k}   c_q^{ij} c^{q k}_p X_i\otimes X_j\otimes X_k\]
and, similarly, 
\[(( \id_\DD\otimes r)\circ r)(X_p)=(\id_\DD\otimes r)(\sum_{i,q} c^{i q}_p X_i\otimes X_q)=\sum_{q,i,j,k}  c^{jk}_q c_p^{i q} X_i\otimes X_j\otimes X_k\]
%\[=\sum_{\ell,j,k,r} c_j^{r\ell } c^{k j}_i X_k\otimes X_r\otimes X_j \]

%Now for any $i,j,k,r$,
Comparing the coefficients of $X_i\otimes X_j\otimes X_k$ in 
$$((r\otimes \id_\DD)\circ r)(X_r) \quad \text{ and in }\quad (( \id_\DD\otimes r)\circ r)(X_r)$$ 
which are equal by co-associativity in a Hopf algebra, we get that 
$$\sum_{q} c^{ij}_\ell c_p^{q k} =\sum_q c^{jk}_q c_p^{i q},$$ 
so $r$ is associative in the sense of Definition~\ref{HScoe}.

\end{example}

%{\bf Perhaps we also want to explain that Daniel-Piotr is an example of this setup}

\

\subsection{The general case}\label{generalcase}

We now work with 
$$\bDD=\{(\DD_1,\bar\epsilon_1),(\DD_2,\bar\epsilon_2)\}$$ 
where each $(\DD_u,\bar\epsilon_u)$, $u\in \{1,2\}$, is a local operator-system. The notation from previous sections carries forward to each operator-system by adding an index; for instance, $\mm_u$ denotes the maximal ideal of $\DD_u$ and $dim_k(\DD_u)=m_u+1$, for $u=1,2$. Similarly, we denote the associated operators by $\dd_{u,i}$, and so these are additive operators satisfying
$$\dd_{u,i}(xy)=\dd_{u,i}(x)\, y +x\, \dd_{u,i}(y)+\sum_{p,q=1}^m\alpha_{u,i}^{pq}\dd_{u,p}(x)\dd_{u,q}(y)$$
where $\alpha_{u,i}^{pq}\in k$ is the coefficient of $\epsilon_{u,i}$ in the product $\epsilon_{u,p}\cdot \epsilon_{u,q}$ happening in $\DD_u$.
%for $(p,q)\in \gamma_u(i)$; otherwise, $\alpha_{p,q}^i=0$. See Remark~\ref{operatornotation}.

\medskip

Again, $(F,\be)$ is a fixed $\bDD$-field and all rings are $F$-algebras and $\bDD$-rings are $(F,\be)$-algebras. By a \emph{commutation system} (for $\bDD$ over $F$) we mean a pair
$$\Gamma=\{r_1,r_2\}$$
where each $r_u:\DD_u\to \DD_u\otimes_k\DD_u(F)$ is a $k$-algebra homomorphism. We say that a commutation-system $\Gamma$ is of Lie-Hasse-Schmidt type (or LHS-type for short or simply LHS-commutation system) if $r_1$ is of Lie-commutation type (as in \S\ref{Lietype}) and $r_2$ is of HS-iteration type (as in \S\ref{HStype}). In this case, by the coefficients of $\Gamma$ we mean the tuple $(c_{u,\ell}^{ij})$ from $F$ where $(c_{1,\ell}^{ij})$ are the Lie-coefficients of $r_1$ and $(c_{2,\ell}^{ij})$ are the HS-coefficients of $r_2$.

\medskip

For the remainder of this section we fix an LHS-commutation system $\Gamma$ with coefficients $(c_{u,\ell}^{ij})$.
% We also fix the following notation: for $ u,v\in \{1,2\}$ we set 
%$$\chi{v,u}:\DD_v\to \DD_v\otimes_k \DD_u(F)$$
%to be the canonical embedding ($a\mapsto a\otimes 1$).

\begin{definition}\label{defgammacomm}
Let $A\leq R\leq S$ be an extension of rings and $\be:R\to S$ a $\bDD$-operator such that $\be(A)\subseteq \bDD(R)$. We say that $\be$ commutes on $(A,R,S)$ with respect to $\Gamma$ if
\begin{enumerate}
\item $e_1$ commutes on $(A,R,S)$ with respect to $r_1^{e_1}$,
\item  $e_2$ commutes on $(A,R,S)$ with respect to $r_2^{\iota}$, and
\item $(e_1,e_2)$ commute on $(A,R,S)$ with respect to $r_{12}^{e_1}$.
\end{enumerate}
where $r_{12}:\DD_2\to \DD_1(\DD_2(F))$ is the canonical embedding (i.e., $x\mapsto 1\otimes x$).
When this occurs and $A=R=S$, we may simply say that $(R,\be)$ is a $\gDD$-ring.
\end{definition}

\begin{lemma}\label{meaningcomm}
Let $(c_{u,\ell}^{ij})$ denote the coefficients of $\Gamma$. On any $\bDD$-ring $(R,\be)$, $\be$ commutes on $R$ with respect to $\Gamma$ if and only if 
$$[\dd_{1,i},\dd_{1,j}]=c_{1,1}^{ij}\dd_{1,1}+\cdots+c_{1,m_1}^{ij}\dd_{1,m_1}$$

$$\dd_{2,i}\dd_{2,j}=c_{2,1}^{ij}\dd_{2,1}+\cdots+c_{2,m_2}^{ij}\dd_{2,m_2}$$
and 
$$[\dd_{1,i},\dd_{2,j}]=0.$$
\end{lemma}
\begin{proof}
As in the previous sections, this is a straightforward computation.
\end{proof}

\begin{definition}\label{Jacobiass}
We say that $\Gamma$ is Jacobi-associative if $r_1$ is Jacobi (as in \S\ref{Lietype}), $r_2$ is associative (as in \S\ref{HStype}), and
for $u,v\in \{1,2\}$ with $u\neq v$, $1\leq k\leq m_u$, and $1\leq i,j,r \leq m_v$, we have
$$\dd_{u,k}(c_{v,r}^{ij})=0.$$ 
%where $(c_{u,\ell}^{ij})$ are the coefficients of the commuting system $\Gamma$.
\end{definition}

\smallskip

Putting Definitions \ref{Liecoe}, \ref{HScoe}, \ref{Jacobiass} together, we get:

\begin{remark}\label{altogether}
%Let $\Gamma$ be an LHS-commutation system with coefficients $(c_{u,\ell}^{ij})$. Then, 
The LHS-commutation system $\Gamma$ is Jacobi-associative if and only if it the following three conditions hold:
\begin{enumerate}
\item  for each $1\leq \ell\leq m_1$, the matrix $(c_{1,\ell}^{ij})_{i,j=1}^{m_1}$ is skew-symmetric, \\
for each $1\leq i,j,k,r\leq m$ we have
$$\sum_{\ell=1}^m \left( c_\ell^{ij}c_r^{\ell k}+ c_{\ell}^{ki}c_r^{\ell j}  +c_\ell^{jk}c_r^{\ell i} \right)=\dd_i(c_r^{jk})+\dd_k(c_r^{ij})+\dd_j(c_r^{ki}),$$
for each $1\leq i,j,k,r\leq m$ we have
$$\sum_{p=1}^m \left(\alpha_i^{pr}\dd_p(c_r^{jk})+\alpha_k^{pr} \dd_p(c_r^{ij}) +\alpha_j^{pr}\dd_p(c_r^{ki}) \right)=0,$$
and for each $1\leq q<r\leq m$ we have
$$\sum_{p=1}^m \left(\alpha_i^{pq}\dd_p(c_r^{jk})+\alpha_k^{pq} \dd_p(c_r^{ij}) +\alpha_j^{pq}\dd_p(c_r^{ki}) + \alpha_i^{pr}\dd_p(c_q^{jk})+\alpha_k^{pr} \dd_p(c_q^{ij}) +\alpha_j^{pr}\dd_p(c_q^{ki}) \right)=0.$$

% for each $1\leq i,j,k,r\leq m_1$
%$$\sum_{\ell=1}^{m_1} \left( c_{1,\ell}^{ij}c_{1,r}^{\ell k}+ c_{1,\ell}^{ki}c_{1,r}^{\ell j}  +c_{1,\ell}^{jk}c_{1,r}^{\ell i} \right)=\dd_{1,i}(c_{1,r}^{jk})+\dd_{1,k}(c_{1,r}^{ij})+\dd_{1,j}(c_{1,r}^{ki})$$
%and for each $1\leq i,j,k,r,q\leq m_1$
%$$\sum_{p=1}^{m_1} \left(\alpha_{1,i}^{pq}\dd_{1,p}(c_{1,r}^{jk})+\alpha_{1,k}^{pq} \dd_{1,p}(c_{1,r}^{ij}) +\alpha_{1,j}^{pq}\dd_{1,p}(c_{1,r}^{ki}) \right)=0$$

\item for each $1\leq i,j,k,r \leq m_2$
 $$\sum_{\ell=1}^{m_2} \left( c_{2,\ell}^{ij} c_{2,r}^{\ell k} - c_{2,\ell}^{jk} c_{2,r}^{i\ell}-\sum_{p,q=1}^{m_2}\alpha_{2,i}^{pq}\dd_{2,p}(c_{2,\ell}^{jk})c_{2,r}^{q\ell}\right)=\dd_{2,i}(c_{2,r}^{jk})$$

\item for $u,v\in \{1,2\}$ with $u\neq v$, for all $1\leq k\leq m_u$, and $1\leq i,j,r\leq m_v$ we have
$$\dd_{u,k}(c_{v,r}^{ij})=0.$$
\end{enumerate}
\end{remark}

\medskip

As before, the notion of Jacobi-associativity is justified by the following. By a $\gDD$-ring $(R,\be)$ we mean a $\bDD$-ring such that $\be$ commutes w.r.t. $\Gamma$.
%(see Definition~\ref{defgammacomm}).

\begin{proposition}\label{needindependent}
Let $\Gamma$ be an LHS-commuting system (for $\bDD$ over $F$). If there exists a $\gDD$-ring $(R,\be)$ where the operators 
\begin{itemize}
\item $(\dd_{u,i}: u\in \{1,2\}, 1\leq i\leq m_u)$, 
\item $(\dd_{1,i}\dd_{1,j}: 1\leq i\leq j\leq m_1)$, and 
\item $(\dd_{1,i}\dd_{2,j}: 1\leq i\leq m_1, 1\leq j\leq m_2)$
\end{itemize}
are $F$-linearly independent (as functions $R\to R$), then $\Gamma$ is Jacobi-associative.
\end{proposition}
\begin{proof}
Conditions (1) and (2) of Remark~\ref{altogether} are satisfied by Propositions~\ref{just1} and \ref{just2}. Thus, it remains only to check condition (3). 
Let $u,v\in \{1,2\}$ with $u\neq v$, $1\leq k\leq m_u$, and $1\leq i,j,r\leq m_v$. On the one hand
$$\dd_{v,i}\dd_{v,j}\dd_{u,k}=\dd_{u,k}\dd_{v,i}\dd_{v,j}$$
On the other hand, setting $\beta=1$ when $v=1$ and $\beta=0$ when $v=2$, we get
\begin{align*}
\dd_{v,i}\dd_{v,j}\dd_{u,k}& =\beta \dd_{v,j}\dd_{v,i}\dd_{u,k}+ \sum_\ell c_{v,\ell}^{ij}\dd_{v,\ell}\dd_{u,k} \\
& = \beta \dd_{u,k}\dd_{v,j}\dd_{v,i}+\dd_{u,k}(\dd_{v,i}\dd_{v,j} -\beta\dd_{v,j}\dd_{v,i})- \sum_\ell\dd_{u,k}(c_{v,\ell}^{ij})\dd_{v,\ell} \\
& \quad -\sum_\ell\sum_{p,q\geq 1}\alpha_{u,k}^{pq}\dd_{u,p}(c_{v,\ell}^{ij})\dd_{u,q}\dd_{v,\ell} \\
& = \dd_{u,k}\dd_{v,i}\dd_{v,j} - \sum_\ell\dd_{u,k}(c_{v,\ell}^{ij})\dd_{v,\ell} - \sum_\ell\sum_{p,q\geq 1}\alpha_{u,k}^{pq}\dd_{u,p}(c_{v,\ell}^{ij})\dd_{u,q}\dd_{v,\ell}.
\end{align*}
By the $F$-linear independence assumption (in particular the third bullet point), comparing coefficients yields $\dd_{u,k}(c_{v,\ell}^{ij})=0$.
\end{proof}

Below, in Corollary~\ref{functionfield}, we prove a strong converse of Proposition~\ref{needindependent}.

%By Subsection \ref{app_single_algebra}, the case of one Jacobi commutativity condition of Lie type and finitely many associative conditions of HS-iteration type (together with the condition saying that the operators associated to different algebra trivially commute) can be naturally expressed in our setting with two algebras $\DD_1,\DD_2$. Hence our framework includes, for example, the following situation.

\begin{example}
Assume $char(k)=p>0$ and $F=k$. Consider the case 
$$\DD_1=k[\epsilon_1,\dots,\epsilon_m]/(\epsilon_1,\dots,\epsilon_m)^2 \quad \text{ and } \quad \DD_2=k[\epsilon]/(\epsilon)^{p^{n}}$$
with natural choice of %$\pi_0$, $\pi_u$, 
and ranked bases. Let $(c_{1,\ell}^{ij})_{i,j,\ell=1}^{m}$ be a tuple from $k$ such that for each $\ell$ the matrix $(c_{1,\ell}^{ij})_{i,j=1}^m$ is skew-symmetric. Set $\Gamma=\{r_1,r_2\}$ where
$$r_1(\epsilon_\ell)=1 \otimes \epsilon_\ell+\sum_{i,j=1}^{m}\epsilon_i\otimes \epsilon_j \, c_{\ell}^{ij}, \quad \text{ for }\ell=1,\dots, m,$$
and 
$$r_2(\epsilon)=\epsilon\otimes 1+1\otimes\epsilon.$$
Then, in any $\bDD$-ring $(R,\be)$, $\be$ commutes on $R$ with respect to $\Gamma$ if and only if
\begin{itemize}
\item the derivations $(\dd_{1,1},\dots,\dd_{1,m})$ Lie-commute, with coefficients $(c_{1,\ell}^{ij})$,
\item the truncated Hasse-Schmidt derivation $(\dd_{2,1},\dots,\dd_{2,p^{n}-1})$ is iterative, and
\item for $1\leq u\neq v\leq 2$, the operators $\dd_{u,i}$ and $\dd_{v,j}$ commute.
\end{itemize}
%(cf. Example \ref{allcombined}). 
We further note that in this case $\Gamma$ is Jacobi-associative if and only if, 
%for each $\ell$, the $m\times m$ matrix $(c_{0,\ell}^{ij})_{i,j=1}^m$ is skew-symmetric and, 
for each $1\leq i,j,k,r\leq m$,
$$\sum_{\ell=1}^m  \left( c_{0,\ell}^{ij}c_{0,r}^{\ell k}+ c_{0,\ell}^{ki}c_{0,r}^{\ell j}  +c_{0,\ell}^{jk}c_{0,r}^{\ell i} \right)=0.$$
Indeed, since the $c_{1,\ell}^{ij}$ are from $k$, we get $\dd_{u,p}(c_{0,\ell}^{ij})=0$ for any $p$. Hence, condition (1) of Remark~\ref{altogether} reduces to the above conditions; condition (3) is now trivial; and we have seen in Example~\ref{assHS} that condition (2) holds in this set up. 
\end{example}

\begin{remark}
At the moment we do not know whether our notion of Jacobi-associative LHS-commutation system can be recovered in the setup of generalised Hasse-Schmidt iterativity introduced by Moosa and Scanlon in \cite{MooScan2011}. We leave this for future work.
\end{remark}

\

\section{$\bDD$-kernels}\label{kernels}\label{sec5}

Based on the notion of differential kernels \cite{GLS2018,Lando1970,Pierce2014}, in this section we introduce the notions of $\bDD$-kernel and $\gDD$-kernel. We prove in Theorem~\ref{thebigone} that, under certain conditions, $\gDD$-kernels have (unique) principal realisations. We carry forward the notation of \S\ref{generalcase}. Namely, 
$$\bDD=\{(\DD_1,\bar\epsilon_1),(\DD_2,\bar\epsilon_2)\}$$ 
where each $(\DD_u,\bar\epsilon_u)$ is a local operator-system. Also, $\mm_u$ denotes the maximal ideal of $\DD_u$ and $dim_k(\DD_u)=m_u+1$, for $u\in \{1,2\}$. $(F,\be)$ is a fixed $\bDD$-field and we assume that all rings are $F$-algebras, and that $\bDD$-rings are $(F,\be)$-algebras. We also fix, throughout, a commutation system $\Gamma=\{r_1,r_2\}$ for $\bDD$ over $F$ of Lie-Hasse-Schmidt type, and denote its coefficients by $(c_{u,\ell}^{ij})$.

Recall that we denote the associated operators by $\dd_{u,i}$, and that they are additive operators satisfying
\begin{equation}\label{rulemulti}
\dd_{u,i}(xy)=\dd_{u,i}(x)\, y +x\, \dd_{u,i}(y)+\sum_{p,q=1}^{m_u}\alpha_{u,i}^{pq}\dd_{u,p}(x)\dd_{u,q}(y)
\end{equation}
where $\alpha_{u,i}^{pq}\in k$ is the coefficient of $\epsilon_{u,i}$ in the product $\epsilon_{u,p}\cdot \epsilon_{u,q}$ happening in $\DD_u$, see Remark~\ref{operatornotation}.

%for $(p,q)\in \gamma_u(i)$; otherwise, $\alpha_{u,i}^{p,q}=0$

Also recall that, by Lemma~\ref{meaningcomm}, $\Gamma$-commutativity translates to
$$[\dd_{1,i},\dd_{1,j}]=c_{1,1}^{ij}\dd_{1,1}+\cdots+c_{1,m_1}^{ij}\dd_{1,m_1}$$
\begin{equation}\label{gammacomm}
\dd_{2,i}\dd_{2,j}=c_{2,1}^{ij}\dd_{2,1}+\cdots+c_{2,m_2}^{ij}\dd_{2,m_2}
\end{equation}
$$[\dd_{1,i},\dd_{2,j}]=0.$$

At this point we would like to warn the reader that this section is rather technical and, while we have attempted to give a clear presentation, it does take effort to go through it. We thank the committed reader in advance. 

\medskip

\subsection{Definitions and notation}\label{secnotation} For presentation sake, we fix some notation first. We set
$$\md=\{(u,i): u\in\{1,2\}, \, 1\leq i\leq m_u\}$$
and, for $r\in \mathbb N_0$, denote by $\md^r$ the set of $r$-tuples with entries in $\md$, and by $\md^{\leq r}$ those tuples of length at most $r$. For $i=(u,i')\in \md$, we say that $i$ is of \emph{Lie-type} if $u=1$ and of \emph{HS-type} otherwise (i.e., $u=2$). We set the \emph{type} of an operator $\dd_i$ to be that of $i$. 

\medskip

For $\xi\in \md^r$, we set 
 $$
\chi_\xi= 
\left\{
\begin{array}{cc}
	0 & \text{ if $\xi$ has at least two entries of HS-type} \\
	1 & \text{otherwise} 
\end{array}
\right.
$$
Furthermore, whenever $i,p,q\in \md$ have the same type, we set $\alpha_{i}^{pq}:=\alpha_{u,i'}^{p'q'}$  and $c_{i}^{pq}:=c_{u,i'}^{p',q'}$ where $i=(u,i')$, $p=(u,p')$ and $q=(u,q')$; otherwise we set $\alpha_{i}^{pq}=0$ and $c_{i}^{pq}=0$. This allows us to write the multiplication rule \eqref{rulemulti} as
$$\dd_{i}(xy)=\dd_{i}(x)\, y +x\, \dd_{i}(y)+\sum_{p,q\in \md}\alpha_{i}^{pq}\dd_{p}(x)\dd_{q}(y)$$
and the $\Gamma$-commutativity rule \eqref{gammacomm} simply as
$$\dd_i\dd_j=\chi_{i,j}\dd_i\dd_j+\sum_{\ell\in \md}c_\ell^{ij}\dd_\ell.$$

%\medskip

\begin{notation}\label{notation1}
From now on we equip $\md$ with the following order
$$(2,1)<\cdots <(2,m_2)<(1,1),\cdots < (1,m_1).$$
For $r\in \NN_0$, we set
$$\NN^\md_{r}=\{(i_1,\cdots, i_r)\in \md^r: i_1\geq \cdots \geq i_r\ \text{ and at most one entry $i_j$ is of HS-type}\}$$
and $\NN^\md_{\leq r}=\NN^\md_{0}\cup \cdots \cup \NN^\md_{r}$, and also $\NN^\md_{<\infty}=\bigcup_{r\in \NN_0} \NN^\md_{r}$. Finally, we define 
$$\rho:\md^{<\infty}\to \NN^\md_{<\infty}$$ 
as follows: if $\xi\in \md^{<\infty}$ has at least two entries of HS-type we put $\rho(\xi)=\emptyset$; otherwise we put $\rho(\xi)$ as the (unique) element of  $\NN^\md_{<\infty}$ obtained from $\xi$ by reordering its entries.
\end{notation}

\smallskip

\begin{remark}\label{notation2}
We explain the notation introduced above. For $i\in \md$ and $\xi\in \md^{<\infty}$, we set 
$$\#_i(\xi)=\text{the number of occurrences of $i$ in $\xi$}.$$
We can then construct the map $\psi: \NN^\md_{<\infty}\to \NN_0^{m_1}\times \NN_0^{m_2}$ given by 
$$\xi\mapsto ((\#_{(1,m_1)}(\xi),\dots,\#_{(1,1)}(\xi)),(\#_{(2,m_2)}(\xi),\dots,\#_{(2,1)}(\xi)))$$
and it is straightforward to check that $\psi$ is injective. Furthermore, the image of $\NN^\md_r$ under $\psi$ corresponds to those elements whose coordinates add up to $r$. It is worth noting that for $\xi\in \NN^\md_{<\infty}$ we have $\#_{(2,m_2)}(\xi)+\dots +\#_{(2,1)}(\xi)\leq 1$.
Also, when $\xi\in \md^{<\infty}$ has at most one entry of HS-type, we have that $\rho(\xi)$ is the unique element of $\NN^\md_{<\infty}$ such that 
$$\psi(\rho(\xi))=((\#_{(1,m_1)}(\xi),\dots,\#_{(1,1)}(\xi)),(\#_{(2,m_2)}(\xi),\dots,\#_{(2,1)}(\xi))).$$
\end{remark}

\medskip

Henceforth, we fix an infinite family $(w^\xi:\, \xi\in \NN^\md_{<\infty})$ of algebraic indeterminates over $F$. We define $F$-vector spaces
$$V_F=span_F((w^{\xi})_{\xi\in \NN^\md_{<\infty}}) \text{ and }V_F(r)=span_F((w^{\xi})_{\xi\in \NN^\md_{\leq r}}).$$
As convention we set $V_F(-1)$ to be the null vector space. The first intermediate step towards the main result of this section (Theorem \ref{thebigone}) is to prove that we can equip $V_F$ with suitable additive operators $(\dd_{i}:i\in \md)$ in a way that they $\Gamma$-commute. This is made precise in Lemma~\ref{horrible} below.

\begin{definition}\label{def_ell}
For each $i\in \md$ and $r\in \NN_0$, we define $\dd_i:V_F(r)\to V_F(r+1)$ by induction on $r$ as follows. For $r=0$, we put
$$\dd_i( cw^\emptyset) =\dd_i(c)w^\emptyset+cw^i+\sum_{p,q\in \md} \alpha_{i}^{pq}\dd_p(c)w^q$$
Now assume $r\geq 1$. We first define $\dd_i(w^{\xi})$ for $\xi\in \NN_{r}^\md$. For $i=(1,m_1)$, we set
$$\dd_{(1,m_1)}(w^\xi)=w^{((1,m_1),\xi)}.$$
We may now assume that we have defined $\dd_k$ for all $k>i$. Now write $\xi=(j,\eta)$ and consider two cases.

\smallskip

\noindent {\bf Case 1}. Suppose $i\geq j$. On the one hand, if $i$ is of HS-type, $j$ must also be of HS-type, and we put
$$\dd_i(w^\xi)=\sum_{\ell\in \md}c_{\ell}^{ij}\dd_\ell(w^{\eta})$$
where note that $\dd_\ell(w^\eta)$ has already been defined by induction (as $\eta\in \NN^\md_{r-1}$). On the other hand, if $i$ is of Lie-type, we put 
$$\dd_i(w^\xi)=w^{(i,\xi)}.$$

\smallskip		

\noindent {\bf Case 2}. Suppose $i<j$. In this case $\dd_{j}(\dd_i(w^\eta))$ has already been defined and hence we may set
		\[\dd_i(w^\xi)=\chi_{i,j}\dd_j(\dd_i(w^\eta))+\sum_{\ell\in \md}c_{\ell}^{ij}\dd_\ell(w^\eta).\]

Now that $\dd_i(w^\xi)$ has been defined for all $\xi\in \NN_{r}^\md$ and $i\in \md$, we put 
$$\dd_{i}(cw^\xi)=\dd_i(c)\,w^\xi+c\, \dd_i(w^\xi)+\sum_{p,q\in \md}\alpha_i^{pq}\dd_p(c) \dd_q (w^\xi)$$
and extend additively to all of $V_F(r)$. This defines the desired operators from $V_F(r)\to V_F(r+1)$.
		\end{definition} 

Note that the operators defined above yield additive operators
$$\dd_i:V_F \to V_F, \quad \text{ for } i\in \md,$$
such that 
$$\dd_{i}(cv)=\dd_i(c)\,v+c\, \dd_i(v)+\sum_{p,q\in \md}\alpha_i^{pq}\dd_p(c) \dd_q (v)$$
for all $c\in F$ and $v\in V_F$.

Now consider the auxiliary function $\ell_*:\md^{<\infty} \to V_F$ defined as follows: for $\xi=(i_1,\dots,i_r)\in\md^{r}$,
$$\ell_\xi:=\dd_\xi w^\emptyset  - \chi_\xi w^{\rho(\xi)}$$
where $\dd_\xi=\dd_{i_1}\cdots\dd_{i_r}$ and $\rho:\md^{<\infty}\to \NN^\md_{<\infty}$ was defined in Notation~\ref{notation1}. Note that from the definition we get $\dd_\xi w^\emptyset=\chi_\xi\dd_{\rho(\xi)}w^\emptyset +\ell_\xi$. It is also worth noting that
\begin{itemize}
\item for $i<j\in \md$; we have $\ell_{i,j}=\sum_{k\in\md}c_k^{ij}\dd_k(w^\emptyset) \in V_F(1)$,
\item if $\xi\in \NN^\md_{<\infty}$, then $\ell_\xi=0$, 
\item for $\xi\in \NN^\md_{<\infty}$ with all entries of Lie-type, if $j\in \md$ is of HS-type, then $\ell_{j,\xi}=0$.
\end{itemize}

\medskip

We now prove a series of lemmas establishing the connection of the auxiliary function $\ell_{*}$ and the operators $\dd_i:V_F\to V_F$, culminating in the promised Lemma~\ref{horrible}.

\medskip

To reduce notation, in the proofs of the following three lemmas when we write $\dd_\tau$ (with no term immediately after) we mean $\dd_\tau w^{\emptyset}$ (i.e., $\dd_\tau(w^\emptyset)$) for any $\tau\in \md^{<\infty}$, where recall that $\dd_\tau=\dd_{i_1}\cdots\dd_{i_r}$ when $\tau=(i_1,\dots,i_r)$. With this notation in mind, we have $\dd_{\rho(\tau)}=w^{\rho(\tau)}$, which will be used repeatedly.

\begin{lemma}\label{properror} If $\xi \in \md^r$, then $\ell_\xi\in V_F(r-1)$.
	\end{lemma}
	\begin{proof}
		We proceed by induction on $r$. When $r\leq 1$ the conclusion is obvious as in this case $\ell_\xi=0$. So let us assume  $r\geq 2$. Write $\xi=(i,\eta)$ for some $\eta\in \md^{r-1}$, and let $j\in \md$ and $\lambda\in \NN_{r-2}^{\md}$ be such that $\chi_\eta\dd_{\rho(\eta)}=\chi_{\eta}\dd_{j,\lambda}$ (recall that when $\rho(\eta)=\emptyset$ we have $\chi_\eta=0$). Similar to Definition~\ref{def_ell}, we consider two cases:
	
	\medskip
\noindent {\bf Case 1.} Suppose $i\geq j$. On the one hand, if $i$ is of HS-type, $j$ must also be of HS-type and then $\chi_{\xi}=0$, and so, using Case (1) of Definition~\ref{def_ell} in the fifth equality below, we get
\begin{align*}
	\ell_{\xi} & =\dd_{\xi}-\chi_\xi\dd_{\rho(\xi)} \\
	& = \dd_i\dd_{\eta} \\
	& = \dd_i(\chi_\eta\dd_{\rho(\eta)} + \ell_\eta)   \\
	& = \chi_\eta \dd_{i,\rho(\eta)} + \dd_i \ell_\eta   \\
	& = \sum_{k\in \md}\chi_\eta \, c_{k}^{ij}\, \dd_k\dd_\lambda + \dd_i \ell_\eta.
	\end{align*}
and both of the last two summands are in $V_F(r-1)$. On the other hand, if $i$ is of Lie-type, we have $\chi_{\xi}=\chi_\eta$ and $\chi_\xi \dd_{\rho(\xi)}=\chi_{\xi}\dd_{i,\rho(\eta)}$; and so 
	\begin{align*}
	\ell_{\xi} %& =\dd_{\xi}-\chi_\xi\dd_{\rho(\xi)} \\
	 & = \dd_i\dd_{\eta} - \chi_{\xi} \dd_{\rho(\xi)} \\
	& = \dd_i(\chi_\eta\dd_{\rho(\eta)} + \ell_\eta)  - \chi_{\xi} \dd_{\rho(\xi)}  \\
	& = \chi_\xi \dd_{i,\rho(\eta)} + \dd_i \ell_\eta   - \chi_{\xi} \dd_{\rho(\xi)} \\
	& = \dd_i \ell_\eta
	\end{align*}
	and, by induction, $\dd_i \ell_n \in V_F(r-1)$. 

\medskip
\noindent {\bf Case 2.} Suppose $i<j$. Using Case (2) of Definition~\ref{def_ell} in the third equality below, we get

\begin{align*}
	\ell_{\xi} %& =\dd_{\xi}-\chi_\xi\dd_{\rho(\xi)} \\
	%& = \dd_i\dd_{\eta} - \chi_{\xi} \dd_{\rho(\xi)} \\
	& = \dd_i(\chi_\eta\dd_{\rho(\eta)} + \ell_\eta)  - \chi_{\xi} \dd_{\rho(\xi)}  \\
	& = \chi_\eta \dd_i\dd_j\dd_{\lambda} + \dd_i \ell_\eta  - \chi_{\xi} \dd_{\rho(\xi)} \\
	& = \chi_\eta\chi_{i,j}\dd_j\dd_i\dd_\lambda +\sum_{k\in \md}\chi_\eta c_k^{ij}\dd_k
	\dd_\lambda+ \dd_i \ell_\eta   - \chi_{\xi} \dd_{\rho(\xi)} \\
	& = \chi_\eta\chi_{i,j}\dd_j(\chi_{i,\lambda}\dd_{\rho(i,\lambda)}+\ell_{i,\lambda}) +\sum_{k\in \md}\chi_\eta c_k^{ij}\dd_k
	\dd_\lambda+ \dd_i \ell_\eta  - \chi_{\xi} \dd_{\rho(\xi)} \\
	& = \chi_\xi\dd_{\rho(\xi)}+\chi_\eta\chi_{i,j}\dd_j \ell_{i,\lambda}  +\sum_{k\in \md}\chi_\eta c_k^{ij}\dd_k
	\dd_\lambda+ \dd_i \ell_\eta  - \chi_{\xi} \dd_{\rho(\xi)} \\
	& =\chi_\eta\chi_{i,j}\dd_j \ell_{i,\lambda}  +\sum_{k\in \md}\chi_\eta c_k^{ij}\dd_k
	\dd_\lambda+ \dd_i \ell_\eta
	\end{align*}
	Using induction, we see all of the last summands are in $V_F(r-1)$.
	\end{proof}

\begin{lemma}\label{errors} Suppose $\lambda=(k,\eta)\in \NN_{<\infty}^\md$ and $i,j\in \md$.

	\begin{enumerate}
			\item\label{for1} 	 If $k>i$ then
		\[
		\ell_{i,\lambda}=\chi_{i,k}\dd_k\ell_{i,\eta}+\sum_{\ell \in \md} c_\ell^{ik}\dd_\ell\dd_\eta w^{\emptyset}
		\]
		
		\item\label{for2} If  $k>i,j$ then
		\[
		\ell_{j,\rho(i,\lambda)}=\chi_{j,k}\dd_k \, \ell_{j,\rho(i,\eta)} + \sum_{\ell \in \md} c_\ell^{jk}\dd_\ell\dd_{\rho(i,\eta)}w^{\emptyset}.
		\]
	\end{enumerate}
\end{lemma}
\begin{proof}
Note that (2) follows by applying (1) with $\rho(i,\lambda)$ in place of $\lambda$ and $j$ in place of $i$ (noting that $\rho(i,\lambda)=(k,\rho(i,\eta))$), so it suffices to prove (1). The computation in Case (2) of Lemma~\ref{properror} with $(i,\lambda)$ in place of $\xi$ yields
$$\ell_{i,\lambda}=\chi_{k,\eta}\chi_{i,k}\dd_k\ell_{i,\eta} + \sum_{\ell\in \md}\chi_{k,\eta}c_\ell^{ik}\dd_\ell\dd_\eta+\dd_i \ell_{k,\eta}$$
but as $(k,\eta)\in \NN_{<\infty}^\md$, we have $\chi_{k,\eta}=1$ and $\ell_{k,\eta}=0$, and so the above reduces to the desired equality. 
\end{proof}

The following formulas will be used in the proof of Theorem \ref{thebigone}.

\begin{lemma}\label{horrible}
	Suppose $\Gamma$ is Jacobi-associative (see \S\ref{generalcase}). For any $\lambda\in  \NN_{<\infty}^\md$ and $i,j\in \md$ we have 
	$$(1) \hspace{1.1cm}  \chi_{j,\lambda}\ell_{i,\rho(j,\lambda)}+\dd_i\ell_{j,\lambda}=\chi_{i,j}\chi_{i,\lambda}\ell_{j,\rho(i,\lambda)}+\chi_{i,j}\dd_j\ell_{i,\lambda}+\sum_{\ell\in \md} c_\ell^{ij}\dd_\ell\dd_\lambda w^\emptyset$$
	and 
	$$(2) \hspace{2.6cm} \dd_i\dd_j\dd_\lambda w^\emptyset=\chi_{i,j}\dd_j\dd_i\dd_\lambda w^\emptyset+\sum_{\ell\in\md}c^{ij}_
	\ell \dd_\ell\dd_\lambda w^\emptyset \hspace{2cm} $$
	
\end{lemma}
\begin{proof}
	We first prove that for any fixed $\lambda\in  \NN_{<\infty}^\md$ and $i,j\in \md$, condition (1) implies condition (2). Indeed, assuming (1) and using it in the third equality below, we get 
	\begin{align*}
	\dd_i\dd_j\dd_\lambda & =\dd_{i}(\chi_{j,\lambda} \dd_{\rho(j,\lambda)})+\dd_i\ell_{j,\lambda} \\
	& =\chi_{i,j,\lambda}\dd_{\rho(i,j,\lambda)}+\chi_{j,\lambda}\ell_{i,\rho(j,\lambda)}+\dd_{i}\ell_{j,\lambda} \\
	& =\chi_{i,j,\lambda}\dd_{\rho(i,j,\lambda)}+\chi_{i,j}\chi_{i,\lambda}\ell_{j,\rho(i,\lambda)}+\chi_{i,j}\dd_{j}\ell_{i,\lambda}+\sum c^{ij}_\ell \dd_\ell \dd_\lambda \\
	& =\chi_{ij}(\chi_{j,i,\lambda}\dd_{\rho(j,i,\lambda)}+\chi_{i,\lambda}\ell_{j,\rho(i,\lambda)})+\chi_{ij}\dd_{j}\ell_{i,\lambda}+\sum c^{ij}_\ell \dd_\ell \dd_\lambda \\
	& =\chi_{ij}(\dd_j(\chi_{i,\lambda}\dd_{\rho(i,\lambda)}))+\chi_{ij}\dd_{j}\ell_{i,\lambda}+\sum c^{ij}_\ell \dd_\ell \dd_\lambda \\
	& =\chi_{ij}(\dd_j(\chi_{i,\lambda}\dd_{\rho(i,\lambda)}+\ell_{i,\lambda}))+\sum c^{ij}_\ell \dd_\ell \dd_\lambda \\
	& =\chi_{ij}\dd_j\dd_i\dd_\lambda+\sum c^{ij}_\ell \dd_\ell \dd_\lambda.
	\end{align*}
	This yields (2).
	
	We prove (1) by induction on $|\lambda|$, the length of $\lambda$ as a tuple (i.e., $\lambda\in \md^{|\lambda |}$). We now assume (1), and hence also (2), holds for all $\eta\in \NN_{<|\lambda|}^\md$. 
	
	We will make use of the following claim (note that its proof relies on the Jacobi assumption on $\Gamma$).
	
	\begin{claim}\label{jacobi_fix}
		Assume $\eta\in \NN_{<|\lambda|}^\md$ and that $i,j,k\in \md$ are of Lie-type. Then
	\begin{align*}
	\sum_\ell \dd_i(c_\ell^{jk}\dd_\ell\dd_{\eta})  & + \sum_\ell \dd_j(c_\ell^{ki}\dd_\ell\dd_{\eta})  +\sum_\ell \dd_k(c_\ell^{ij}\dd_\ell\dd_{\eta})\\
	& +\sum_\ell c_\ell^{kj}\dd_\ell (\dd_{i}\dd_{\eta})+ \sum_\ell c_\ell^{ik}\dd_\ell (\dd_{j}\dd_{\eta})+ \sum_\ell c_\ell^{ji}\dd_\ell (\dd_{k}\dd_{\eta})=0.
	\end{align*}
	\end{claim}
	\begin{proof}%[Proof of Claim~\ref{jacobi_fix}]
	Using the Jacobi condition and that equality (2) holds for $\eta$, going backwards in the calculation in the second part of the proof of Proposition \ref{just1} yields the desired equality.
	\end{proof}

To continue with the proof of the lemma, we consider two cases.

\medskip

\noindent {\bf Case 1.} Assume that at least one of $i$ and $j$ is of Lie-type, i.e. $\chi_{i,j}=1$.
	
				\smallskip
	
	Note that if $i=j$ then both $i$ and $j$ are of Lie-type and the equality in (1) is clear (recalling that $\Gamma$ being Jacobi implies $c_\ell^{ii}=0$). So we may assume $i< j$ and that $j$ is of Lie-type. 
	
	The formula in (1) slightly reduces to
	$$\ell_{i,\rho(j,\lambda)}+\dd_i\ell_{j,\lambda}=\chi_{i,\lambda}\ell_{j,\rho(i,\lambda)}+\dd_j\ell_{i,\lambda}+\sum c_\ell^{ij}\dd_\ell\dd_\lambda$$
	
 Assume first that  $|\lambda|=0$. When $i$ is of Lie-type (recall that $j$ is assumed to be of Lie-type) the equality reduces to
	$$\ell_{i,j}=\sum c_\ell^{ij}\dd_\ell$$
	which is clearly true. On the other hand, if $i$ is of HS-type, then $c_\ell^{ij}=0$ and $\ell_{i,j}=\ell_{j,i}=0$ from which the equality follows.
	
	\
	
	Now assume $|\lambda|>0$ and write $\lambda=(k,\eta)$ with $k\in \md$ and $\eta\in \NN^\md_{<|\lambda|}$. 	
	We consider two subcases.
	
	\
	
	\noindent {\bf Case 1.1.} Assume $k\leq j$. In this case we have $\ell_{j,\lambda}=0$ and $\ell_{j,\rho(i,\lambda)}=0$. Thus equality (1) reduces to 
	$$\ell_{i,\rho(j,\lambda)}=\dd_j\ell_{i,\lambda}+\sum_\ell c_\ell^{ij}\dd_\ell\dd_\lambda $$
	but this holds by Lemma ~\ref{errors}(1).
		
	\
	
	\noindent {\bf Case 1.2.} Assume $k>j$ (note that this implies that $k$ is of Lie-type). In this case, using Lemma \ref{errors} and that $\chi_{i,k}=\chi_{j,k}=1$ in the second equality below, and the inductive assumption (2) applied to (the summands in) $\ell_{j,\eta}\in V_F(|\lambda|-1)$ in the third equality, we get:
	\begin{align*}
		\ell_{i,\rho(j,\lambda)}+\dd_i\ell_{j,\lambda} & =\ell_{i,k,\rho(j,\eta)}+\dd_i\ell_{j,k,\eta} \\
		& = \dd_k\ell_{i,\rho(j,\eta)}+ \sum_\ell c_\ell^{ik}\dd_\ell\dd_{\rho(j,\eta)} + \dd_i(\dd_k\ell_{j,\eta})+\sum_\ell \dd_i(c_\ell^{jk}\dd_\ell\dd_\eta) \\
		& = \dd_k\left( \ell_{i,\rho(j,\eta)} +\dd_i\ell_{j,\eta} \right)+ \sum_\ell c_\ell^{ik}\dd_\ell\ell_{j,\eta} + \sum_\ell c_\ell^{ik}\dd_\ell\dd_{\rho(j,\eta)} + \sum_\ell \dd_i(c_\ell^{jk}\dd_\ell\dd_\eta)\\
		& = \dd_k\left( \ell_{i,\rho(j,\eta)} +\dd_i\ell_{j,\eta} \right)+ \sum_\ell c_\ell^{ik}\dd_\ell\dd_{j}\dd_{\eta} + \sum_\ell \dd_i(c_\ell^{jk}\dd_\ell\dd_\eta)\\
	\end{align*}
	
	Since $|\eta|<|\lambda|$ and $\ell_{i,\eta}\in V_F(|\lambda|-1)$, we can apply inductive assumptions (1) and (2), and Lemma \ref{errors}\eqref{for2} together with $c_\ell^{jk}=-c_\ell^{kj}$, to get that the above is equal to
	\begin{align*}
		&  \dd_k(\chi_{i,\eta}\ell_{j,\rho(i,\eta)}+ \dd_j\ell_{i,\eta}+\sum_\ell c_\ell^{ij}\dd_\ell\dd_\eta) + \sum_\ell c_\ell^{ik}\dd_\ell\dd_{j}\dd_{\eta} +  \sum_\ell \dd_i(c_\ell^{jk}\dd_\ell\dd_\eta) \\
		& =\chi_{i,\eta} \dd_k\ell_{j,\rho(i,\eta)}+ \dd_k\dd_j\ell_{i,\eta}+\sum_\ell \dd_k(c_\ell^{ij}\dd_\ell\dd_\eta) + \sum_\ell c_\ell^{ik}\dd_\ell\dd_{j}\dd_{\eta} +  \sum_\ell \dd_i(c_\ell^{jk}\dd_\ell\dd_\eta)\\
		& =\chi_{i,\lambda} \ell_{j,\rho(i,\lambda)}+ \chi_{i,\eta}\sum_\ell c_\ell^{kj}\dd_\ell\dd_{\rho(i,\eta)} + \dd_j\dd_k\ell_{i,\eta}+ \sum_\ell c_\ell^{kj}\dd_\ell \ell_{i,\eta}+\\ &  + \sum_\ell \dd_k(c_\ell^{ij}\dd_\ell\dd_\eta) + \sum_\ell c_\ell^{ik}\dd_\ell \dd_{j}\dd_{\eta} + \sum_\ell \dd_i(c_\ell^{jk}\dd_\ell\dd_\eta)\\
		& =\chi_{i,\lambda}\ell_{j,\rho(i,\lambda)}+  \dd_j\dd_k\ell_{i,\eta}+ \sum_\ell c_\ell^{kj}\dd_\ell \dd_{i}\dd_{\eta}  + \sum_\ell \dd_k(c_\ell^{ij}\dd_\ell\dd_\eta) + \sum_\ell c_\ell^{ik}\dd_\ell \dd_{j}\dd_{\eta} + \sum_\ell \dd_i(c_\ell^{jk}\dd_\ell\dd_\eta)\\
	\end{align*}
	
	Now applying Lemma \ref{errors}\eqref{for1}, and the fact that $c_\ell^{ik}=-c_\ell^{ki}$, to the 2nd term of the last line we get that the above sum equals
	\begin{align*}
		& \chi_{i,\lambda} \ell_{j,\rho(i,\lambda)}+  \dd_j\ell_{i,\lambda}  \\
		\quad & + \sum_\ell \dd_j(c_\ell^{ki}\dd_\ell\dd_\eta) + \sum_\ell c_\ell^{kj}\dd_\ell \dd_{i}\dd_{\eta} + \sum_\ell \dd_k(c_\ell^{ij}\dd_\ell\dd_\eta) + \sum_\ell c_\ell^{ik}\dd_\ell \dd_{j}\dd_{\eta} + \sum_\ell \dd_i(c_\ell^{jk}\dd_\ell\dd_\eta)\\
	\end{align*}
	
	Equality (1) will follow once we show that the above sum equals 
	$$\chi_{i,\lambda} \ell_{j,\rho(i,\lambda)}+  \dd_j\ell_{i,\lambda}+\sum_\ell c_\ell^{ij}\dd_\ell(\dd_{k}\dd_{\eta})$$
	%$$\chi_{i,\lambda} \ell_{j,\rho(i,\lambda)}+  \dd_j\ell_{i,\lambda}+\sum_\ell c_\ell^{ij}\dd_\ell\dd_{\lambda}$$ 
	In other words, using that $c_\ell^{ij}=-c_\ell^{ji}$, we must show that
	
	\begin{align*}
	(\dagger)\quad \sum_\ell \dd_i(c_\ell^{jk}\dd_\ell\dd_{\eta})  & + \sum_\ell \dd_j(c_\ell^{ki}\dd_\ell\dd_{\eta})  +\sum_\ell \dd_k(c_\ell^{ij}\dd_\ell\dd_{\eta})\\
	& +\sum_\ell c_\ell^{kj}\dd_\ell \dd_{i}\dd_{\eta}+ \sum_\ell c_\ell^{ik}\dd_\ell \dd_{j}\dd_{\eta}+ \sum_\ell c_\ell^{ji}\dd_\ell \dd_{k}\dd_{\eta}=0.
	\end{align*}
	
	\begin{itemize}
	\item When $i$ is of Lie-type, ($\dagger$) holds by Claim \ref{jacobi_fix} (recall that this claim uses the Jacobi condition on $\Gamma$).
	\medskip
	
	\item When $i$ is of HS-type, using that $j, k$ are of Lie-type, note that $c_\ell^{ki}=c_\ell^{ij}=c_\ell^{ik}=c^{ji}_\ell=0$. Also, $\dd_i(c_\ell^{jk})=0$ by Jacobi-associativity of $\Gamma$, and $\dd_\ell\dd_i=\dd_i\dd_\ell$ when $\ell$ is of Lie-type. All this implies that in this case ($\dagger$) reduces to 
	$$\sum_\ell c_\ell^{kj}\dd_\ell \dd_{i}\dd_{\eta} +  \sum_\ell  \dd_i(c_\ell^{jk}\dd_\ell\dd_\eta)=\sum_\ell (c^{kj}+c^{jk})\dd_\ell\dd_i\dd_\eta =0$$
	but this holds as $c_\ell^{kj}+c_\ell^{jk}=0$ (since $j,k$ are of Lie-type, recall that Jacobi requires skew-symmetry of the Lie-coefficients). 
	\end{itemize}
We have finished the proof for (sub)Case 1.2 and hence also for Case 1.

	\medskip
	
	\noindent{\bf Case 2.} Now assume that both $i$ and $j$ are of HS-type, so $\chi_{ij}=0$.
	
	\medskip

	In this case we have	$$\dd_i\ell_{j,\lambda}=\dd_i\dd_j\dd_\lambda - \chi_{j,\lambda}\dd_i\dd_{\rho(j,\lambda)}$$
	
	so, to prove (1), it is enough to prove that $\dd_i\dd_j\dd_\lambda =\sum_\ell c_\ell^{ij}\dd_\ell\dd_\lambda $.
	
	\medskip
	
	Write $\lambda=(k,\eta)$ for some $k\in \md$ and $\eta\in \NN^\md_{<|\lambda|}$.  First assume that $k$ is of HS-type. Using the inductive assumption (2) and the same calculation as in Proposition \ref{just2}, we get	
	\begin{align*}
	 \dd_{i}\dd_{j}\dd_{\lambda} & = \dd_{i}\dd_{j}\dd_{k}\dd_{\eta} \\
	 & =\sum_{r}\left( \dd_{i}(c_{r}^{jk}) + \sum_{\ell}c_{\ell}^{jk}c_{r}^{i\ell}+\sum_{\ell}\sum_{pq}\alpha_{i}^{pq}\dd_{p}(c_{\ell}^{jk})c_{r}^{q\ell}\right) \dd_{r}\dd_{\eta}.
	 \end{align*}
	By associativity of $\Gamma$, this equals $ \sum_r(\sum_\ell c_{\ell}^{ij}c_{r}^{\ell k})\dd_{r}\dd_{\eta}$, which, again by the inductive assumption, equals 
	$$\sum_{\ell} c_{\ell}^{ij}\dd_{\ell}\dd_{k}\dd_{\eta}=\sum_{\ell} c_{\ell}^{ij}\dd_{\ell}\dd_{\lambda}$$ 
	yielding the desired equality.
	
	\smallskip
	Now assume $k$ is of Lie-type. Then, using the inductive assumption (2) in the second equality below, using that we have already proved (2) in case one of the operators is of Lie-type (Case 1) in the third equality,  using the inductive assumption (2) in the fourth equality, and $\dd_k(c_\ell^{ij})=0$ in the fifth one (this is due to Jacobi-associativity of $\Gamma$), we get
	\begin{align*}
	\dd_{i}\dd_{j}\dd_{\lambda} & =\dd_{i}\dd_{j}\dd_{k} \dd_{\eta} \\
	& =\dd_{i}\dd_{k}\dd_{j} \dd_{\eta} \\
	& = \dd_{k}\dd_{i}\dd_{j} \dd_{\eta} \\
	& = \dd_{k}( \sum_{\ell} c^{ij}_\ell \dd_\ell\dd_{\eta}) \\
	& =\sum_{\ell} c^{ij}_\ell \dd_{k} \dd_\ell\dd_{\eta} \\
	& =\sum_{\ell} c^{ij}_\ell  \dd_\ell \dd_{k}\dd_{\eta} \\
	& =\sum_{\ell} c^{ij}_\ell  \dd_\ell  \dd_{\lambda}
	\end{align*} 
	as required.

\smallskip
	
	This covers all cases, and so we have proved the lemma.

	%\[\dd_i(\dd_j\dd_\lambda=\chi_{ij}\dd_j\dd_i\lambda+\sum_{\ell}c^{ij}_
	%\ell \dd_\ell(\dd_\lambda)\]
\end{proof}

\begin{remark}\
\begin{enumerate}
\item The above lemma shows that the operators $(\dd_i)_{i\in \md}$ defined on $V_F$ satisfy $\Gamma$-commutativity; namely, that for all $v\in V_F$ we have
$$(\star) \hspace{2.2cm} \dd_i\dd_j (v) =\chi_{i,j}\dd_j\dd_i (v)+\sum_{\ell\in\md}c^{ij}_\ell \dd_\ell(v). \hspace{2cm}$$
Indeed, it suffices to check this holds for all $w^\lambda$ for $\lambda\in \NN^\md_{<\infty}$, but this follows from part (2) of Lemma~\ref{horrible} after noting $w^\lambda=\dd_\lambda w^\emptyset$.
\item One can consider the natural notion of $\bDD$-vector spaces over $(F,\be)$; namely, an $F$-vector space $U$ equipped with additive operators $(\dd_i:U \to U)_{i\in \md}$ such that  for $c\in F$ and $v\in U$
$$\dd_{i}(cv)=\dd_i(c)\,v+c\, \dd_i(v)+\sum_{p,q\in \md}\alpha_i^{pq}\dd_p(c) \dd_q (v),$$
and we can say that $(U,(\dd_i)_{i\in \md})$ is a $\gDD$-vector space if the operators $\Gamma$-commute (in the sense of ($\star$) above). One can easily see that the argument in Proposition~\ref{needindependent} holds in this context; namely, if there exists a $\gDD$-vector space $(U,(\dd_i)_{i\in \md})$ where the operators 
\begin{itemize}
\item $(\dd_{i}: i\in \md)$, and
\item $(\dd_{i}\dd_{j}: i\geq j\in \md \text{ with $i$ of Lie-type})$
%\item $(\dd_{i}\dd_{j}: i,j\in \md \text{ with $i$ of Lie-type and $j$ of HS-type})$
\end{itemize}
are $F$-linearly independent (as functions $U\to U$), then $\Gamma$ is Jacobi-associative. What part (2) of Lemma~\ref{horrible} shows is the converse; that is, if $\Gamma$ is Jacobi-associative we can build a $\gDD$-vector space satisfying these linear independencies -- namely the vector space $V_F$ -- .
\end{enumerate}
\end{remark}

\bigskip

We now introduce the notion of $\bDD$-kernel and $\gDD$-kernel. Let $(K,\be)$ be a $\bDD$-field such that $\be$ commutes on $K$ with respect to $\Gamma$. When the context is clear we will simply say that $(K,\be)$ is a  $\bDD$-field with $\Gamma$-commuting operators or that $(K,\be)$ is a $\gDD$-field. 

\medskip

Let $r\in \NN_0$, $n\in \NN$ and $\mn:=\{1,\dots,n\}$. A \emph{$\bDD$-kernel of length $r$ (in $n$-variables) over $(K,\be)$} consists of a field extension of the form 
$$L_r=K(a_t^\xi: (\xi,t)\in \md^{\leq r}\times \mn)$$
together with a $\bDD$-operator $\be:L_{r-1}\to \bDD(L_r)$ extending that on $K$ such that for each $u\in\{1,2\}$ and $\xi\in \md^{\leq r-1}$ we have
$$e_u(a_t^\xi)=\epsilon_{u,0}\, a_t^\xi +\epsilon_{u,1}\, a_t^{((u,1),\xi)}+\cdots +\epsilon_{u,m_u}\, a_t^{((u,m_u),\xi)},$$
where $\epsilon_{u,0}=1$. In terms of the associated operators, and using the notation introduced above, this can be written as 
$$\dd_{i}(a_t^\xi)=a_t^{(i,\xi)} \quad \text{ for $i\in \md$}.$$
We set $L_{-1}=K$ and normally assume that $r\geq 2$. Note that $\be(L_{r-2})\subseteq \bDD(L_{r-1})$. 

\begin{definition}\label{def5.1} Let $(L_r,\be)$ be a $\bDD$-kernel. When $r\geq 2$, we say that the $\bDD$-kernel has $\Gamma$-commuting operators if $\be$ commutes on $(L_{r-2},L_{r-1},L_r)$ with respect to $\Gamma$ (see Definition~\ref{defgammacomm}). In this case we may also say that $L_r$ is a $\gDD$-kernel.
\end{definition}

Note that, by Lemma~\ref{meaningcomm}, a $\bDD$-kernel has $\Gamma$-commuting operators (i.e. it is a $\gDD$-kernel) if and only if for every $a_t^{\xi}$ with $\xi\in \md^{\leq r-2}$ we have
$$\dd_i\dd_j(a_t^\xi)=\chi_{i,j}\dd_i\dd_j(a_t^\xi)+\sum_{\ell\in \md}c_\ell^{ij}\dd_\ell(a_t^\xi).$$
for all $i,j\in \md$

\medskip

Using the injective map $\psi:\NN^\md_{<\infty}\to \NN_0^{m_1+m_2}$ introduced in Remark~\ref{notation2}, we now equip $\NN_{<\infty}^\md\times \mn$ with two orders: for $(\xi,t)$ and $(\eta,t)$ in $\NN_{<\infty}^\md\times \mn$ we set $(\xi,t)\leq(\eta,t')$ if $t=t'$ and $\psi(\xi)\leq \psi(\eta)$ in the product order of $\NN^{m_1+m_2}$; on the other hand, $(\xi,t) \ineq(\eta,t')$ when 
$$(|\xi|,t,\psi(\xi))\leq_{\operatorname{lex}}(|\eta|,t',\psi(\eta))$$
where $\leq_{\operatorname{lex}}$ denotes the left-lexicographic order on $\NN_0^{2+m_1+m_2}$ and recall that $|\xi|$ denotes the length of $\xi$ (i.e., $\xi\in \md^{|\xi|}$). Note that $(\xi,t) \ineq(\eta,t')$ if and only if $(|\xi|,t,\xi)\leq_{\operatorname{lex}}(|\eta|,t',\eta)$ where now $\leq_{\operatorname{lex}}$ is left-lexicographic on $\NN_0^2\times \md^{<\infty}$.

\medskip

\begin{definition}
Let $(L_r,\be)$ be a $\gDD$-kernel and $(\xi,t)\in \NN_{\leq r}^\md\times \mn$.
\begin{enumerate}
\item We set 
$$\hat L_r:=K(a_{t'}^\eta:  (\eta,t')\in \NN_r^\md\times \mn)$$

$$\hat L_{\ine(\xi,t)}:=K(a_{t'}^{\eta}:(\eta,t')\in \NN_r^{\md}\times \mn \text{ and }(\eta,t')\ine (\xi,t))$$ and
$$\hat L_{\ineq(\xi,t)}:=K(a_{t'}^{\eta}:(\eta,t')\in \NN_r^{\md}\times \mn \text{ and }(\eta,t')\ineq (\xi,t)).$$
\smallskip

\item We say that $(\xi,t)\in \NN_{\leq r}^\md\times \mn$ is a separable leader (inseparable leader) of $L_r$ if $a_t^\xi$ is separably algebraic (inseparably algebraic) over $\hat L_{\ine(\xi,t)}$. We say $(\xi,t)$ is a leader (of $L_r$) if it is either a separable or inseparable leader.
\smallskip

\item We say that $(\xi,t)$ is a minimal-separable leader if it is a separable leader and there is \emph{no} separable leader $(\eta,t)$ with $(\eta,t)<(\xi,t)$.
\smallskip

\item $L_r$ is said to be separable if there is \emph{no} inseparable leader $(\xi,t)$ with $|\xi|=r$.

\item For any $\ell$ in the vector space $V_F(r)$ and $t\in \mn$, we let
$$\ell(L_{r,t}):=\ell((a^\eta_t)_{\eta\in \NN_r^\md})$$
where the latter is the (unique) element of $L_r$ obtained by substituting each $w^{\eta}$ in $\ell$ with $a_t^\eta$ (recall that $\ell$ is a an $F$-linear combination of $(w^\eta)_{\eta\in \NN_r^\md})$.
\end{enumerate}
\end{definition}

\smallskip		
		\begin{remark}\label{a_xi}
		Let $L_r=K(a^\xi_t)_{(\xi,t)\in \md^{\leq r}\times \mn}$ be a  $\bDD^\Gamma$-kernel. 
			\begin{enumerate}
				\item For any $\ell\in V_F(r-1)$, any $i\in \md$ and any $t\in \mn$ we have  
				$$(\dd_i(\ell))(L_{r,t})=\dd_i(\ell(L_{r,t})).$$
				\item Consequently, $(\dd_\xi(w^\emptyset))(L_{r,t})= a^\xi_t$ for all $\xi\in \md^{\leq r}$.
				
				\end{enumerate}
		\end{remark}
		
		\smallskip

		\begin{lemma}\label{reorder}
For every  $\xi\in\md^{\leq r}$, every $\gDD$-kernel $L_r$, and every $t\in \mn$ we have 
$$
a_t^{\xi}=
\chi_{\xi}a_t^{\rho(\xi)}+\ell_\xi(L_{r-1,t}).
$$ 
Note that the latter is in $L_{|\xi|-1}$.
%where $L_{s,t}:=(a_t^\xi)_{\xi\in \md^r}$.
\end{lemma}
\begin{proof}
This is clear by the definition of the function $\ell_*:\md^{\leq r}\to V_F(r-1)$ and by Remark \ref{a_xi}(2).
\end{proof} 
		
\smallskip

\begin{corollary}\label{aboveseparable}
Let $L_r$ be a $\gDD$-kernel. 
\begin{enumerate}
\item[(i)] Then, $L_r=\hat L_r$.
\item[(ii)] If $(\xi,t)$ is a separable leader of $L_r$, then any $(\eta,t)>(\xi,t)$ is also a separable leader. In fact, $a_t^\eta\in \hat L_{\ine (\eta,t)}$.
\end{enumerate}
\end{corollary}
\begin{proof}
(i) is immediate from Lemma~\ref{reorder}.

\smallskip

(ii) It suffices to consider the case when $\eta=\rho(i,\xi)$ for some $i\in \md$ (as one can then iterate). Since $(\xi,t)$ is a separable leader,  Lemma~\ref{explicit} implies that 
$$\dd_{i}(a_t^\xi)\in K((a_s^\tau,\dd_{q}(a_s^\tau))_{(\tau,s)\ineq(\xi,t), q\in\supp(i)}, (\dd_{i}(a_s^\tau))_{(\tau,s)\ine(\xi,t)}).$$ 
But then Lemma~\ref{reorder} implies that $a_t^\eta\in \hat L_{\ine (\eta,i)}$ (noting that $\chi_{i,\xi}=1$).
\end{proof}

\medskip

Fix a $\gDD$-kernel $L_r=K(a_i^\xi: (\xi,i)\in \md^{\leq r}\times\mn)$.

\begin{definition}
Let $s\geq r$ and $E_s=K(b_i^\xi: (\xi,i)\in \md^{\leq s}\times\mn)$ be a $\gDD$-kernel (of length $s$ in $n$-variables over $(K,\be)$).
\begin{enumerate}
\item We say that $E_s$ is a prolongation of $L_r$ if $b_i^\xi=a_i^\xi$ for all $\xi\in \md^{\leq r}$.
\item We say that a prolongation $E_s$ of $L_r$ is generic if the minimal-separable leaders and inseparable leaders of $E_s$ are the same as those of $L_r$.
\item If $s=r$, we say that $E_r$ is a specialisation of $L_r$ if the tuple $(b_i^\xi:(\xi,i)\in \NN_{\leq r}^\md\times \mn)$ is an algebraic specialisation of the tuple $(a_i^\xi:(\xi,i)\in \NN_{\leq r}^\md\times \mn)$ over $K$ (this means that the algebraic vanishing ideal over $K$ of the latter tuple is contained in that of the former). In case the specialisation is generic (i.e. the ideals coincide) we say that $E_r$ and $L_r$ are isomorphic.
\end{enumerate}
\end{definition}

\begin{proposition}\label{genericspecial}
Suppose $L_s$ and $E_s$ are $\gDD$-kernels that prolong $L_r$. If $L_s$ is a generic prolongation, then $E_s$ is a specialisation of $L_s$. It follows that any two generic prolongations of length $s$ are isomorphic.
\end{proposition}
\begin{proof}
Let $L_s=K(a_{i}^{\xi}:(\xi,i)\in\md^{\leq s}\times \mn)$ and $E_s=K(b_i^{\xi}:(\xi,i)\in\md^{\leq s}\times \mn)$. Then, 
$$a_i^\xi=b_i^\xi, \quad \text{ for } |\xi|\leq r.$$
We proceed by induction with respect to $\ineq$. Let $(\tau,j)\in \NN_{\leq s}^\md\times \mn$ with $|\tau|>r$ and suppose we have an algebraic specialisation over $L_r$ 
$$\phi:K[a_{i}^{\xi}:(\xi,i)\ine (\tau,j)]\to K[b_{i}^{\xi}:(\xi,i)\ine(\tau,j)].$$
If $(\tau,j)$ is not a leader, then we can clearly extend the specialisation to $a_j^{\tau}$. Now assume $(\tau,j)$ is a leader. Since $L_s$ is generic, $(\tau,j)$ must be a separable leader but not a minimal one. Thus, there are $\xi\in \NN^\md_{<s}$ and $\ell\in \md$ such that $(\xi,j)$ is a separable leader and $\tau=\rho(\ell,\xi)$. Then, there exists a polynomial $f$ over $K$ such that
$$f((a_i^\eta)_{(\eta,i)\ineq (\xi,j)})=0$$
and 
$$\frac{\partial f}{\partial x_j^\xi}((a_i^\eta)_{(\eta,i)\ineq (\xi,j)})\neq 0$$
We now observe that we may choose $\frac{\partial f}{\partial x_j^\xi}((a_i^\eta)_{(\eta,i)\ineq (\xi,j)})$ to be an element of $L_r$. Indeed, if $|\xi|=r$ then we are done; otherwise, by Lemma~\ref{aboveseparable}(ii), $f$ can be chosen monic of degree one in $x_j^\xi$ and so $\frac{\partial f}{\partial x_j^\xi}=1$. 

By applying the specialisation $\phi$ we get
$$f((b_i^\eta)_{(\eta,i)\ineq (\xi,j)})=0$$
By Lemma~\ref{explicit} and Lemma \ref{reorder}, there is a polynomial $h$ over $K$ such that 
$$\frac{\partial f}{\partial x_j^\xi}((a_i^\eta)_{(\eta,i)\ineq (\xi,j)})\cdot \dd_{\ell}(a_j^\xi)=h((a_i^\eta)_{(\eta,i)\ine(\tau,j)})$$
and similarly 
$$\frac{\partial f}{\partial x_j^\xi}((b_i^\eta)_{(\eta,i)\ineq (\xi,j)})\cdot \dd_{\ell}(b_j^\xi)=h((b_i^\eta)_{(\eta,i)\ine(\tau,j)}).$$
Apply $\phi$ to the former and use $\frac{\partial f}{\partial x_j^\xi}((a_i^\eta)_{(\eta,i)\ineq (\xi,j)})=\frac{\partial f}{\partial x_j^\xi}((b_i^\eta)_{(\eta,i)\ineq (\xi,j)})\in L_r$ to get $\phi(\dd_{\ell}(a_j^\xi))=\dd_{\ell}(b_j^\xi)$. Then Lemma~\ref{reorder} yields $\phi(a_j^\tau)=b_j^\tau$. Hence, the specialisation $\phi$ on $L_{\ine (\tau,j)}\to E_{\ine(\tau,j)}$ is already a specialisation on $L_{\ineq (\tau,j)}\to E_{\ineq(\tau,j)}$.
\end{proof}

\

\subsection{Realisations of $\gDD$-kernels} We carry forward the notation from the previous subsection. Given a $\gDD$-field $(L,\be)$ extension of $(K,\be)$ generated (as a $\bDD$-field) by an $n$-tuple $(a_1,\dots,a_n)$, we may think of $L$ as a $\gDD$-kernel of lenght $\infty$, and sometimes write $L=L_{\infty}$. The notions of separable/inseparable leader and minimal-separable leader are defined for $\gDD$-extensions such as $L$ in the natural manner, as well as the notions of specialisations and isomorphisms. The latter notions yield $\bDD$-homomorphisms and $\bDD$-isomorphisms, and thus we will refer to them as $\bDD$-specialisation and $\bDD$-isomorphism, respectively.
\smallskip

Our first observation is that the set of minimal-separable leaders is always finite. This is a key feature that will enable us to show the existence of a model-companion in Section~\ref{modeltheory}.

\begin{lemma}\label{dick}
Let $(L,\be)$ be a $\gDD$-field extension of $(K,\be)$ which is finitely generated (as a $\bDD$-field over $K$). Then, the set of minimal-separable leaders of $L$ is finite.
\end{lemma}
\begin{proof}
Let $M$ be the set of minimal-separable leaders of $L$ and assume $M$ is infinite. By Corollary~\ref{aboveseparable}(ii), $M$ forms an antichain in $\NN_{<\infty}^{\md}\times \mn$ with respect to $\leq$. But, via the injective map $\psi:\NN^\md_{<\infty}\to \NN_0^{m_1+m_2}$ defined in Remark~\ref{notation2}, this yields an infinite antichain in $\NN_0^{m_1+m_2}$ with respect to the product order. However, this is impossible as such antichains are finite by Dickson's lemma (see for instance \cite{figueira2011}). 
\end{proof}

\begin{definition} Let $L_r=K(a_i^\xi:(\xi,i)\in \md^{\leq r}\times \mn)$ be a $\gDD$-kernel (in $n$-variables).
\begin{enumerate}
\item A regular realisation of $L_r$ is a $\gDD$-field $(L,\be)$ generated (as $\bDD$-field) by the tuple $(a_i^
\emptyset)_{i=1}^n$ from $L_r$ such that the $\bDD$-structure on $L$ extends that on $L_r$.
\item A principal realisation of $L_r$ is a regular realisation $(L,\be)$ such that the minimal-separable leaders and inseparable leaders of $L$ are the same as those of $L_r$.
\end{enumerate}
\end{definition}

We now observe that principal realisations, if they exist, are unique (up to $\bDD$-isomorphism).

\begin{lemma}\label{uniquereal}
Suppose $L_r$ is a $\gDD$-kernel. If $(E,\be)$ is a regular realisation and $(L,\be)$ is a principal realisation, then $E$ is a $\bDD$-specialisation of $L$. It follows that any two principal realisations are $\bDD$-isomorphic.
\end{lemma}
\begin{proof}
%For $\xi\in \md^{r}$ we write $\dd^\xi:=\dd_{\xi_1}\cdots\dd_{\xi_r}$ when $\xi=(\xi_1,\dots,\xi_r)$ and each $\xi_i$ is of the form $(u,j)$ for $0\leq u\leq s$ and $1\leq j\leq m_u$. 
Write $L=K(a_{i}^{\xi}:(\xi,i)\in\md^{<\omega}\times \mn)$ and $E=K(b_i^{\xi}:(\xi,i)\in\md^{<\omega}\times \mn)$. Now, for each $s>r$, the $\gDD$-kernels
$$L_s=K(\dd_\xi a_i: (\xi,i)\in \md^{\leq s}\times\mn)$$
and 
$$E_s=K(\dd_\xi b_i: (\xi,i)\in \md^{\leq s}\times\mn)$$
are prolongations of $L_r$. Furthermore, by assumption, $L_s$ is a generic prolongation. Hence, by Proposition~\ref{genericspecial}, $E_s$ is a specialisation of $L_s$. Iterating on $s$ yields the desired $\bDD$-specialisation.
\end{proof}

%{\bf for the proof of the next two lemmas we write $\rho(\tau)$ to mean $\tau$ re-shuffled in decreasing order (rather than increasing, as originally) as for this modified meaning of $\rho$ the elements $\tau=\rho(\tau)$ are the ones that are special for the constructions that are inductive with respect to $<_{\rlex}$ (as all the other ones are automatically algebraic over the previous ones by Lemma \ref{reorder}). This should not cause any mathematical issues as being a leader for this $\rho$ is the same as for the original one, Lemma \ref{reorder} works unchanged etc.)}

%The next lemma says that specialising the highest level of a $\bDD$-kernel yields another $\bDD$-kernel.

The main result of this section is Theorem~\ref{thebigone} which gives sufficient conditions for a $\gDD$-kernel to have a principal realisation. For the proof we will make use of the results obtained in \S\ref{secnotation} together with the following lemma.

\begin{lemma}\label{spec_kernel}
%Suppose $\Gamma$ is Jacobi-associative.
 Let $r\geq 1$ and let $L_{r+1}=K(a^\eta_t)_{ (\eta,t)\in \md^{\leq r+1}\times\mn}$ be a $\bDD$-kernel. 
 Suppose $(b^\eta_t)_{ (\eta,t)\in \md^{r+1}\times\mn}$ is an algebraic specialisation of $(a^\eta_t)_{ (\eta,t)\in \md^{r+1}\times\mn}$ over $L_r$. Then 
$$E_{r+1}:=K((a^\eta_t)_{ (\eta,t)\in \md^{\leq r}\times\mn},(b^\eta_t)_{ (\eta,t)\in \md^{r+1}\times\mn})$$ 
is a $\bDD$-kernel over $(K,\be)$.
\end{lemma}
\begin{proof} 
Let $u\in \{1,2\}$. We need to extend the $k$-algebra homomorphism $e_u:L_{r-1}\to \DD_u(L_r)$ to $e'_u:L_r\to \DD_u(E_{r+1})$ in such a way that 
$$e'_u(a_t^\xi)=\epsilon_{u,0}a_t^\xi+ \epsilon_{u,1}b_t^{((u,1),\xi)}+\cdots + \epsilon_{u,m_u}b_t^{((u,m_u),\xi)}$$
for all $(\xi,t)\in \md^r\times\mn$. Hence, it suffices to show that 
\begin{equation}\label{toshow}
f^{e_u}((\epsilon_{u,0}a_t^\xi+ \epsilon_{u,1}b_t^{((u,1),\xi)}+\cdots + \epsilon_{u,m_u}b_t^{((u,m_u),\xi)})_{(\xi,t)\in \md^r\times\mn})=0
\end{equation}
for $f\in L_{r-1}[(x_t^\xi)_{(\xi,t)\in \md^r\times\mn}]$ vanishing at $(a_t^\xi)_{(\xi,t)\in \md^r\times\mn}$, where recall that $f^{e_u}$ is the polynomial over $\DD_u(L_r)$ obtained by applying $e_u$ to the coefficients of $f$. 

Let $(x^\eta_t)_{ (\eta,t)\in \md^{r+1}\times\mn}$ be a set of variables over $L_r$ and consider the polynomial equation
\begin{equation}\label{exhom}
f^{e_u}((\epsilon_{u,0}a_t^\xi+ \epsilon_{u,1}x_t^{((u,1),\xi)}+\cdots + \epsilon_{u,m_u}x_t^{((u,m_u),\xi)})_{(\xi,t)\in \md^r\times\mn})=0
\end{equation}
Because $(\epsilon_{u,i})_{i=0}^{m_u}$ forms a linear basis for $\DD_u(L_r[(x^\eta_t)_{ (\eta,t)\in \md^{r+1}\times\mn}])$ over $L_r[(x^\eta_t)_{ (\eta,t)\in \md^{r+1}\times\mn}]$, we can rewrite the polynomial in \eqref{exhom} as
$$\epsilon_{u,0}f_0((x^\eta_t)_{ (\eta,t)\in \md^{r+1}\times\mn})+\epsilon_{u,1} f_{1}((x^\eta_t)_{ (\eta,t)\in \md^{r+1}\times\mn})+\cdots + \epsilon_{u,m_u}f_{m_u}((x^\eta_t)_{ (\eta,t)\in \md^{r+1}\times\mn})$$
for some $f_i\in L_r[(x^\eta_t)_{ (\eta,t)\in \md^{r+1}\times\mn}]$, and hence equation~\eqref{exhom} is equivalent to
\begin{equation}\label{reduced}
f_0((x^\eta_t)_{ (\eta,t)\in \md^{r+1}\times\mn})=0 \; \land \cdots \land f_{m_u}((x^\eta_t)_{ (\eta,t)\in \md^{r+1}\times\mn})=0
\end{equation}
Since $L_{r+1}$ is a $\bDD$-kernel, there is an extension of $e_u:L_{r-1}\to \DD_u(L_r)$ to $e_u:L_r\to \DD_u(E_{r+1})$ such that
$$e_u(a_t^\xi)=\epsilon_{u,0}a_t^\xi+ \epsilon_{u,1}a_t^{((u,1),\xi)}+\cdots + \epsilon_{u,m_u}a_t^{((u,m_u),\xi)}$$
for all $(\xi,t)\in \md^r\times\mn$. Thus, $(a^\eta_t)_{ (\eta,t)\in \md^{r+1}\times\mn}$ satisfies \eqref{reduced}, but as $(b^\eta_t)_{ (\eta,t)\in \md^{r+1}\times\mn}$ is an algebraic specialisation of $(a^\eta_t)_{ (\eta,t)\in \md^{r+1}\times\mn}$ over $L_r$, the former must also be a solution to \eqref{reduced}. It follows that \eqref{toshow} holds, as desired.
\end{proof}

%\medskip

We can now prove

\begin{theorem}\label{thebigone}
Suppose $\Gamma$ is Jacobi-associative. Let $r\geq 1$ and let $(L_{2r},\be)$ be a separable $\gDD$-kernel over $(K,\be)$. If the minimal-separable leaders of $L_{2r}$ are the same as those of $L_r$, then $L_{2r}$ has a principal realisation.
\end{theorem}
\begin{proof}
We assume we have a $\gDD$-kernel $L_s=K(a^\xi_t)_{(\xi,t)\in \md^{\leq s}\times \mn}$ which is a generic prolongation of $L_{2r}$ and show the existence of a generic prolongation $L_{s+1}$ (this clearly suffices). Fix a universal field $\Omega$; i.e., a sufficiently big algebraically closed field containing $L_s$.

Let $(X_{t}^\mu)_{(\mu,t)\in \NN_{s+1}^\md\times \mn}$ be a tuple from $\Omega$ which is algebraically independent over $L_{s}$. We will first define a (not necessarily $\Gamma$-commuting) $\bDD$-structure $\be: L_s\to \bDD(\Omega)$ extending $\be:L_{s-1}\to \bDD(L_s)$; this will be obtained by choosing elements $(a'^{(i,\tau)}_t)_{i\in \md}$ inductively on $(\tau, t)\in \NN_s^\md\times \mn$ with respect to $\ineq$.

Let $(\tau,t)\in \NN_s^\md\times\mn$ and suppose we have extended the $\bDD$-structure $\be:L_{s-1}\to \bDD(L_s)$ to $\be: L_{\ine (\tau,t)}\to \bDD(\Omega)$. 
We consider two cases:
\medskip

\noindent {\bf Case 1.} Suppose $(\tau,t)$ is a leader. Since $|\tau|\geq 2r>r$ and $L_s$ is a generic prolongation of $L_r$, $(\tau,t)$ is a separable leader. By Lemma~\ref{extendstructure}(i), there is a unique $\bDD$-structure  $L_{\ineq(\tau,j)}\to \bDD(\Omega)$ extending  $L_{\ine(\tau,t)}\to \bDD(\Omega)$. Hence we can put $a'^{(i,\tau)}_t:=\dd_i(a^{\tau}_t)$ for each $i\in \md$.

\medskip

\noindent {\bf Case 2.} Suppose $(\tau,t)$ is not a leader. By Lemma~\ref{extendstructure}(ii), we can choose $(a'^{(i,\tau_0)}_t)_{i\in \md}$ arbitrarily; so we may set
$$a'^{(i,\tau)}_t:= \chi_{i,\tau} X^{\rho(i,\tau)}_t+\ell_{i,\tau}(L_{s,t}).$$

\medskip

This construction yields a $\bDD$-structure $\be:L_s\to \bDD(\Omega)$ (here we are using the fact that $L_s=\hat L_s$, see Lemma~\ref{aboveseparable}), which in turn yields a $\bDD$-kernel $L'_{s+1}$. However, this might not be a $\gDD$-kernel (i.e., the operators need not $\Gamma$-commute on $L_{s-1}$). We fix this now. 
%Write $({a'}_t^\xi)_{(\xi,t)\in \md^{ s+1}\times\mn}$ for the last level of the kernel $L'_{s+1}$.

\medskip

We will prove by induction on $(\mu,t)\in \NN_{s+1}^\md\times \mn$ with respect to $\ineq$ that there is a specialisation $({a}_t^\xi)_{(\xi,j)\in \md^{ s+1}\times\mn}$ of the tuple $({a'}_t^\xi)_{(\xi,t)\in \md^{ s+1}\times\mn}$ over $L_s$ such that, for the $\bDD$-kernel $L_s(a_t^\xi)_{(\xi,t)\in \md^{s+1}\times \mn}$ (note this is indeed a $\bDD$-kernel by Lemma~\ref{spec_kernel}), the following conditions hold:
\begin{enumerate}
\item\label{1}  
For all $(\tau,t')\in \NN_{s-1}^\md\times \mn$ and $i,j\in \md$ with $(\rho(i,j,\tau),t')\trianglelefteq (\mu,t)$ we have \[\dd_i\dd_j(a^\tau_{t'})=\chi_{ij}\dd_j\dd_i(a^\tau_{t'})+\sum_\ell c^{ij}_\ell \dd_\ell(a^\tau_{t'}).\]

%for all $(\tau,t')$, $1\leq p,q\leq m_1$ and $1\leq  p',q'\leq m_2$ with $(\rho(p,q,\tau),t')\ineq (\mu, t)$ we have $$\dd_{2,p'}\dd_{2,q'}(a_{t'}^\tau)=c_{2,1}^{p'q'}\dd_{2,1}(a_{t'}^\tau)+\cdots +c_{2,m_2}^{p'q'}\dd_{2,m_2}(a_{t'}^\tau),$$ $$[\dd_{1,p},\dd_{1,q}](a_{t'}^\tau)=c_{1,1}^{pq}\dd_{1,1}(a_{t'}^\tau)+\cdots+c_{1,m_1}^{pq}\dd_{1,m_1}(a_{t'}^\tau),$$ and $$[\dd_{1,p},\dd_{2,q'}](a_{t'}^\tau)=0.$$

\item\label{2} Let $(\nu,t')\in\NN_{s+1}^\md\times\mn$ with $(\nu,t')\ineq (\mu,t)$. If there are no $\eta\in \NN_s^\md$ and $i\in \md$ such that $(\eta,t')$ is a leader and $\rho(i,\eta)=\nu$, then $a_{t'}^\nu = X_{t'}^\nu$; otherwise (when there are such $\eta$ and $i$) we have $a_{t'}^\nu\in K(a_{t''}^\tau)_{(\tau,t'')\ine (\nu,t')}$.
 
 %\item\label{2} If $(\mu',t')\in \NN_{s+1}^\md\times \mn$  with $(\mu',t')\trianglelefteq (\mu,t)$, and there are no $\eta\in \NN_s^\md$ and $i\in \md$ such that $(\eta,t')$ is a leader and $\rho(i,\eta)=\mu'$, then $a_{t'}^{\mu'}=X^{\mu'}_{t'}$.

%\item\label{3} If $(\nu,t')\in \md^{s+1}\times \mn$ and  $(\rho(\nu),t')\triangleright (\mu,t)$ then  $a^{\nu}_{t'}=a'^{\nu}_{t'}$.
 
% \item\label{4} For all $\tau\in \NN_0^\md(s+1)$ and  ${t'}\in \mn$ we have $a^{\tau}_{t'}\in L_s(X_{\mu',t''})_{(\rho(\mu'),t'')\trianglelefteq (\rho(\tau),{t'})}$.
 \end{enumerate}
 
 \medskip
 
 This will clearly be enough, as then for $(\mu,t):=\max(\NN_{s+1}^\md\times \mn, \trianglelefteq)$ we obtain the desired $\gDD$-kernel: condition (1) yields $\Gamma$-commutativity  while condition (2) guarantees that we introduce no new minimal-separable leaders nor inseparable leaders.
 
 \medskip
 
 Suppose that $(\mu,t)\in \NN_{s+1}^\md\times \mn$ and that we have found a specialisation $({b}_t^\xi)_{(\xi,t)\in \md^{ s+1}\times\mn}$ of $({a'}_t^\xi)_{(\xi,t)\in \md^{ s+1}\times\mn}$ over $L_s$ such that the above conditions hold at the $\trianglelefteq$-predecessor of $(\mu,t)$. In order to perform the desired specialisation at step $(\mu,t)$ we will need the following claim -- whose proof is lengthy and quite technical but is arguably the main ingredient to prove the theorem --. In the claim (and its subclaims) all computations take place in the $\bDD$-kernel $L_s(b_t^\xi)_{(\xi,t)\in \md^{s+1}\times \mn}$.
  
 \begin{claim}\label{claim_well_def} Suppose that $\rho(i,\eta)=\mu=\rho(j,\eta')$ for some $i, j \in \md$ and some $\eta,\eta'\in \NN_s^\md$ such that $(\eta,t)$ and $(\eta', t)$ are (necessarily separable) leaders. Then 
 $$\dd_{i}(a^{\eta}_t)-\ell_{i,\eta}(L_{s,t})=\dd_{j}(a^{\eta'}_t)-\ell_{j,\eta'}(L_{s,t}).$$
 \end{claim}
 
 \begin{proof} We may assume that $i\neq j$ (when they are equal the desired equality is obvious). Let $\tau\in \NN_{s-1}^\md$ be such that $\rho(i,j,\tau)=\mu$; which implies $\eta=\rho(j,\tau)$ and $\eta'=\rho(i,\tau)$. We consider three cases; Cases A, B, and C.
 
 \smallskip
 
{\bf Case A}. Assume $i$ and $j$ are both of Lie-type.
  
 \begin{Sclaim}\label{ETS}
To prove Claim~\ref{claim_well_def}, it suffices to prove that 
$$(*)\hspace{2cm} \dd_i\dd_j (a^\tau_{t}) - \dd_j\dd_i(a^\tau_{t}) =\sum_\ell c_\ell^{ij}\dd_{\ell}(a^\tau_{t}).\hspace{2cm}$$
\end{Sclaim}
\begin{proof}
Suppose the (*) holds, so 
\begin{align*}
\sum_\ell c_\ell^{ij}\dd_{\ell}(a^\tau_{t}) & =\dd_{i}\dd_{j} (a^\tau_{t}) - \dd_{j}\dd_{i}(a^\tau_{t}) \\
& =\dd_{i}(a^{\rho(j,\tau)}_{t}+\ell_{j,\tau}(L_{s,t}))-\dd_{j}(a^{\rho(i,\tau)}_{t}+\ell_{i,\tau}(L_{s,t})) \\
& = \dd_{i} (a^\eta_{t})+\dd_{i}(\ell_{j,\tau})(L_{s,t})-\dd_{j}(a^{\eta'}_{t})-\dd_{j}(\ell_{i,\tau}(L_{s,t})) \\
& = \dd_{i}(a^{\eta}_t)-\ell_{i,\eta}(L_{s,t})-\dd_{j}(a_t^{\eta'})+\ell_{j,\eta'}(L_{s,t})+\sum_\ell c_\ell^{ij}\dd_{\ell}(a^\tau_t)
\end{align*}
where the last equality follows by Lemma \ref{horrible}(1). Thus 
\[ \dd_{i}(a^{\eta}_t)-\ell_{i,\eta}(L_{s,t})=\dd_{j}(a^{\eta'}_t)-\ell_{j,\eta'}(L_{s,t}),\] as required.
\end{proof}
 
We now aim to prove that (*) holds. 

\smallskip

As both $(\eta,t)=(\rho(j,\tau),t)$ and $(\eta',t)=(\rho(i,\tau),t)$ are separable leaders, by the assumption on $L_{2r}$ there are minimal-separable leaders $(\xi'_1,t)\leq (\rho(j,\tau),t)$ and $(\xi'_2,t)\leq (\rho(i,\tau),t)$ with $|\xi'_1|,|\xi'_2|\leq r$. 
Then, letting  $\xi'_1\vee\xi'_2$  
be the least upper bound of $\xi'_1$ and $\xi'_2$ in $\NN^\md_{<\infty}$ with respect to $\leq$, we get $|\xi'_1\vee\xi'_2|\leq 2r$ and so $\xi'_1\vee\xi'_2< \rho(i,j,\tau)=\mu$. Hence there is some $k\in \md $ such that $\rho(k,\xi'_1\vee\xi'_2)\leq \rho(i,j,\tau)$.
%Suppose first that $k\neq i,j$. 
%+++
Then, choosing $\xi_1,\xi_2\in \NN_{s-1}^\md$  so that $\rho(j,k,\xi_1)=\rho(i,j,\tau)=\rho(i,k,\xi_2)$, we have that $\xi'_1\leq \xi_1$ and $\xi'_2\leq \xi_2$, so $(\xi_1,t)$ and $(\xi_2,t)$ are separable leaders (by Lemma~\ref{aboveseparable}(2)).  Let $\xi\in\NN^\md_{s-2}$ be such that $\rho(i,\xi)=\xi_1$; hence also $\rho(j,\xi)=\xi_2$ and $\rho(k,\xi)=\tau.$

%As $\xi_1\geq \xi'_1$ and $\xi_2\geq \xi'_2$, both $\xi_1$ and $\xi_2$ are separable leaders. 

%We will reuse some calculations from the proof of Proposition \ref{just1}, hence, to match the notations, below we will write $i,j,k$ for $(1,p)$, $(1,q)$ and $(u,k)$, respectively.
%change j in a_j to j_0

\begin{Sclaim}\label{6swaps} For $\xi\in\NN^\md_{s-2}$ and $k\in\md$ as chosen above, we have
\begin{align*}
\dd_i\dd_j\dd_k (a^\xi_t) & =\dd_j\dd_i\dd_k(a^\xi_t)+ \sum_\ell \dd_j(c_\ell^{ki}\dd_\ell(a^\xi_t)) \\
& + \sum_\ell c_\ell^{kj}\dd_\ell\dd_i(a^\xi_t) + \sum_\ell \dd_k(c_\ell^{ij}\dd_\ell(a^\xi_t))+ \sum_\ell c_\ell^{ik}\dd_\ell\dd_j(a^\xi_t)+\sum_\ell \dd_i(c_\ell^{jk} \dd_\ell(a^\xi_t))
\end{align*}
\end{Sclaim}
\begin{proof}
The calculation follows the same lines as the one in the proof of Proposition \ref{just1}. However, that computation uses $\Gamma$-commutativity; hence, to repeat the calculation, we have to check (using our induction hypothesis) that we can apply $\Gamma$-commutativity at each step of that computation. To do this, we need:

\begin{SSclaim}\label{partialop} For $\xi\in\NN^\md_{s-2}$ and $k\in\md$ as chosen above, we have
$$\dd_i\dd_k\dd_j(a^\xi_t)=\dd_k\dd_i\dd_j(a^\xi_t) +\sum_\ell c_\ell^{ik}\dd_\ell\dd_j(a^\xi_t)$$
\end{SSclaim}
\begin{proof}
	Assume first that $k$ is of Lie type. 
Consider the local algebra  $ \DD_1\otimes \DD_1$ with basis $(\epsilon_{i',j'})_{0\leq i',j'\leq m_1}$  where $\epsilon_{i',j'}:=\epsilon_{i'}\otimes \epsilon_{j'}$. Note that, for any $1\leq i',j'\leq m_1$, we have $\supp((i',j'))=\supp(i')\times \supp(j')$ where the former support is computed in $\DD_1\otimes \DD_1$ while the latter supports are computed in $\DD_1$. Consider the maps 
$$f:=\DD_1(e_1)\circ e_1:L_s\to \DD_1\otimes \DD_1 (\Omega)$$ 
and 
$$f':= r_1^{e_1}\circ e_1:L_s\to \DD_1\otimes \DD_1 (\Omega)$$ 
with co-ordinate functions $(D_{i',j'})_{1\leq i',j'\leq m_1}$ and $(D'_{i',j'})_{1\leq i',j'\leq m_1}$, respectively, where $e_1$ is the homomorphism $L_s\to \DD_1(\Omega)$ induced by the kernel $L_s(b^\nu_{t'})_{(\nu,t')\in\md^{s+1}\times \mn}$. 
Then  $D_{i',j'}=\dd_{i'}\dd_{j'}$ and $D'_{i',j'}=\dd_{j'}\dd_{i'}+\sum_\ell c_\ell^{i'j'}\dd_\ell$. % for all $1\leq i,j\leq m_1$. 
Note that if $(i',k')\in \supp((i,k))=\supp(i)\times \supp(k)$ and $(\nu,z)\trianglelefteq (\rho(j,\xi),t)$, or $(i',k')=(i,k)$ and  $(\nu,z)\triangleleft (\rho(j,\xi),t)$, then $(\rho(i',k',\nu),z)\triangleleft (\mu,t)$, so, by the inductive assumption (1), we have   $D_{i',k'}(a^\nu_z)=D'_{i',k'}(a^\nu_z)$. Hence, as $(\rho(j,\xi),t)=(\xi_2,t)$ is a separable leader, %,,, double check it's really \xi_1 
  by Lemma \ref{partial_com} we get \[\dd_i\dd_k(a^{\rho(j,\xi)}_t)=D_{i,k}(a_t^{\rho(j,\xi)})=D'_{i,k}(a_t^{\rho(j,\xi)})=\dd_k\dd_i(a_t^{\rho(j,\xi)})+\sum_\ell c_\ell^{ik}\dd_\ell (a_t^{\rho(j,\xi)}).\]

By $\Gamma$-commutativity of the kernel $L_s$ we have $\dd_j(a^\xi_t)=a_t^{\rho(j,\xi)}+\ell_{j,\xi}(L_{s-2,t})$. As the latter term is in $L_{s-2}$, it also  follows that 
$$\dd_i\dd_k(\ell_{j,\xi}(L_{s-2,t}))=\dd_k\dd_i(\ell_{j,\xi}(L_{s-2,t}))+\sum_\ell c_\ell^{ik}\dd_\ell( \ell_{j,\xi}(L_{s-2,t}))$$
and so the conclusion of the subsubclaim follows by additivity.

\medskip

The case when $k$ is of HS-type can be dealt with similarly by applying Lemma~\ref{partial_com} to the homomorphisms $$\DD_1(e_2)\circ e_1:L_s\to \DD_1\otimes \DD_2 (\Omega)$$ 
and 
$$\chi\circ \DD_2(e_1)\circ e_2:L_s\to \DD_1\otimes \DD_2 (\Omega)$$ 
where $\chi:\DD_2\otimes \DD_1 (\Omega)\to \DD_1\otimes \DD_2 (\Omega)$ is the natural isomorphism. These yield co-ordinate functions $D_{(i',j')}=\dd_{i'}\dd_{j'}$ and $D'_{(i',j')}=\dd_{j'}\dd_{i'}$.
 
%instead of applying Lemma \ref{partial_com}, we directly use Remarks~\ref{operatornotation} and \ref{explicit}  to deduce that if $\dd_{1,i'}$ and $\dd_{2,j'}$ commute on some set $A\subseteq L_s$ for all $i'\in \supp (p)$ and $j'\in \supp( q)$, then $\dd_{1,p}$ and $\dd_{2,q}$ commute on the separable closure of the field generated by $A$.

This concludes the proof of the subsubclaim.
\end{proof}

%First, as $a^\xi_j\in L_{s-2}$ and $\dd_j(a^\xi_j)=a_j^{\rho(j,\xi)}+\ell_{(j,\xi)}(\tilde a_j^{\rho(j,\xi)})=a_j^{\xi_2}+\ell_{(j,\xi_2)}(\tilde a_j^{\xi_2})$ is algebraic over $L_{s-2}$, by $\Gamma$-commutativity of $L_s$ and Lemma~\ref{partial_com} we get

To finish the proof of Subclaim~\ref{6swaps}, note that as $a^\xi_t\in L_{s-2}$, we can apply $\Gamma$-commutativity at $a^\xi_t$; hence applying Subsubclaim~\ref{partialop} we obtain
\begin{align*}
 \dd_i\dd_j\dd_k (a^\xi_t)& =\dd_i\dd_k\dd_j(a^\xi_t)+\sum_\ell \dd_i(c_\ell^{jk} \dd_\ell(a^\xi_t)) \\
& =\dd_k\dd_i\dd_j(a^\xi_t) +\sum_\ell c_\ell^{ik}\dd_\ell\dd_j(a^\xi_t)+\sum_\ell \dd_i(c_\ell^{jk} \dd_\ell(a^\xi_t))%\hspace{2cm} (\dagger)
\end{align*}

Similarly, using $\Gamma$-commutativity at $a^\xi_t\in L_{s-2}$ and  applying Subsubclaim~\ref{partialop} with the roles of $i$ and $j$ interchanged, we get that the last expression equals
%$$\dd_i(a^\xi_t)=a_t^{\rho(i,\xi)}+\ell_{(i,\xi)}({L_{s-2,t}})=a_t^{\xi_1}+\ell_{(i,\xi)}(L_{s-2,t})$$ and applying Lemma \ref{partial_com} again, 
%we can write ($\dagger$) above as 
\[\dd_k\dd_j\dd_i(a^\xi_t) +\sum_\ell \dd_k(c_\ell^{ij}\dd_\ell(a^\xi_t))+ \sum_\ell c_\ell^{ik}\dd_\ell\dd_j(a^\xi_t)+\sum_\ell \dd_i(c_\ell^{jk} \dd_\ell(a^\xi_t))\]
\[ = \dd_j\dd_k\dd_i(a^\xi_t) +\sum_\ell c_\ell^{kj}\dd_\ell\dd_i(a^\xi_t) + \sum_\ell \dd_k(c_\ell^{ij}\dd_\ell(a^\xi_t))+ \sum_\ell c_\ell^{ik}\dd_\ell\dd_j(a^\xi_t)+\sum_\ell \dd_i(c_\ell^{jk} \dd_\ell(a^\xi_t))\] 
%and, using once more $\Gamma$-commutativity at $a^\xi_t\in L_{s-2}$, we obtain that the last term equals
\begin{align*}
=\dd_j\dd_i\dd_k(a^\xi_t)+ \sum_\ell \dd_j(c_\ell^{ki}\dd_\ell(a^\xi_t)) & +\sum_\ell c_\ell^{kj}\dd_\ell\dd_i(a^\xi_t)  \\
& + \sum_\ell \dd_k(c_\ell^{ij}\dd_\ell(a^\xi_t))+ \sum_\ell c_\ell^{ik}\dd_\ell\dd_j(a^\xi_t)+\sum_\ell \dd_i(c_\ell^{jk} \dd_\ell(a^\xi_t))
\end{align*}
This proves Subclaim~\ref{6swaps}.
\end{proof}

By Subclaim \ref{6swaps} and using $c_\ell^{ij}=-c_\ell^{ji}$ (since $i,j$ are of Lie-type), we get 
\[\dd_i\dd_j\dd_k (a^\xi_t) - \dd_j\dd_i\dd_k (a^\xi_t) -\sum_\ell c_\ell^{ij}\dd_\ell\dd_k(a^\xi_t) \] 
\[=\sum_\ell c_\ell^{ji}\dd_\ell\dd_k(a^\xi_t)+\sum_\ell \dd_j(c_\ell^{ki}\dd_\ell(a^\xi_t))+\sum_\ell c_\ell^{kj}\dd_\ell\dd_i(a^\xi_t)+\]   \[+\sum_\ell \dd_k(c_\ell^{ij}\dd_\ell(a^\xi_t))+ \sum_\ell c_\ell^{ik}\dd_\ell\dd_j(a^\xi_t)+\sum_\ell \dd_i(c_\ell^{jk} \dd_\ell(a^\xi_t))\]

Using that $\Gamma$ is Jacobi-associative, we can deduce as in $(\dagger)$, inside the proof of Lemma~\ref{horrible}, that the following holds
%But, as  (in particular $r_1$ is Jacobi), by the (reversed) proof of Proposition \ref{just1} we get that 
\[\sum_\ell c_\ell^{ji}\dd_\ell\dd_k+\sum_\ell \dd_j(c_\ell^{ki}\dd_\ell)+\sum_\ell c_\ell^{kj}\dd_\ell\dd_i+\sum_\ell \dd_k(c_\ell^{ij}\dd_\ell)+ \sum_\ell c_\ell^{ik}\dd_\ell\dd_j+\sum_\ell \dd_i(c_\ell^{jk} \dd_\ell)=0\] 
on $L_{s-2}$, so in particular, evaluated at $a_t^\xi$. Hence 
$$\dd_i\dd_j\dd_k (a^\xi_t) - \dd_j\dd_i\dd_k (a^\xi_t) -\sum_\ell c_\ell^{ij}\dd_\ell\dd_k(a^\xi_t)=0$$
 Since $\dd_k(a^\xi_t)=a^\tau_t+\ell_{k,\xi}(L_{s-2,t})$, using additivity and $\Gamma$-commutativity at $\ell_{k,\xi}(L_{s-2,t})$, the above equality implies (*). Hence we have proven Case A of Claim~\ref{claim_well_def}.

\

\noindent {\bf Case B.} Assume $i$ and $j$ are of different type. As in Subclaim \ref{ETS}, one easily concludes using Lemma \ref{horrible} that it is enough to show that
$$[\dd_i,\dd_j](a_t^\tau)=0.$$
The proof of this equality is very similar to the proof in Case A, so we omit the details.
% but instead of applying Lemma \ref{partial_com}, we directly use Remarks~\ref{operatornotation} and \ref{explicit}  to deduce that if $\dd_{1,i'}$ and $\dd_{2,j'}$ commute on some set $A\subseteq L_s$ for all $i'\in \supp (p)$ and $j'\in \supp( q)$, then $\dd_{1,p}$ and $\dd_{2,q}$ commute on the separable closure of the field generated by $A$.

% Now let $\Phi:L_s((X_{\mu,t})_{(\mu,t)\in \NN^\md(s+1)\times \mn})$+++

%We may assume $u=1$ and $v=2$.
%Let $(w,k)\in \mm$ and $\tau'\in \md^{s-2}$ be such that $\tau=((w,k),\tau')$. If $w=1$, then, using Case A in the third equality below, we have \[ \dd_{1,p}\dd_{2,q} (a_j^\tau)=\dd_{1,p}\dd_{2,q}\dd_{1,k}(a_j^\tau)=\dd_{1,p}\dd_{1,k}\dd_{2,q}(a_j^{\tau'})=\dd_{1,k}\dd_{1,p}(\dd_{2,q}(a_j^{\tau'})+ \sum_{\ell} c^{pk}_\ell(\dd_{2,q}(a_j^{\tau'})))=\]
%\[= \]

\

\noindent {\bf Case C.} Assume $i$ and $j$ are of HS-type. Note that in this case, as $\rho(i,j,\tau)=\mu$, we must have $\chi_{i,\tau}=\chi_{j,\tau}=1$ (i.e., $\tau$ has no entry of HS-type).
%Again by Lemma \ref{horrible}, 
\begin{Sclaim}
In this case, to prove Claim~\ref{claim_well_def}, it is enough to prove
$$(+)\hspace{2cm} \dd_i\dd_j(a_t^\tau)=\sum_\ell c^{ij}_\ell\dd_\ell(a_t^\tau).\hspace{3cm}$$
\end{Sclaim}
\begin{proof}
Assume 	$(+)$. Using Lemma \ref{horrible}(1) in the fourth equality below 
%(and $\Gamma$-commutativity of the kernel $L_s$ in other equalities) 
we get 
	 \begin{align*}
	 \sum_\ell c^{ij}_\ell\dd_\ell(a_t^\tau) & =\dd_i\dd_j(a_t^\tau) \\
	 & =\dd_i(a^{\rho(j,\tau)}_t+\ell_{j,\tau}(L_{s,t})) \\
	 & =\dd_i(a^\eta_t)+(\dd_i(\ell_{j,\tau}))(L_{s,t})\\
	 & =\dd_i(a^\eta_t)+	\sum_\ell c^{ij}_\ell\dd_\ell(a_t^\tau)-\ell_{i,\eta}(L_{s,t}).
	 \end{align*}
	 So $\dd_i(a^\eta_t)-\ell_{i,\eta}(L_{s,t})=0$ and similarly we get $\dd_j(a^{\eta'}_t)-\ell_{j,\eta'}(L_{s,t})=0$ and hence $\dd_i(a^\eta_t)-\ell_{i,\eta}(L_{s,t})=\dd_j(a^{\eta'}_t)-\ell_{j,\eta'}(L_{s,t})$, as required.
	 
	\end{proof}
	
	We now prove that ($+$) holds.
	\smallskip
	
	By the assumption on $L_{2r}$, $(\eta,t)$ is not a minimal-separable leader. Hence, there are some $\nu\in \NN_{s-1}^\md$ and $k\in \md$ such that $\eta=\rho(k,\nu)$ and $(\nu,t)$ is a leader. Note that $k\neq j$, as otherwise both $k$ and $i$ would be of HS-type, so, as, $\mu(\rho(i,\eta))$ and $\eta=\rho(k,\nu)$, we would get $\mu=\emptyset$, which cannot happen as $|\mu|=s+1>0$. 
	
	\smallskip
	 
	 Since $i$ and $j$ are of HS-type, $k$ must be of Lie type.
	Let $\xi\in \NN_{s-2}^\md$ be such that $\mu=\rho(i,j,k,\xi)$, so $\nu=\rho(j,\xi)$ and $\tau=\rho(k,\xi)$.
	Then, using Case~B in the second equality below and that $\dd_p(c_\ell^{ij})=0$ whenever $p$ of Lie-type (due to Jacobi-associativity of $\Gamma$) in the fourth one, we get
	\begin{align*}
\dd_i\dd_j\dd_{k}(a^\xi_t) & =\dd_i\dd_k\dd_{j} (a_t^\xi) \\
 & =\dd_{k}\dd_{i}\dd_{j} (a_t^\xi) \\
 & =\dd_{k}( \sum_{\ell} c^{ij}_\ell \dd_\ell(a_t^\xi))\\
 & = \sum_{\ell} c^{ij}_\ell \dd_{k}( \dd_\ell(a_t^\xi)) \\
 & =\sum_{\ell} c^{ij}_\ell  \dd_\ell  \dd_{k}(a_t^\xi)
% &=\sum_{\ell} c^{ij}_\ell  \dd_\ell  (a_t^\tau)
\end{align*}
Since $\dd_k(a^\xi_t)=a^\tau_t+\ell_{k,\xi}(L_{s-2,t})$, once again using additivity, the above equality implies (+), as required.

%\smallskip
	
%	Now assume $k=j$, which implies $\nu=\tau$ and so $(\tau,t)$ is a leader. One can carry out the same argument as in Subsubclaim~\ref{partialop} but now in the (local) algebra $\DD_2\otimes \DD_2$ with morphisms
 %$$\DD_2(e_2)\circ e_2:L_s\to \DD_2\otimes \DD_2 (\Omega)$$ 
%and 
%$$r_2^\iota\circ e_2:L_s\to \DD_2\otimes \DD_2 (\Omega)$$ 
%Note that these yield co-ordinate functions $D_{(i',j')}=\dd_{i'}\dd_{j'}$ and $D'_{(i',j')}=\sum_\ell c^{i'j'}_\ell\dd_\ell$. One can then conclude, as $(\tau,t)$ is a separable leader, that Lemma~\ref{partial_com} yields
%	$$\dd_i\dd_j(a_t^\tau)=\sum_\ell c^{ij}_\ell\dd_\ell(a_t^\tau)$$
%	as desired.
	
%	\ 

 This completes the proof of Claim \ref{claim_well_def}.
 \end{proof}
 
 We now resume where we left off in the proof of the theorem (see paragraph above Claim~\ref{claim_well_def}). Recall that $(b_t^\xi)_{(\xi,t)\in \md^{s+1}\times \mn}$ denotes the specialisation of $(a'^\xi_t)_{(\xi,t)\in \md^{s+1}\times \mn}$ such that conditions (1) and (2) hold for the $\ineq$-predecessor of $(\mu,t)$.
 
 We now produce the specialisation $(a_t^\xi)_{(\xi,t)\in \md^{s+1}\times \mn}$ that will yield conditions (1) and (2) for $(\mu,t)$. On the one hand, if there are no $i\in \md$ and $\tau\in \NN_s^\md$ such that $(\tau,t)$ is a leader and $\mu=\rho(i,\tau)$, then put $a_t^\xi:=b_t^\xi$ for all $(\xi,t)\in \md^{s+1}\times\mn$; on the other hand, if there are such $i\in \md$ and $\tau\in \NN_s^\md$, consider the $L_s$-algebra homomorphism 
$$\Psi:L_s[(X_{z}^\nu)_{(\nu,z)\in \NN_{s+1}^\md\times \mn}]\to L_s[(X_{z}^\nu)_{(\nu,z)\in \NN_{s+1}^\md\times \mn}]$$ given by $X_{z}^\nu\mapsto X_{z}^\nu$ for $(\nu,z)\neq (\mu,t)$, and $X_{t}^\mu\mapsto \dd_i(a^\tau_t)-\ell_{i,\tau}(L_{s,t})$. Note that by Claim \ref{claim_well_def} this does {\bf not} depend on the choice of $i$ and $\tau$. In this case we put $a_t^\xi:=\Psi(b_t^\xi)$ for all $(\xi,t)\in \md^{s+1}\times\mn$.

\smallskip
 
Condition (2) follows easily by construction. %and the inductive assumption. 
We verify condition (1) at $(\mu,t)$. To do so, we will need the following (expected) identity.
 
 \begin{claim}\label{cl2}
 	For all $i\in \md$ and $\tau\in \NN_s^\md$ such that $\chi_{i,\tau}\rho(i,\tau)=\chi_{i,\tau}\mu$ we have 
	$$a^{(i,\tau)}_t=\chi_{i,\tau}a^{\rho(i,\tau)}_t+\ell_{i,\tau}(L_{s,t}).$$
 \end{claim}
 \begin{proof}
 Assume first that $\chi_{i,\tau}=0$. If $(\tau,t)$ is not a leader, then we land in Case 2 of the kernel construction, which gives $a'^{(i,\tau)}_t=\ell_{i,\tau}(L_{s,t})\in L_s$. As the specialisation performed above preserves $L_s$, we also have $a^{(i,\tau)}_t=\ell_{i,\tau}(L_{s,t})$, as required. On the other hand, if $(\tau,t)$ is a leader, then Case 1 of the kernel construction yields $a'^{(i,\tau)}_t=\dd_i(a_t^\tau)$, hence $a^{(i,\tau)}_t=\dd_i(a^\tau_t)$. Let $j\in \md$ be the unique entry of $\tau$ of HS-type, then $\tau=\rho(j,\nu)$ with $\nu\in \NN_{s-1}^\md$ having all entries of Lie-type. Note that then $\ell_{j,\nu}=0$. It follows that $\dd_j(a_t^\nu)=a_t^{\rho(j,\nu)}=a_t^\tau$; and so, by (+) in Case C of Claim~\ref{claim_well_def}, we have 
	$$\dd_i(a_t^\tau)=\dd_i\dd_j(a_t^\nu)=\sum_\ell c^{ij}_\ell\dd_\ell(a_t^\nu)$$
Now, by Lemma \ref{horrible}(1), using $\ell_{j,\nu}=0$ we have 
$$\sum_\ell c^{ij}_\ell\dd_\ell\dd_\nu(a_t^\emptyset)=\ell_{i,\rho(j,\nu)}(L_{s,t})=\ell_{i,\tau}(L_{s,t})$$
putting all this together yields $a^{(i,\tau)}_t=\ell_{i,\tau}(L_{s,t})$, as required.

 \

 We may now assume that $\chi_{i,\tau}=1$. This implies that $\rho(i,\tau)=\mu$. Let $k\in\md$ and $\eta\in \NN^\md_{s}$ such that $(k,\eta)=\rho(i,\tau)=\mu$. Note that then $\ell_{k,\eta}=0$. We must show that 
 $$a^{(i,\tau)}_t=a^{(k,\eta)}_t+\ell_{i,\tau}(L_{s,t}).$$
 We consider cases.
 
 \medskip
 
 \noindent {\bf Case (i).} Suppose $(\tau,t)$ and $(\eta,t)$ are leaders. Since $\rho(i,\tau)=\mu=(k,\eta)=\rho(k,\eta)$, by Claim~\ref{claim_well_def} we have
 $$\dd_i(a_t^\tau)-\ell_{i,\tau}(L_{s,t})=\dd_k(a_t^\eta)-\ell_{k,\eta}(L_{s,t})=\dd_k(a_t^\eta)$$
It then follows, using Case 1 of the kernel construction, that
 $$a_t^{(i,\tau)}=\dd_i(a_t^\tau)=\dd_k(a_t^\eta)+\ell_{i,\tau}(L_{s,t})=a_t^{(k,\eta)}+\ell_{i,\tau}(L_{s,t})$$
 as desired.
 
 \medskip
 
 \noindent {\bf Case (ii).} Suppose exactly one of $(\tau,t)$ and $(\eta,t)$ is a leader. Assume $(\tau,t)$ is a leader; the other case can be treated similarly. In this case, by Case 2 of the construction of the kernel, $a'^{(k,\eta)}_t=X_t^{(k,\eta)}$. Furthermore, since $(\tau,t)$ is a leader, the specialisation maps $X_t^{(k,\eta)}\mapsto \dd_i(a_t^\tau)-\ell_{i,\tau}(L_{s,t})$, and the latter equals $a_t^{(k,\eta)}$. It then follows that $a_t^{(i,\tau)}=a_t^{(k,\eta)}+\ell_{i,\tau}(L_{s,t})$.
 
 \medskip

 \noindent {\bf Case (iii).} Suppose $(\tau,t)$ and $(\eta,t)$ are not leaders. By Case 2 of the kernel construction $a_t'^{(i,\tau)}=X_t^{(k,\eta)}+\ell_{i,\tau}(L_{s,t})=a'^{(k,\eta)}_t+\ell_{i,\tau}(L_{s,t})$. Specialising then yields $a_t^{(i,\tau)}=a^{(k,\eta)}_t+\ell_{i,\tau}(L_{s,t})$.
 
 \smallskip
 
 This finishes the proof of Claim~\ref{cl2}.
 \end{proof}

We now check that condition (1) holds at $(\mu,t)$. Fix any $i,j\in \md$, $\lambda \in \NN_0^\md(s-1)$ with $\chi_{i,j,\lambda}\rho(i,j,\lambda)=\chi_{i,j,\lambda}\mu$. Using $\Gamma$-commutativity of $L_s$ in the second equality below, and Claim~\ref{cl2} in the third one, we get
\begin{align*}
\dd_i\dd_j(a^\lambda_t) & =\dd_ia^{(j,\lambda)}_t \\
& =\chi_{j,\lambda}\dd_ia^{\rho(j,\lambda)}_t+\dd_i\ell_{j,\lambda}(L_{s,t}) \\
& =\chi_{i,j,\lambda}a^{\rho(i,j,\lambda)}_t+\chi_{j,\lambda}\ell_{i,\rho(j,\lambda)}(L_{s,t})+\dd_i\ell_{j,\lambda}(L_{s,t})
\end{align*}
Similarly, we get
\begin{align*}
\chi_{ij}\dd_j\dd_ia^\lambda_t & =\chi_{ij}\dd_ja^{(i,\lambda)}_t \\
& =\chi_{ij}\chi_{i,\lambda}\dd_j a^{\rho(i,\lambda)}_t+\chi_{ij}\dd_j\ell_{i,\lambda}(L_{s,t}) \\
& =\chi_{i,j,\lambda}a^{\rho(j,i,\lambda)}_t+\chi_{ij}\chi_{i,\lambda}\ell_{j,\rho(i,\lambda)}(L_{s,t})+\chi_{ij}\dd_j\ell_{i,\lambda}(L_{s,t})
\end{align*}
and we conclude by Lemma \ref{horrible}(1) that $\dd_i\dd_j(a^\lambda_{t})=\chi_{ij}\dd_j\dd_i(a^\lambda_{t})+\sum_\ell c^{ij}_\ell \dd_\ell(a^\lambda_{t})$, hence proving condition (\ref{1}) and completing the inductive step. 

\smallskip

% If both $(\rho(i,\tau),t)$ and $(\rho(j,\tau),t)$ are leaders, then (1) holds by the proof of Claim \ref{claim_well_def}.

This concludes the proof of Theorem ~\ref{thebigone}
\end{proof}

\begin{remark}\label{removeminimal}
We note that when $\dim_k(\DD_1)=1$ the assumption on minimal-separable leaders in Theorem~\ref{thebigone} is automatically satisfied for arbitrary $r\geq 1$,  as, in this case, $\NN_s^\md=\emptyset$ for any $s\geq 2$.
\end{remark}

As not all kernels are separable (and this is part of the assumptions in Theorem~\ref{thebigone}), we will need the following proposition to get around inseparability issues in the next section. Recall that when $char(k)=p>0$, for $u\in\{1,2\}$, we denote by $\operatorname{Fr}_u$ the Frobenius endomorphism on $\DD_u$.

\begin{proposition}\label{careful}
Suppose $\Gamma$ is Jacobi-associative. Assume $char(k)=p>0$ and $\mm_u=\ker(\operatorname{Fr}_u)$ for $u\in \{1,2\}$. If $L_{s+1}$ is a $\gDD$-kernel, then there is a separable $\gDD$-kernel $E_{s+1}$ that prolongs $L_{s}$ such that the minimal-separable leaders of $E_{s+1}$ coincide with those of $L_{s}$.
\end{proposition}
\begin{proof}
	
	We proceed as in the proof of Theorem \ref{thebigone}. Let $(X_t^\mu)_{(\mu,t)\in \NN_{s+1}^\md\times \mn}$ be a tuple from the algebraically closed field $\Omega$ which is algebraically independent over $L_{s}$. Let $(\tau,t)\in \NN_s^\md\times\mn$ and suppose we have extended the $\DD$-structure $\be:L_{s-1}\to \bDD(L_s)$ to $\be: L_{\ine (\tau,t)}\to \bDD(\Omega)$. We consider cases:
	
	\medskip

	\noindent {\bf Case 1.} Suppose $(\tau,t)$ is a separable leader. By Lemma~\ref{extendstructure}(i), there is a unique $\bDD$-structure  $L_{\ineq(\tau,t)}\to \bDD(\Omega)$ extending  $L_{\ine(\tau,t)}\to \bDD(\Omega)$. We put $a'^{(i,\tau)}_t:=\dd_i(a^{\tau}_t)$.
	
	\medskip
	
	\noindent {\bf Case 2.} Suppose $(\tau,t)$ is not a separable leader. Using Lemma~\ref{extendstructure}(ii) when $(\tau,t)$ is not a leader and  Lemma~\ref{extendstructure}(iii) when $(\tau, t)$ is an inseparable leader (note this uses the fact the that $L_s$ has at least one prolongation, namely $L_{s+1}$), we can extend the $\bDD$-structure to $a_t^\tau$ arbitrarily; i.e., we have freedom to choose $a'^{(i,\tau)}$. We put
	$$a'^{(i,t)}_t:=\chi_{i,\tau}X_t^{\rho(i,\tau)}+\ell_{i,\tau}(L_{s,t}).$$ 
	
As in Theorem~\ref{thebigone}, this construction yields a $\bDD$-kernel $L'_{s+1}$ extending $L_s$	 that is not necessarily a $\gDD$-kernel. But this can be fixed in the same way as in the theorem. In fact, Claim~\ref{claim_well_def} can be proved easily with our current assumptions; indeed, if $\rho(i,\eta)=\rho(j,\eta')$ for some separable leaders $(\eta,t),(\eta',t)\in \NN_{s}^{\md}\times \mn$, then because $L_{s+1}$ is a $\gDD$-kernel we have $\dd_{i}(a^{\eta}_t)-\ell_{i,\eta}(L_{s,t})=\dd_{j}(a^{\eta'}_t)-\ell_{j,\eta'}(L_{s,t})$. From here, the same specialisation process performed in Theorem~\ref{thebigone} yields a separable $\gDD$-kernel $E_{s+1}$ prolonging $L_s$.
\end{proof}

We provide an example that shows that the conclusion of the proposition above does not generally hold if we remove the assumption $\mm_u=\ker(\operatorname{Fr}_u)$.

\begin{example}
	Let $k=\mathbb F_2$ and $\DD=k[\epsilon]/(\epsilon)^3$ 
	%with $\pi(\epsilon)=0$ 
	with ranked basis $(1,\epsilon,\epsilon^2)$. Note that in this case $\mm=(\epsilon)$ while $\ker(\operatorname{Fr})=(\epsilon^2)$. We set $\Gamma=\{r\}$ where $r:\DD\to \DD\otimes_k\DD$ is the canonical embedding ($a\mapsto 1\otimes a$). Thus, $\Gamma$ is Jacobi of Lie-commutation type. In this instance, a $\DD^\Gamma$-ring is a ring equipped with a 2-truncated Hasse-Schmidt derivation $(\dd_1,\dd_2)$ such that $\dd_1\dd_2=\dd_2\dd_1$. Let $K=k(s,t)$, where $s$ and $t$ are algebraically independent over $k$, and equip $K$ with the $\gD$-structure determined by
	$$\dd_1(t)=0,\quad \dd_1(s)=0, \quad \dd_2(t)=s,\quad  \dd_2(s)=0.$$ 
	Let $a^\emptyset=t^{1/2}$, $a^{(1)}=s^{1/2}$, and $a^{(2)}=0$. Then, one can readily check that 
	$$L_1=K(a^\xi: \xi\in \{1,2\}^{\leq 1})=K(a^\emptyset, a^{(1)}, a^{(2)})$$
	is a $\gD$-kernel over $K$ of length one. Note that $L_1$ is not a separable kernel as $a^{(1)}$ is an inseparable leader. We claim that, in fact, there is no separable $\gD$-kernel of length one that prolongs $L_0=K(a^\emptyset)=K(t^{1/2})$. Indeed, let $(E_1,e)$ be a $\gD$-kernel that prolongs $L_0$. Writing $E_1=K(b^\emptyset, b^{(1)}, b^{(2)})$, we have $b^\emptyset=a^\emptyset=t^{1/2}$. Then $(e(b^{\emptyset}))^2=e(t)$, and, after expressing this in terms of $\dd_i$'s, this is equivalent to
	$$(\dd_1(b^\emptyset))^2=\dd_2(t),$$
	which is the same as $b^{(1)}=s^{1/2}$, and hence $E_1$ is not a separable kernel.
\end{example}

We finish this section by proving the converse of Proposition~\ref{needindependent}. In fact, we prove a stronger statement which yields the existence of $\gDD$-polynomial rings.

\begin{corollary}\label{functionfield}
Assume $\Gamma$ is Jacobi-associative and let $x$ be a transcendental over $K$. Then, there exists a $\gDD$-field generated (as $\bDD$-field) by $x$ such that the tuple
$$(\dd_{i_1}\dots\dd_{i_s}(x))_{(i_1,\dots,i_s)\in \NN_{<\infty}^\md}$$
is algebraically independent over $K$. We denote this $\gDD$-field by $K(x)_{\gDD}$.
\end{corollary}
\begin{proof}
Consider the $\gDD$-kernel $L_2$ given by 
$$L_2=K(x^\xi: \xi\in \NN_{\leq 2}^\md)$$
where $(x^\xi: \xi\in \NN_{\leq 2}^\md)$ are new independent transcendentals with $x=x^\emptyset$ and we set $\dd_i(x)=x^i$ and $\dd_i(x^j)=\chi_{i,j}x^{\rho(i,j)}+\ell_{i,j}(L_1)$. Note that $L_{2}$ has no leaders (in particular, no separable leaders). Thus, we can apply Theorem~\ref{thebigone} to obtain a principal realisation of $L_{2}$. This realisation has the desired properties.
%... {\bf perhaps need to add dertails...}
\end{proof}

\begin{remark}\label{Dpolyring} Assuming $\Gamma$ is Jacobi-associative.
\begin{enumerate}
\item We denote the $\gDD$-algebra over $(K,\be)$ generated by $x$ inside $K(x)_{\gDD}$ by $K[x]_{\gDD}$. This algebra can be considered as the \emph{$\gDD$-polynomial ring over $(K,\be)$}. Indeed, it follows that if $A$ is a $\gDD$-algebra over $(K,\be)$ generated by a singleton $a\in A$, then $a$ is a $\bDD$-specialisation of $x$.
\item The above can be extended to a tuple $\bar x=(x_1,\dots,x_n)$ and yields the $\gDD$-polynomial ring in $n$-variables $K[\bar x]_{\gDD}$.
\end{enumerate}
\end{remark}

\

\subsection{Alternative axioms for $\DD$-CF}

In this section we restrict ourselves to a single local operator-system $(\DD,\bar\epsilon)$ and note how to use $\DD$-kernels to provide alternative axioms for $\DD$-CF (when $\DD$ is local). The original axioms in characteristic zero appear in \cite{MooScan2014}, while in positive characteristic appear in \cite{BHKK2019}. This will somewhat serve as a warm-up for our axioms $\gDD$-CF in the next section. 

We fix a $\DD$-field $(K,e)$ (we no longer require any commutativity). As we are working with a single operator system, note that $\md=\{1,\dots,m\}$ where $\dim_k(\DD)=m+1$. We also note that the order $\ineq$ can be extended from $\NN_{<\infty}^\md\times\mn$ to $\md^{<\omega}\times \mn$ by setting $(\xi,i)\ineq (\tau,j)$ when
$$(|\xi|,i,\xi)\leq_{\operatorname{lex}}(|\tau|,j,\tau)$$
where $\leq_{\operatorname{lex}}$ is the left-lexicographic order on $\NN_0^2\times\md^{<\infty}$ (recall $|\xi|$ is the length of $\xi$). 
%(the convention here is that when we are down to comparing pairs of the form $(u,i)$, with $0\leq u\leq s$ and $1\leq i\leq m_u$, we do this left-lexicographically). 
With respect to this order, we extend the notion of a leader to tuples from $\md^{<\omega}\times \mn$ as follows: Let $(L_r,e)$ be a $\DD$-kernel over $(K,e)$ and $(\xi,i)\in \md^{\leq r}\times\mn$; we say that $(\xi,i)$ is a separable (inseparable) leader of $L_r$ if $a_i^{\xi}$ is separably (inseparably) algebraic over 
$$L_{\ine (\xi,i)}:=K(a_j^\eta: (\eta,j)\ine (\xi,i)).$$
We may now say that $L_r$ is a separable kernel if there is no inseparable leader $(\xi,i)\in \md^{\leq r}\times\mn$ with $|\xi|=r$. Similarly, we may adapt the notions of prolongations and regular/principal realisation by removing the requirement that the operators $\Gamma$-commute.

Straightforward adaptations of the proofs of Theorem~\ref{thebigone} and Proposition~\ref{careful} yield the following proposition. We leave the details of these adaptations to the interested reader; in fact, the arguments are much shorter as Claim~\ref{claim_well_def} and the specialisation processes performed there were only required to obtain $\Gamma$-commutativity, and hence can be omitted if one simply wishes to produce $\bDD$-kernels.

\begin{proposition}\label{adaptD}
Let $L_r$ be a $\DD$-kernel. 
\begin{enumerate}
\item If $L_r$ is separable, then $L_r$ has a principal realisation.
\item Assume ${\operatorname{char}}(k)=p>0$ and $\mm=\ker(\operatorname{Fr})$. Then, there is a separable $\DD$-kernel $E_r$ that prolongs $L_{r-1}$.
\end{enumerate}
\end{proposition}

We will make the following assumption.

\begin{assumption}\label{onFro}
If $char(k)=p>0$, then $\mm=\ker(\operatorname{Fr})$.
\end{assumption}

\begin{theorem}\label{altaxioms}
Under Assumption~\ref{onFro}. Let $(K,e)$ be a $\DD$-field. Then, $(K,e)$ is existentially closed in the class of $\DD$-fields if and only if, for every $r,n\in \NN$, if $L_r=K(a_i^\xi:(\xi,i)\in \md^{\leq r}\times \mn)$ is a separable $\DD$-kernel of length $r$ in $n$-variables over $(K,e)$, then the tuple 
$$(a_i^\xi:\, (\xi,i)\in \md^{\leq r}\times \mn)$$
has an algebraic specialisation over $K$ of the form
$$(\dd_{\xi}b_i:\, (\xi,i)\in \md^{\leq r}\times \mn)$$
for some $n$-tuple $(b_1,\dots,b_n)$ from $K$, where 
$\dd_{\xi}=\dd_{\xi_s}\cdots\dd_{\xi_s}$ when $\xi=(\xi_1,\dots,\xi_s)$. 
\end{theorem}
\begin{proof}
($\Rightarrow$) By Proposition~\ref{adaptD}(i), $L_r$ has a principal realisation. As $(K,e)$ is existentially closed in such realisation, the existence of the algebraic specialisation follows.

\smallskip

($\Leftarrow$) Let $\phi(x_1,\dots,x_n)$ be a quantifier-free formula over $K$ with a realisation $(a_1,\dots,a_n)$ in some $\DD$-field extension $(E,e)$. Let $r\in \NN$ such that the maximum length of compositions of the operators appearing in $\phi$ is at most $r-1$. Now consider the $\DD$-kernel
$$E_{r}=K(\dd_\xi a_i:(\xi,i)\in \md^{\leq r}\times \mn).$$
By Proposition~\ref{adaptD}(ii), there is a separable $\DD$-kernel $L_{r}$ that prolongs $E_{r-1}$. Then, the algebraic specialisation produced by our assumption realises $\phi$ in $K$.
\end{proof}

\

\section{The model companion $\gDD$-CF}\label{modeltheory}\label{sec6}

In this section we prove that, under Assumption~\ref{theassumption} below (cf. the assumption in Proposition~\ref{careful}), the theory of $\gDD$-fields has a model companion which we denote by $\gDD$-CF. Then, in \S\ref{char0}, we restrict to the case $char(k)=0$ and establish some of the model-theoretic properties of this theory; namely, we observe that it is a complete $|F|$-stable theory with quantifier elimination and elimination of imaginaries, it has the canonical base property and satisfies (the expected form of) Zilber's dichotomy for finite dimensional types.

\medskip

We carry forward the notation of \S\ref{generalcase} and \S\ref{kernels}. Namely, 
$$\bDD=\{(\DD_1,\bar\epsilon_1),(\DD_2,\bar\epsilon_2)\}$$ 
where each $(\DD_u,\bar\epsilon_u)$ is a local operator-system. $(F,\be)$ is a fixed $\bDD$-field and all rings are $F$-algebras and $\bDD$-rings are $(F,\be)$-algebras. We also fix, throughout, a commutation system $\Gamma=\{r_1,r_2\}$ for $\bDD$ over $F$ of Lie-Hasse-Schmidt type, and denote its coefficients by $(c_{u,\ell}^{ij})$. By a $\gDD$-field we mean a $\bDD$-field with $\Gamma$-commuting operators. We further assume that $\Gamma$ is Jacobi-associative and that $(F,\be)$ is a $\gDD$-field. 

Recall that we denote the associated operators by $\dd_{u,i}$, and that they are additive operators satisfying
\begin{equation}\label{rulesaxioms}
	\dd_{u,i}(xy)=\dd_{u,i}(x)\, y +x\, \dd_{u,i}(y)+\sum_{p,q=1}^m\alpha_{u,i}^{pq}\dd_{u,p}(x)\dd_{u,q}(y)
\end{equation}
where $\alpha_{u,i}^{pq}\in k$ is the coefficient of $\epsilon_{u,i}$ in the product $\epsilon_{u,p}\cdot \epsilon_{u,q}$ happening in $\DD_u$. 
%for $(p,q)\in \gamma_u(i)$; otherwise, $\alpha_{p,q}^i=0$ (cf. Remark~\ref{operatornotation}).

\bigskip

The class of $\gDD$-fields is (universally) axiomatisable in the language
$$\mathcal L_{\bDD}(F)=\{0,1,+,-,*,\;^{-1},(\lambda_a)_{a\in F},(\dd_{u,i}:u\in \{1,2\},1\leq i\leq m_u )\}$$
where $\lambda_a$ denotes scalar multiplication by $a\in F$. Note that the axioms include the quantifier-free diagram of $(F,\be)$, since $\bDD$-fields are assumed to be $\bDD$-extensions of $F$. They also specify that the operators $\dd_{i,u}$ are additive, satisfying \eqref{rulesaxioms}, and that they $\Gamma$-commute (which is given by universal sentences using the coefficients $c_{u,\ell}^{ij}$, see Lemma~\ref{meaningcomm}). We denote the theory of $\gDD$-fields by $\gDD$-F. Note that $\gDD$-F is a consistent theory as $(F,\be)$ is a model.

\smallskip

We will prove that a model companion for $\gDD$-fields exists under the following assumption (see the hypothesis in Proposition~\ref{careful}). Note that this type of assumption also appears in \cite[Assumption 2.5]{BHKK2019} where they prove that in the ``free'' context (i.e., not requiring $\Gamma$-commutativity) this assumption is equivalent to the existence of a companion. We make further comments around this after Remark~\ref{remark61} below.

\begin{assumption}\label{theassumption}
	If $char(k)=p>0$ and $\dim_k(\DD_1)>1$, then $\mm_u=\ker(\operatorname{Fr_u})$ for $u\in \{1,2\}$. Recall that $\operatorname{Fr}_u$ denotes the Frobenius endomorphism on $\DD_u$.
\end{assumption}

For the remainder of this section we work under Assumption~\ref{theassumption}. Note that this assumption holds in the context of Examples ~\ref{Liecommexample} to \ref{examplesevHS}; and it also holds in Example~\ref{allcombined} when the $n$ there equals 1.

\smallskip

We are now in the position to prove companiability. 

\begin{theorem}\label{companiontheory}
	Under Assumption~\ref{theassumption}. Let $(K,\be)$ be a $\gDD$-field. Then, $(K,\be)$ is existentially closed in the class of $\gDD$-fields if and only if, for every $r,n\in \mathbb N$, if $L_{2r}=K(a_i^\xi:(\xi,i)\in \md^{\leq 2r}\times \mn)$ is a separable $\gDD$-kernel of length $2r$ in $n$-variables over $(K,\be)$ such that the minimal-separable leaders of $L_{2r}$ are the same as those of $L_r$, then the tuple 
	$$(a_i^\xi:\, (\xi,i)\in \md^{\leq r}\times \mn)$$
	has an algebraic specialisation over $K$ of the form
	$$(\dd_{\xi}b_i:\, (\xi,i)\in \md^{\leq r}\times \mn)$$
	for some $n$-tuple $(b_1,\dots,b_n)$ from $K$, where 
	$\dd_{\xi}=\dd_{\xi_1}\cdots\dd_{\xi_r}$ when $\xi=(\xi_1,\dots,\xi_r)$. 
\end{theorem}
\begin{proof}
	Essentially this is the content of Theorem~\ref{thebigone} together with Proposition~\ref{careful}; and the argument follows the same lines as the proof of Theorem~\ref{altaxioms}.
	
	\smallskip
	
	($\Rightarrow$) By Theorem~\ref{thebigone}, $L_r$ has a principal realisation. As $(K,\be)$ is existentially closed in such realisation, the existence of the algebraic specialisation follows.
	
	\smallskip
	
	($\Leftarrow$) Let $\phi(x_1,\dots,x_n)$ be a quantifier-free formula over $K$ with a realisation $(a_1,\dots,a_n)$ in some $\gDD$-field extension $(E,\be)$. We may assume that $E$ is $\DD$-generated by $(a_1,\dots,a_n)$ over $K$. Let $r\in \NN$ be such that the maximum length of compositions of the operators appearing in $\phi$ is at most $r-1$ and if $(\xi,i)$ is a minimal-separable leader of $E$ then $|\xi|\leq r-1$ (this can be achieved as the set of minimal-separable leaders is finite by Lemma~\ref{dick}). Now consider the (nonnecessarily separable) $\gDD$-kernel
	$$E_{2r}=K(\dd^\xi a_i:(\xi,i)\in \md^{\leq 2r}\times \mn).$$
	By Proposition~\ref{careful} (this is where we use Assumption~\ref{theassumption}), there is a separable $\gDD$-kernel $L_{2r}$ that prolongs $E_{2r-1}$ whose minimal-separable leaders are those of $E_{2r-1}$. By the choice of $r$, the minimal-separable leaders of $L_{2r}$ are the same as those of $L_r$. Then, the algebraic specialisation produced by our assumption realises $\phi$ in $K$.
\end{proof}

\begin{corollary}
	Under Assumption~\ref{theassumption}. The theory of $\gDD$-fields has a model companion. We denote the model companion by $\gDD$-CF and call its models $\gDD$-closed fields.
\end{corollary}
\begin{proof}
	This an immediate consequence of Theorem~\ref{companiontheory} as the theory $\gDD$-F is inductive (in fact universal), and so a model companion exists if and only if the class of existentially closed models is elementary. In addition, a standard/common argument explains why the conditions of the theorem (i.e., expressing the property of being a separable $\gDD$-kernel and when leaders are separable-leaders) can be written as a scheme of first-order axioms in the language $\mathcal L_{\bDD}(F)$; for example, see \cite[\S 2]{Kow} and \cite[Proof of Corollary 4.6]{Pierce2014} for such explanations.
\end{proof}

\begin{remark}\label{remark61}
	We note that when $\dim_k(\DD_1)=1$ in the axioms of $\gDD$-CF appearing in Theorem~\ref{companiontheory} one can remove the condition that ``the minimal-separable leaders of $L_{2r}$ are the same as those of $L_r$'' by requiring that $r\geq 2$. Indeed, as we pointed out in Remark~\ref{removeminimal}, in this case the condition on minimal-separable leaders is immediately satisfied as soon as $r\geq 2$.
\end{remark}

It follows from Corollary~\ref{commuting_sep} that models of $\gDD$-CF are separably closed fields. In the following subsection we derive several model-theoretic properties of the theory $\gDD$-CF when restricted to $char(k)=0$. We leave the exploration of the case when $char(k)=p>0$ for future work.

\medskip

We conclude this section by noting that at the moment we do not know whether the companion still exists after removing Assumption~\ref{theassumption}. As we already noted before Assumption ~\ref{theassumption}, in the case of $\DD$-fields (i.e., the \emph{free} context) the main result of \cite{BHKK2019} shows that a companion exists if and only if this assumption holds. The argument is based on 
%\cite[Proposition 7.2]{MooScan2014} which states that when Assumption~\ref{theassumption} does not hold then the theory of $\DD$-fields does not admit a model companion, and this is based on 
\cite[Proposition 7.1]{MooScan2014} which states that when Assumption~\ref{theassumption} fails then the condition of having a $p$-th root in a $\DD$-field extension is not first-order. More precisely, if there is $\epsilon\in \mm$ such that $\epsilon^p\neq 0$, then there is no first-order formula $\phi$ that describes (uniformly) the following set in an arbitrary $\DD$-field $(K,e)$:
\begin{equation}\label{defineroot}
	\{a\in K: \; a \text{ has a $p$-th root in a $\DD$-field extension of } K\}.
\end{equation}
However, we now observe that that proposition (i.e., \cite[Proposition 7.1]{MooScan2014}) does not generally hold in the case of $\Gamma$-commutativity. Consider the case $\DD=k[\epsilon]/(\epsilon)^{n+1}$ and $\Gamma$ imposes that the operators $(\dd_1,\dots,\dd_{n})$ pairwise commute. Note that when $n\geq p$ then Assumption~\ref{theassumption} fails (as $\epsilon^p\neq 0$). We claim that in this instance the formula 
\begin{equation}\label{oneformula}
	\bigwedge_{p\;  \nmid \;  i} \; \dd_i(x)=0
\end{equation}
describes the set in \eqref{defineroot} for any $\DD^\Gamma$-field $(K,e)$. Indeed, the argument in \cite[Proposition 7.1]{MooScan2014} yields that the set \eqref{defineroot} is type-definable and in this case given by
$$\{\dd_i\dd_{s_1p}\cdots\dd_{s_mp}(x)=0: \; p\;\nmid \; i, m\geq 0, \text{ and integers } 1\leq s_1,\dots,s_m\leq n/p\}.$$
In the \emph{free} context this cannot be reduced to a single formula; however, in the context of pairwise commuting operators this partial type is equivalent to \eqref{oneformula}.

%$a$ will have a $p$-th root in some $\DD^\Gamma$-field extension iff there is $b$ in some $\DD^\Gamma$-extension with $e(a)=e(b)^p$. This is equivalent to  $$a+\dd_1(a)\epsilon+\cdots+ \dd_n(a)\epsilon^n=b^p+\dd_1(b)^p \epsilon^p+\cdots \dd_n(b)^p\epsilon^{np};$$ comparing coefficients this translates to 

\medskip

In fact we expect that such a formula will exist for any theory of $\gDD$-fields and thus we conjecture: 

\begin{conjecture} \
	\begin{enumerate}
		\item There is an $\mathcal L_{\bDD}(F)$-formula which (uniformly) describes the set \eqref{defineroot} in any $\gDD$-field.
		\item The model companion of $\gDD$-fields exists even when Assumption~\ref{theassumption} is dropped.
	\end{enumerate}
\end{conjecture}

\

\subsection{Model theoretic properties of $\gD$-CF in characteristic zero} \label{char0}

In this subsection we assume $char(k)=0$ (while carrying forward the same data $(\bDD,\Gamma)$ with $\Gamma$ an LHS-commutation system that satisfies Jacobi-associativity). As we noted in Remark~\ref{impliespositive}, the characteristic zero assumption implies that $\DD_2=k$, and thus we simply write $\DD$ for $\DD_1$ and use the notation $\gD$ instead of $\gDD$. Note that, in this case, $\DD$-CF recovers the theory explored in \cite{MooScan2014} by Moosa-Scanlon. We will deploy some of the results  there, together with the more recent \cite{LeonMohamed2024}, to deduce properties of $\gD$-CF. 

\smallskip

In \cite{LeonMohamed2024}, the notion of a theory $T$ being derivation-like with respect to another theory $T_0$, equipped with a suitable notion of independence $\ind^0$, was introduced. When the model companion $T_+$ of $T$ exists, it was shown in that paper that several model-theoretic properties transfer from $T_0$ to $T_+$ (the ones relevant to us will be stated in Corollary~\ref{fromderivationlike} below). We now observe that $\gD$-F is derivation-like with respect to ACF$_0$ (we provide the definition in the course of the proof).

\begin{proposition}
	The theory $\gD$-F (i.e., the theory of $\gD$-fields) is derivation-like w.r.t. $(ACF_0,\ind^{\operatorname{alg}})$. Here $\ind^{\operatorname{alg}}$ denotes the algebraic disjointness relation.
\end{proposition}
\begin{proof}
	We first note that in \cite[\S3.2]{LeonMohamed2024} it was already observed that $\gD$-F is derivation-like in the case when $\Gamma$ imposes that the operators pairwise commute; we provide details to cover the general case (i.e., arbitrary $\Gamma$).   Fix a monster model $\mathfrak C$ of ACFA$_0$. Let $K, L, E$ be $\gD$-fields which are subfields of $\mathfrak C$ (note that $\mathfrak C$ is not equipped with a $\DD$-structure) such that $E$ is a common $\DD$-subfield of $K$ and $L$, the field extensions $K/E$ and $L/E$ are regular (i.e., $E$ is algebraically closed in $K$ and $L$), and $K\ind^{\operatorname{alg}}_E L$. From \cite[Definition 2.1]{LeonMohamed2024}, to prove $\gD$-CF is derivation-like we must show that
	\begin{enumerate}
		\item there exists a field $M< \mathfrak C$ equipped with a $\gD$-structure such that $K$ and $L$ are $\DD$-subfields of $M$, and 
		\item for any $\gD$-field $M$ as above and any field $F$ with
		$$K\cdot L\leq F\leq (K\cdot L)^{\operatorname{alg}}\cap M$$
		we have that $F$ is a $\DD$-subfield of $M$ and this $\DD$-structure on $F$ is the unique one making it a $\gD$-field and extending those on $K$ and $L$.
	\end{enumerate}
	Indeed, since $K/E$ is regular, algebraic-disjointness implies that $K$ and $L$ are linearly-disjoint over $E$. Then, the compositum $K\cdot L$ is isomorphic to the fraction field of $K\otimes_E L$. By \cite[Proposition 2.20]{BHKK2019}, there is a (unique) $\DD$-structure on $K\otimes_E L$ extending those of $K$ and $L$. This induces a $\DD$-structure on $K\cdot L$, and by Lemma~\ref{commcomp} this is a $\gD$-structure (i.e., the operators $\Gamma$-commute). This yields part (1) with $M=K\cdot L$. Part (2) follows from the fact that $\gD$-structures extend uniquely to $\gD$-structures on separably algebraic extensions, see Theorem ~\ref{commutingext} and Corollary~\ref{commuting_sep} (and recall that we are in characteristic zero).
\end{proof}

Given that $\gD$-F is derivation-like w.r.t. to ACF$_0$ and it has a companion $\gD$-CF, we may collect some of the model-theoretic properties that follow from the results of \cite[\S2]{LeonMohamed2024}.

\begin{corollary}\label{fromderivationlike}
	The theory $\gD$-CF has the following properties:
	\begin{enumerate}
		\item [(i)] it is a complete theory with quantifier elimination,
		\item [(ii)] the model-theoretic $\dcl$ equals the $\DD$-field generated,
		\item [(iii)] the model-theoretic $\acl$ equals the field-theoretic algebraic closure of the $\DD$-field generated,
		\item [(iv)] it is a stable theory and nonforking independence coincides with algebraic disjointness of the $\gD$-fields generated by the parameter sets.
	\end{enumerate}
\end{corollary}

We now note that, more than stable, the theory is $|F|$-stable.

\begin{lemma}
	The theory $\gD$-CF is $|F|$-stable.
\end{lemma}
\begin{proof}
	Let $(K,e)$ be a $\DD$-subfield of a model of $\gD$-CF (in particular, $F\leq K$). By quantifier elimination, there is a 1-1 correspondence between 1-types over $K$ and prime $\DD$-ideals of the $\gD$-polynomial ring $K[x]_{\gD}$ (see Remark~\ref{Dpolyring}(1) for the construction of this ring), the notion of $\DD$-ideal being the natural one (i.e., an ideal closed under the operators $\dd_i$'s). By Theorem~\ref{thebigone} and Proposition~\ref{uniquereal}, every prime $\DD$-ideal is determined by its minimal-separable leaders (in characteristic zero there are no inseparable leaders). By Lemma~\ref{dick}, the set of minimal-separable leaders is finite, and thus prime $\DD$-ideals are completely determined by this finite set of leaders. It follows that $|S_1^{\gD\text{-CF}}(K)|=|K|$.
\end{proof}

A consequence of quantifier elimination and stability is:

\begin{corollary}
	The theory $\gD$-CF eliminates imaginaries
\end{corollary}
\begin{proof}
	Let $(K,e)\models \gD$-CF. By quantifier elimination, for every $n$, there is a 1-1 correspondence between $n$-types over $K$ and prime $\DD$-ideals of the $\gD$-polynomial ring $K[x_1,\dots,x_n]_{\gD}$ (see Remark~\ref{Dpolyring}(2) for the construction of this ring). This implies that the $\DD$-fields generated by the field of definition of $\DD$-ideals serve as canonical base for types. Now stability yields elimination of imaginaries by the criterion of Evans-Pillay-Poizat, see \cite[Corollary 5.9]{Messmer1996}.
\end{proof}

We now aim to prove that $\gD$-CF satisfies the (expected form of) Zilber's dichotomy for finite-dimensional types. We do this via proving the strong form of the Canonical Base Property (CBP) for such types. This will follow as a consequence of the analogous result in \cite{MooScan2014} for $\DD$-CF. To achieve this, we will use the following lemma. But first, for a $\DD$-field $(K,e)$ (not necessarily with $\Gamma$-commuting operators), recall that the field of $\DD$-constants of $K$ is
$$C_K=\{a\in K: \dd_i(a)=0 \text{ for all } 1\leq i\leq m\}.$$
Furthermore, we define the $\Gamma$-commuting subfield of $K$ as
$$K^{\Gamma}=\{a\in K:\, e \text{ commutes on } F(a)_{\DD} \text{ w.r.t. } \Gamma\}$$
where $F(a)_{\DD}$ is the $\DD$-field generated by $a$ over $F$. Note that, by Lemma~\ref{commcomp}, $K^\Gamma$ is a $\DD$-subfield of $K$, and, of course, $(K^\Gamma,e)$ is a $\gD$-field (i.e., the $\DD$-operators $\Gamma$-commute). Furthermore, $K^\Gamma$ is quantifier-free type-definable; defined by countably-many quantifier-free $\mathcal L_{\DD}(F)$-formulas. Clearly, $C_K\leq K^\Gamma$.

\begin{lemma}\label{transfer}
	Suppose $(K,e)$ is an $\aleph_1$-saturated model of $\DD$-CF. Then, $(K^\Gamma, e)$ is a model of $\gD$-CF.
\end{lemma}
\begin{proof}
	Let $\Lambda(x)$ be the collection of quantifier-free formulas defining $K^\Gamma$. Namely, $\Lambda$ consists of formulas of the form
	$$[\dd_i,\dd_j](\dd^\xi x)=c_1^{ij}\dd_1(\dd^\xi x)+\cdots + c_m^{ij}\dd_m(\dd^\xi x)$$
	where $\dim\DD=m+1$, $\xi$ varies in (finite) tuples with entries in $\{1,\dots,m\}$, and $\dd^\xi$ denotes $\dd_{\xi_1}\cdots\dd_{\xi_s}$.
	Also, let $\phi(x)$ be a quantifier-free formula over $K^\Gamma$ with a realisation $a$ in some $\gD$-field extension $(L,e)$. Since $\DD$-fields have the amalgamation property, we can find an amalgam $(E,e)$ of $K$ and $L$ over $K^\Gamma$. Then, $a$ realises $\phi(x)$ and also $\Lambda(x)$ in $E$. By saturation of $K$, there is a realisation $b$ in $K$, but then $b\in K^\Gamma$ (as it realises $\Lambda(x)$). 
\end{proof}

For the following theorem, we fix $(L,e)$ a sufficiently saturated model of $\gD$-CF and, by Lemma~\ref{transfer}, we may assume $L=\mathcal U^\Gamma$ where $\mathcal U$ is a monster for $\DD$-CF. This implies $C_L=C_{\mathcal U}$, and we denote this simply by $C$. We also let $(K,e)$ be a $\DD$-subfield of $L$. We say that a type $p\in S_n^{\gD\text{-}CF}(K)$ is finite-dimensional if for some (equivalently any) realisation $a\models p$ the $\DD$-field generated by $a$ over $K$ is of finite transcendence degree over $K$. We can now conclude with:

\begin{theorem}[The CBP and the dichotomy for finite-dimensional types]
	Suppose $p=tp(a/K)\in S_n^{\gD\text{-}CF}(K)$ is a finite-dimensional type. If $b$ is a tuple from $L$ such that 
	$$\dcl(Cb(a/K(b)_{\DD}))=\dcl(K,b),$$ 
	then $tp(b/K(a)_{\DD})\in S_n^{\gD\text{-}CF}(K(a)_{\DD})$ is internal to the $\DD$-constants $C$. As a result, if $p$ is minimal (U-rank one) and non locally modular, then $p$ is nonorthogonal to the $\DD$-constants.
\end{theorem}
\begin{proof}
	By \cite[Theorem 5.9]{MooScan2014}, we know that for $\DD$-fields forking independence in $\DD$-CF is equivalent to algebraic disjointness. Thus, for $\DD$-subfields of $L$, being forking independent in the sense of $\gD$-CF and in the sense of $\DD$-CF coincide. It follows that the notion of canonical basis (for a type in $S_n^{\gD\text{-}CF}$) also coincide. Hence we may apply the CBP for $\DD$-CF \cite[Corollary 6.18]{MooScan2014}, which yields that the type $q$ of $b$ over $K(a)_{\DD}$ in the sense of $\DD$-CF is internal to $C$. Namely, we may find a tuple $d$ of realisations of $q$ and a tuple from $C$ such that $b\in \dcl(K,a,c,d)$. But $a,c,d$ are in $L$ (the latter because any realisation of $q$ is in $L$), and thus the type is $C$-internal in the sense of $\gD$-CF as well.
	
	The claimed form of the dichotomy follows now from the CBP by a standard argument (see for instance \cite[Corollary 6.19]{MooScan2014}).
\end{proof}

\

\section{Further Remarks}\label{sec7}

In this final section we address two points:

\begin{enumerate} 
	\item There is a theory that, for large fields, axiomatises the class of those $\gDD$-fields that are existentially closed in $\gDD$-field extensions when they are existentially closed in the language of fields. When $char(k)=0$, this theory serves as a uniform companion for model-complete theories of large $\gDD$-fields. This generalises Tressl's uniform companion for several commuting derivations \cite{Tressl2005}.
	\item Point (1) suggests a natural notion of $\gDD$-largeness (similar to differential largeness \cite{LSTr2023,LSTr2024}) in arbitrary characteristic. We observe that PAC-substructures in $\gDD$-CF are precisely those $\gDD$-fields that are PAC-fields and $\gDD$-large. 
\end{enumerate} 

Unless otherwise stated the characteristic of $k$ remains arbitrary and we carry forward the notation from Section~\ref{modeltheory} (in particular, the assumptions on $\bDD$, $\Gamma$ and $(F,\be)$ remain). Furthermore, throughout this section we work under Assumption~\ref{theassumption}; recall that this assumption states that, if $char(k)=p>0$ and $\dim_k(\DD_1)>1$, then $\mm_u=\ker(\operatorname{Fr}_u)$ for all $u=1,2$.

%Also, for the language of $\bDD$-fields we will write $\mathcal L_{\bDD}$ instead of $\mathcal L_{\bDD}(F)$, since $F$ is fixed.

\

\subsection{The uniform companion UC$_{\gDD}$}

Let $(K,\be)$ be a $\gDD$-field. A $\gDD$-kernel $L_r=K(a_i^\xi: (\xi,i)\in \md^{\leq r}\times\mn)$ is said to have a \emph{smooth $K$-point} if the affine variety with coordinate ring 
$$K[a_i^\xi: (\xi,i)\in \md^{\leq r}\times\mn]$$
has a smooth $K$-rational point.

\begin{definition}
	We define the $\mathcal L_{\bDD}(F)$-theory UC$_{\gDD}$ as follows: a $\gDD$-field $(K,\be)$ is a model of UC$_{\gDD}$ if and only if, for every $r,n\in \mathbb N$, if $L_{2r}=K(a_i^\xi:(\xi,i)\in \md^{\leq 2r}\times \mn)$ is a separable $\gDD$-kernel of length $2r$ in $n$-variables over $(K,\be)$ with a smooth $K$-point such that the minimal-separable leaders of $L_{2r}$ are the same as those of $L_r$, then the tuple 
	$$(a_i^\xi:\, (\xi,i)\in \md^{\leq r}\times \mn)$$
	has an algebraic specialisation over $K$ of the form
	$$(\dd^{\xi}b_i:\, (\xi,i)\in \md^{\leq r}\times \mn)$$
	for some $n$-tuple $(b_1,\dots,b_n)$ from $K$, where 
	$\dd^{\xi}=\dd_{\xi_1}\cdots\dd_{\xi_r}$ when $\xi=(\xi_1,\dots,\xi_r)$. 
\end{definition}

Recall that a field $K$ is said to be large if every irreducible variety $V$ over $K$ with a smooth $K$-rational point has a Zariski-dense set of $K$-rational points. The latter condition is equivalent to $K$ being existentially closed in the function field $K(V)$. Large fields include local fields as well as pseudo-classically closed fields (such as PAC and PRC) and fraction fields of local Henselian domains. Examples of non large fields include number fields and algebraic function fields. We refer the reader to \cite{Pop2013} for a little survey on large fields.

The following proposition provides an \emph{existentially closed} characterisation of models of UC$_{\gDD}$ which are large as fields. 

\begin{proposition}\label{Dlargeaxioms}
	Let $(K,\be)$ be a $\gDD$-field which is large (as a field). Then, $(K,\be)\models$UC$_{\gDD}$ if and only if, for every $\gDD$-extension $(L,\be)$, if $K$ is existentially closed in $L$ as fields then it is existentially closed as $\gDD$-fields.
\end{proposition}
\begin{proof}
	$(\Rightarrow)$ Assume $(K,\be)\models$ UC$_{\gDD}$ and $K$ is existentially closed in $L$ as fields. Let $\phi(x_1,\dots,x_n)$ be a quantifier-free $\mathcal L_{\bDD}$-formula over $K$ with a realisation $(a_1,\dots,a_n)$ in $(L,\be)$. Using the same argument as in the proof of Theorem~\ref{companiontheory}, for $r$ sufficiently large (in particular larger than any number of compositions of the operators appearing in $\phi$), we may assume that the $\gDD$-kernel
	$$L_{2r}=K(a_i^\xi:(\xi,i)\in \md^{\leq r}\times \mn)$$
	is separable and its minimal separable leaders agree with those of $L_r$. As $K$ is e.c. in $L$ as fields (and so also in $L_{2r}$), we have that $L_{2r}$ has a smooth $K$-point. Now the algebraic specialisation given by the UC$_{\gDD}$ axioms yields a realisation of $\phi$ in $K$.
	
	\medskip
	
	$(\Leftarrow)$ Since $K$ is large and $L_{2r}$ has a smooth $K$-point, $K$ must be existentially closed in $L_{2r}$ (as fields). Now, by Theorem \ref{thebigone}, $L_{2r}$ has a principal realisation, call it $L$. As the realisation is principal, the extension $L/L_{2r}$ is purely transcendental, and thus $K$ is also existentially closed in $L$ as fields. By assumption, the latter yields that $(K,\be)$ is e.c. in $(L,\be)$, and then the existence of the desired specialisation follows.
\end{proof}

We now observe that, when $char(k)=0$, the theory UC$_{\gDD}$ serves as a uniform model-companion for model-complete theories of large $\gDD$-fields (generalising Tressl's uniform companion in the case of several commuting derivations \cite{Tressl2005}). 

\begin{theorem}\label{uniformcompanion}
	Assume $char(k)=0$. Suppose $T$ is a model-complete theory of large fields in the language of fields. Let $T_{\gDD}$ be the $\mathcal L_{\bDD}(F)$ -theory of $\gDD$-fields that are models of $T$. Then, $T_{\gDD}^+:=T_{\gDD}\cup$UC$_{\gDD}$ is the model companion of $T_{\gDD}$. 
\end{theorem}
\begin{proof}
	We need to show that: (1) every model of $T_{\gDD}$ embeds in a model of $T_{\gDD}^+$, and (2) each model of $T_{\gDD}^+$ is existentially closed in models of $T_{\gDD}$.
	
	(1) It suffices to show that every $\gDD$-field $(K,\be)$ which is large (as a field) has a $\gDD$-extension which is a model of UC$_{\gDD}$ and is an elementary extension in the language of fields. So let $L_{2r}=K(a_i^\xi:(\xi,i)\in \md^{\leq 2r}\times \mn)$ be a separable $\gDD$-kernel with a smooth $K$-point such that the minimal-separable leaders of $L_{2r}$ are the same as those of $L_r$. By Theorem~\ref{thebigone}, $L_{2r}$ has a principal realisation, say $L$. Recall that $L/L_{2r}$ is purely transcendental and hence, since $L_{2r}$ has a smooth $K$-point and $K$ is large, we obtain that $K$ is e.c.in $L$ as fields. Then we can find an extension $E$ of $L$ such that $E$ is an elementary extension of $K$ as fields. We now argue that there is $\gDD$-structure on $E$ extending that on $K$. Let $A$ be a transcendence basis for $E$ over $K$. By Lemma~\ref{extendstructure}(ii), we may extend the $\DD$-structure from $K$ to $K(A)$ by $\dd_i(A)=0$. By the latter choice, the operators $\Gamma$-commute on $A$. Finally, by Theorem~\ref{commutingext}, the unique extension of the $\DD$-structure from $K(A)$ to $E$ also $\Gamma$-commutes (recall that we are in characteristic zero).
	
	Repeating this argument, we can use transfinite induction to construct the desired model of UC$_{\gDD}$ which is an elementary extension of $K$.
	
	\medskip
	
	(2) Assume $(K,\be)$ is a model of $T_{\gDD}^+$ and $(L,\be)$ is an extension which is a model of $T_{\gDD}$. Since $T$ is model complete $K\preceq L$; in particular, $K$ is e.c. in $L$ as fields. Then, by Proposition~\ref{Dlargeaxioms}, $(K,\be)$ is e.c. in $(L,\be)$.
\end{proof}

\begin{remark}
	We observe that the previous corollary does not hold in characteristic $p>0$. Indeed, models of UC$_{\gDD}$ have nontrivial $\DD$-structure (and therefore cannot be perfect). Thus, UC$_{\gDD}$ cannot serve as a model-companion for $ACF_p$.
\end{remark}

\

\subsection{$\gDD$-large fields and PAC substructures of $\gDD$-CF}

In this subsection the characteristic of $k$ is arbitrary and recall that we work under Assumption~\ref{theassumption}. Based on the notion of differentially large fields \cite{LSTr2023,LSTr2024}, we define:

\begin{definition}
	A $\gDD$-field $(K,\be)$ is said to be $\gDD$-large if it is large as a field and, for every $\gDD$-extension $(L,\be)$, if $K$ is existentially closed in $L$ as fields, then it is existentially closed as $\gDD$-fields.
\end{definition}

We observe that Proposition~\ref{Dlargeaxioms} has the following immediate consequence.

\begin{corollary}
	The class of $\gDD$-large fields is elementary: a set of axioms is given by axioms for large fields together with the UC$_{\gDD}$-axioms from the previous section.
\end{corollary}

This recovers the corresponding result in \cite[\S4]{LSTr2023} for several commuting derivations in characteristic zero and also the corresponding result in \cite[\S2]{LSTr2024} for a single derivation in positive characteristic. The new situation covered by our result that is of interest is the case of several commuting derivations in positive characteristic; in other words, the class of differentially large fields in several commuting derivations of arbitrary characteristic is elementary. 

\

Our final result is a generalisation of the fact that being a PAC substructure in DCF$_0$ is equivalent to being a PAC-field and differentially large (see \cite[Theorem~5.18]{LSTr2023}). Recall that a field $K$ is said to be a PAC-field if every absolutely irreducible variety over $K$ has a $K$-rational point. A $\gDD$-subfield $(K,\be)$ of a model of $\gDD$-CF is said to be a PAC substructure if $K$ is perfect (as a field) and, for every $\gDD$-extension $(L,\be)$, if $K$ is algebraically closed in $L$ then $(K,\be)$ is e.c. in $(L,\be)$.

\begin{proposition}
	Let $(K,\be)$ be a $\gDD$-field with $K$ perfect (as a field). Then, $(K,\be)$ is a PAC substructure for $\gDD$-CF if and only of $K$ is a PAC-field and $(K,\be)$ is $\gDD$-large.
\end{proposition}
\begin{proof}
	($\Rightarrow$) Let $V$ be an absolutely irreducible variety over $K$ and $L=K(V)$ the function field of $V$. Then $K$ must be algebraically closed in $L$. As $K$ is perfect, $L/K$ is separable and thus, by the same argument as in the proof of (1) of Theorem~\ref{uniformcompanion} (now using a \emph{separating} transcendence basis of $L/K$, which possible as $L$ is finitely generated over $K$), we can equip $L$ with a $\gDD$-structure extending that on $K$. Hence, $(K,\be)$ is e.c. in $(L,\be)$; in particular $V$ has a $K$-rational point. This shows that $K$ is a PAC-field. Now assume $(L,\be)$ is a $\gDD$-field extension of $(K,\be)$ such that $K$ is e.c. in $L$ as fields. The latter implies that $K$ is algebraically closed in $L$, and hence $(K,\be)$ is e.c. in $(L,\be)$.
	
	($\Leftarrow$) Let $(L,\be)$ be a $\gDD$-field extension of $(K,\be)$ such that $K$ is algebraically closed in $L$. As $K$ is a perfect PAC-field, $K$ must be e.c. in $L$ as fields. By $\gDD$-largeness we get that $(K,\be)$ is e.c. in $(L,\be)$.
\end{proof}

A direct consequence of these results is that the class of PAC substructures of $\gDD$-CF is elementary.

\

%%%%%%%%%%%%%BIBLIOGRAPHY

%\bibliographystyle{plain}
%\bibliography{isolated}

\bigskip

\begin{thebibliography}{99}


\bibitem{2}
J. Bell, S. Launois, O. Le\'on S\'anchez and R. Moosa.
\newblock Poisson algebras via model theory and differential algebraic geometry. 
\newblock Journal of the European Math. Soc., 19:2019--2049, 2017.

\bibitem{3}
J. Bell, O. Le\'on S\'anchez and R. Moosa.
\newblock D-groups and the Dixmier-Moeglin equivalence. 
\newblock Algebra \& Number Theory, 12(2):343--378, 2018.

\bibitem{BHKK2019}
O. Beyarslan, D. Hoffmann, M. Kamensky and P. Kowalski.
\newblock Model theory of free operators in positive characteristic.
\newblock Trans. of the Amer. Math. Soc., 379(8):5991--6016, 2019.


%\bibitem{6}
%Brown, O'Hagan, Zhang, Zhuang.
%\newblock Connected Hopf algebras and iterated Ore extensions. 
%\newblock Journal of Pure and Applied Algebra, 2015.




	\bibitem{Buium97}
	A. Buium.
	\newblock Arithmetic analogues of derivations. 
	\newblock Journal of Algebra 198(1): 290--299, 1997.
	
	\bibitem{Burton}
	C. Burton.
	\newblock A basis theorem for rings with commuting operators in characteristic zero.
	\newblock Preprint: arXiv:2501.19379v2, 2025.
	
	%\bibitem{Busta}
	%R. Bustamante.
	%\newblock Differentially closed fields of characteristic zero with a generic automorphism.
	%\newblock Revista de Matem\'atica: Teor\'ia y Aplicaciones, 14(1):81--100, 2007.
	
	\bibitem{Chatoperators}
	Z. Chatzidakis.
	\newblock Model theory of fields with operators - a survey.
	\newblock Logic without Borders, 2015.
	
	\bibitem{8}
Z. Chatzidakis and E. Hrushovski. 
\newblock Difference fields and descent in algebraic dynamics. 
\newblock Journal of the Inst. of Math. Jussieu, 7(4):653--686, 2008.
	
	\bibitem{Cohn51}
	R. Cohn. 
	\newblock Singular manifolds of difference polynomials. 
	\newblock Ann. Math., 53: 445--463, 1951.

	\bibitem{Cohn52}
R. Cohn. 
\newblock Extensions of difference fields. 
\newblock Amer. J. Math., 74: 507--530, 1992.

	\bibitem{Cohn53}
R. Cohn. 
\newblock On extensions of difference fields and the resolvents of prime difference ideals. 
\newblock Proc. Amer. Math. Soc, 3: 178--182, 1952.

\bibitem{Cohn} 
R. Cohn.
\newblock Specializations of differential kernels and the Ritt problem
\newblock Journal of Algebra 61: 256--268, 1979.
	
	
	\bibitem{figueira2011}
D. Figueira, S. Figueira, S. Schmitz, and P. Schnoebelen.
\newblock Ackermannian and primitive-recursive bounds with Dickson's lemma.
\newblock Logic in Computer Science (LICS), 26th Annual IEEE Symposium, pp.269--278, 2011. 
%\url{http://dx.doi.org/10.1109/LICS.2011.39}.

\bibitem{GLS2018}
R. Gustavson and O. Le\'on S\'anchez.
\newblock Effective bounds for the consistency of differential equations.
\newblock Journal of Symbolic Computation. 89: 41--72, 2018.

	
	\bibitem{Hasse37}
	H. Hasse and F. Schmidt.
	\newblock Noch eine Bergrundung der Theorie der hoheren Differentialquotienten in einem algebraischen Funktionen korper einer Unbestimmten.
	\newblock J. Reine Angew. Math. 177:215-237, 1937.
	



\bibitem{HK}
D. Hoffmann and P. Kowalski.
\newblock Existentially closed fields with G-derivations.
\newblock Journal of the London Math. Soc., 93:590--618, 2016.



\bibitem{12}
E. Hrushovski. 
\newblock The Manin-Mumford conjecture and the model theory of difference fields. 
\newblock Annals of Pure and Applied Logic, 112(1):43--115, 2001.

\bibitem{11}
E. Hrushovski.
\newblock The Mordell-Lang conjecture for function fields. 
\newblock Journal of the American Math. Soc., 9(3):667--690, 1996

\bibitem{13}
E. Hrushovski and A. Pillay
\newblock Effective bounds for the number of transcendental points on subvarieties of semi-abelian varieties.
\newblock American Journal of Mathematics, 122(3):439--450, 2000.

\bibitem{Hub2005}
E. Hubert.
\newblock Differential algebra for derivations with nontrivial commutation rules. 
\newblock Journal of Pure and Applied Algebra. 200: 163-- 190, 2005.

\bibitem{InoLS}
K. Ino and O. Le\'on S\'anchez.
\newblock The theory DCF$_p$A exists for $p>0$.
\newblock Preprint: arXiv:2410.17892v1, 2024.

\bibitem{Kolchin42}
E. Kolchin.
\newblock Extensions of differential fields, I.
\newblock Annals of Mathematics 43(4): 724-729, 1942.

\bibitem{Kolchin} 
E. Kolchin.
\newblock Differential Algebra and Algebraic Groups.
\newblock Academic Press, New York, 1973.

\bibitem{Kow}
P. Kowalski. 
\newblock Derivations of the Frobenius map. 
\newblock The Journal of Symbolic Logic, 70(1):99--110, 2005.

\bibitem{Kow2005}
P. Kowalski. 
\newblock Geometric axioms for existentially closed Hasse fields.
\newblock Annals of Pure and Applied Logic, 135:286--302, 2005.

\bibitem{Lando1970}
B. Lando.
\newblock Jacobi's bound for the order of systems of first order differential equations.
\newblock Trans. of the Amer. Math. Soc. 152(1): 119-135, 1970.

\bibitem{16}
S. Launois and O. Le\'on S\'anchez. 
\newblock The Dixmier-Moeglin equivalence for Poisson-Hopf algebras. 
\newblock Advances in Mathematics, 346:48--69 2019.

\bibitem{LeonSan}
O. Le\'on S\'anchez.
\newblock On the model companion of partial differential fields with an automorphism.
\newblock Israel Journal of Mathematics, 212:419--442, 2016.


\bibitem{LeonMohamed2024}
O. Le\'on S\'anchez and S. Mohamed.
\newblock Neostability transfers in derivation-like theories.
\newblock Model Theory, 4(2):177--201, 2025.

\bibitem{22}
O. Le\'on S\'anchez and A. Pillay.
\newblock Some definable Galois theory and examples. 
\newblock Bulletin of Symbolic Logic, 23(2):145--159, 2017.

\bibitem{LSTr2023}
O. Le\'on S\'anchez and M. Tressl.
\newblock Differentially large fields.
\newblock Algebra \& Number Theory, 18(2):250--280, 2024.

\bibitem{LSTr2024}
O. Le\'on S\'anchez and M. Tressl.
\newblock On ordinary differentially large fields.
\newblock Canadian Journal of Mathematics. Published Online, doi:10.4153/S0008414X24001172, 2024.

\bibitem{Levi42}
H. Levi. 
\newblock On the structure of differential polynomials and on their theory of ideals. 
\newblock Transactions of the American Math. Soc. 51(3): 532--568, 1942.

\bibitem{ACFA}
A. Macintyre.
\newblock Generic Automorphisms of Fields.
\newblock Annals of Pure and Applied Logic. Joint AILA-KGS Model Theory Meeting 88(2): 165--180, 1997.



\bibitem{Messmer1996}
D. Marker, M. Messmer and A. Pillay.
\newblock Model theory of fields.
\newblock Lecture Notes in Logic 5. ASL, 2005.

\bibitem{MooScan2011}
R. Moosa and T. Scanlon.
\newblock Generalised Hasse-Schmidt varieties and their jet spaces. 
\newblock Proc. of the London Math. Soc. 103(2):197--234, 2011.

\bibitem{MooScan2014}
R. Moosa and T. Scanlon.
\newblock Model theory of fields with free operators in characteristic zero.
\newblock Journal of Mathematical Logic. 14(2), 2014.


\bibitem{Pierce2014}
D. Pierce. 
\newblock Fields with several commuting derivations.
\newblock Journal of Symbolic Logic. 79(1): 1--19, 2014.

\bibitem{27}
A. Pillay
\newblock  Differential Galois theory I. 
\newblock Illinois Journal of Mathematics, 42(4):678--699,1998.

\bibitem{Pop2013}
F. Pop.
\newblock Little survey on large fields - Old \& New -. 2013. 


\bibitem{Raudenbush34}
H. Raudenbush. 
\newblock Ideal theory and algebraic differential equations.
\newblock Transactions of the American Math. Soc., 36(2): 361--368, 1934.

\bibitem{Ritt32}
J. Ritt.
\newblock Differential equations from the algebraic standpoint. 
\newblock Vol. 14. American Math. Soc., 1932.

\bibitem{Ritt34}
J. Ritt. 
\newblock Algebraic difference equations.
\newblock Bulletin of the American Math. Soc., 40:303--308, 1934.

\bibitem{Ritt35}
J. Ritt. 
\newblock Jacobi's problem on the order of a system of differential equations. 
\newblock Annals of Mathematics, 36(2):303--312, 1935.

\bibitem{Ritt33}
J. Ritt and J. Doob. 
\newblock Systems of algebraic difference equations. 
\newblock American Journal of Mathematics, 55:505--514, 1933.


\bibitem{Ritt39}
J. Ritt and H. Raudenbush. 
\newblock  Ideal theory and algebraic difference equations.
\newblock Transactions of the American Math. Soc., 46:445--453, 1939.


\bibitem{Robinson} 
A. Robinson. 
\newblock On the concept of a differentially closed field. 
\newblock Office of Scientific Research, US Air Force, 1959.

\bibitem{Sjoergen} 
N. Sj\"ogren.
\newblock The Model Theory of Fields with a Group Action. 
\newblock Research Reports in Mathematics, Department of Mathematics Stockholm University, 2005. %Available on \url{http://www2.math.su.se/reports/2005/7/}.

\bibitem{Swee1975}
M. E. Sweedler.
\newblock When is the tensor product of algebras local?
\newblock Proceedings of the Amer. Math. Soc. 48(1): 8--10, 1975.  

\bibitem{Tressl2005}
M. Tressl.
\newblock The uniform companion for large differential fields of characteristic zero.
\newblock Trans. of the Amer. Math. Soc. 357(10): 3933--3951.

\bibitem{Yaffe2001}
Y. Yaffe.
\newblock Model completion of Lie differential fields. 
\newblock Annals of Pure and Applied Logic. 107: 49--86, 2001.

%\bibitem{Zi} 
%M. Ziegler. 
%\newblock Separably Closed Fields with Hasse Derivations
%\newblock The Journal of Symbolic Logic, 68(1): 311-318, 2003.














\end{thebibliography}

\newpage

\appendix 

%\section{XXXX}

\section{Reduction of HS-iteration systems}\label{app_single_algebra}
%\section{Replacing several HS-commutation systems with a single one}\label{app_single_algebra}

We carry forward the notation used throughout the paper for local operator-systems $(\DD_u,\bar\epsilon_u)_{u=1}^n$ over the field $k$, and we assume that $F=k$. In this appendix we show (in Proposition~\ref{many_algebras}) that given finitely many homomorphisms 
$$(r_u:\DD_u\to \DD_u\otimes \DD_u)_{u=1}^n$$ 
of HS-iteration type that satisfy associativity (individually), a tuple of operators $e_1,\dots,e_n$ commuting with each other and such that each $e_u$ commutes with respect to $r_u^\iota$, corresponds to a single $\DD_1\otimes\dots\otimes \DD_n$-operator satisfying $r^\iota$-commutativity for a homomorphism $r$ of HS-iteration type that satisfies associativity and which is naturally obtained from $r_1,\dots,r_n$.

	For a homomorphism $r:\DD\to \DD\otimes \DD$ and any $0\leq i,j,\ell \leq m$ we  let $c^{ij}_\ell$ denote the coefficient of $\epsilon_i\otimes \epsilon_j$ in $r(\epsilon_\ell)$ (note that now we are allowing any of $i,j,\ell$ to be equal to zero and recall that $\epsilon_0=1$). We first observe that, as we are assuming $F=k$, the formula in the definition of associativity (see \S\ref{HStype}) holds for arbitrary $i,j,k,r\geq 0$ if we sum over all $\ell$ including $\ell=0$.
	
\begin{lemma}\label{associative_0}
	%(Assuming $F=k$) 
	Let $r:\DD\to \DD\otimes \DD$ be associative of HS-iteration type. Then, for all $i,j,k,r\geq 0$, we have
	$$\sum_{\ell\geq 0}  c_\ell^{ij} c_r^{\ell k} =\sum_{\ell\geq 0} c_\ell^{jk} c_r^{i\ell}.$$
	
\end{lemma}
\begin{proof} We may assume that one of the sides of the desired equality is not zero, so by symmetry we may assume it is the left-hand side; i.e., $\sum_{\ell\geq 0}  c_\ell^{ij} c_r^{\ell k}\neq 0$.
	If $i,k\neq 0$, then for $\ell=0$ we obtain $c^{ij}_\ell=0=c^{jk}_\ell$; hence, using associativity in the second equality below, we get \[\sum_{\ell\geq 0}  c_\ell^{ij} c_r^{\ell k} =\sum_{\ell\geq 1}  c_\ell^{ij} c_r^{\ell k}=\sum_{\ell\geq 1} c_\ell^{jk} c_r^{i\ell}=\sum_{\ell\geq 0} c_\ell^{jk} c_r^{i\ell},\] as desired.
	
	So we may assume that $i=0$ or $k=0$. Let us assume $i= 0$; the case $k=0$ can be treated similarly.
	%	If we also have $j=0$, then $\sum_{\ell\geq 0}  c_\ell^{ij} c_r^{\ell k}=c^{00}_0c^{0k}_r=c^{0k}_r=c^{0k}_kc^{0k}_r=\sum_{\ell\geq 0} c_\ell^{jk} c_r^{i\ell}$ (as all the other summands in both sums are zero).
	Then 
	$$\sum_{\ell\geq 0}  c_\ell^{ij} c_r^{\ell k}=c^{0j}_jc^{jk}_r=c^{jk}_r=c^{jk}_rc^{0r}_r=\sum_{\ell\geq 0} c_\ell^{jk} c_r^{i\ell}$$
	 as required.

\end{proof}

\begin{proposition}\label{many_algebras}
	Suppose $(\DD_u,\bar\epsilon_u)_{u=1}^n$ are local operator-systems and $r_u:\DD_u\to \DD_u\otimes \DD_u$ are associative of HS-iteration type. Put $\DD:=\DD_1\otimes \DD_2\otimes  \dots \otimes \DD_n$, and let $r:\DD\to \DD\otimes \DD$
	be obtained by composing $r_1\otimes r_2\dots \otimes r_n$ with the natural isomorphism 
	\[\DD_1\otimes\DD_1\otimes \DD_2\otimes\DD_2\otimes \dots \otimes \DD_n\otimes \DD_n\to \DD_1\otimes \DD_2\otimes \dots \otimes  \DD_n\otimes \DD_1\otimes\DD_2\otimes\dots \otimes \DD_n.\] 
	%permuting the co-ordinates. 
	%For each $1\leq u\leq n$ let $(\epsilon_{u,0}=1,\epsilon_{u,1},\dots,\epsilon_{u,m_u})$ be a $k$-linear basis of $\DD_s$. 
	%Consider the $k$-basis $(\epsilon)_{1\leq i_s\leq m_s}$
	
	Then, $\DD$ is a local operator-system with natural basis $\bar\epsilon=(\epsilon_{1,i_1}\otimes\cdots\otimes \epsilon_{n,i_n})_{i_j=0}^{m_j}$, and $r$ is associative of HS-iteration type. Furthermore, 
	\begin{itemize}
	\item [(i)] Given a $\DD$-ring $(R,e)=(R,\dd_{(i_1,\dots,i_n)})$, if we set $e_u:R\to\DD(R)$ to be
	$$e_u(r)=\sum_{i}\epsilon_{u,i}\otimes \dd_{u,i}(r)$$
	where $\dd_{u,i}:=\dd_{(0,\dots,0,i,0,\dots,0)}$ with $i$ in the $u$-th position, then each $(R,e_u)$ is a $\DD_u$-ring. In addition, $e$ commutes w.r.t. $r^\iota$ if and only 
	\begin{itemize}
	\item [(i.1)]  $e_u$ commutes with respect to $r_u^\iota$ for $u=1,\dots,n$,
		\item [(i.2)]  $\dd_{u,i}$  and $\dd_{v,j}$ commute for $u\neq v$, and
		\item [(i.3)] $\dd_{(i_1,\dots,i_n)}=\dd_{1,i_1}\cdots \dd_{n,i_n}$.
\end{itemize}
	
	\item [(ii)] Conversely, given $\DD_u$-ring structures $(R,e_u)=(R,\dd_{u,i})$ for $u=1,\dots,n$, if we set
	$e:R\to \DD(R)$ to be
	$$e(r)=\sum_{i_1,\dots,i_n}\epsilon_{1,i_1}\otimes \dots\otimes \epsilon_{n,i_n}\otimes \dd_{1,i_1}\cdots\dd_{n,i_n}(r)$$
	then $(R,e)$ is a $\DD$-ring. Moreover, $e$ commutes w.r.t. $r^\iota$ if and only if (i.1) and (i.2) hold.
	\end{itemize}
	%Moreover, if $F=k$ and all $r_i$ are associative, then $r$ is associative.
\end{proposition}
\begin{proof}
	By an easy induction, we may assume that $n=2$. Let us write $(\epsilon_{1,0},\dots, \epsilon_{1,m_1})=(\epsilon_0,\dots, \epsilon_m)$ and $(\epsilon_{2,0},\dots, \epsilon_{2,m_2})=(\epsilon'_0,\dots, \epsilon'_{m'})$. The fact that $\DD=\DD_1\otimes_k \DD_2$ is a local operator-system follows from a result of Sweedler \cite{Swee1975} stating that $\DD$ is local with maximal ideal
$$\mm=\mm_1\otimes_k \DD_2+\DD_1\otimes_k \mm_2.$$	
	
	For any $0\leq i,j,\ell\leq m_1$ let $c^{ij}_\ell$ be the coefficient of $\epsilon_i\otimes \epsilon_j$ in $r_1(\epsilon_\ell)$ and for any $0\leq i',j',\ell'\leq m_2$ let $c'^{i'j'}_{\ell'}$ be the coefficient of $\epsilon'_{i'}\otimes \epsilon'_{j'}$ in $r_2(\epsilon'_{\ell'})$. Then denoting by 
	$c^{(i,i'),(j,j')}_{(\ell,\ell')}$ the coefficient of  $\epsilon_{i}\otimes \epsilon'_{i'}\otimes \epsilon_j\otimes \epsilon'_{j'}$ in $r(\epsilon_\ell\otimes \epsilon'_{\ell'})$, we easily get that  \[c^{(i,i'),(j,j')}_{(\ell,\ell')}=c^{ij}_\ell c'^{i'j'}_{\ell'}\] so in particular $r$ is of HS-iteration type.
	
	Now assume $r_1$ and $r_2$ are associative. 	Then for any $0\leq i,j,k,r\leq m_1$ and $0\leq i',j',k',r'\leq m_2$, we have 
	\begin{align*}
	\sum_{\ell,\ell'\geq 0}  c_{(\ell,\ell')}^{(i,i'),(j,j')} c_{(r,r')}^{(\ell,\ell'), (k,k')} & =\sum_{\ell,\ell'\geq 0} c^{ij}_\ell c'^{i'j'}_{\ell'}c^{\ell k}_rc'^{\ell'k'}_{r'} \\
	& = (\sum_{\ell\geq 0} c^{ij}_\ell c^{\ell k}_r)( \sum_{\ell'\geq 0}c'^{i'j'}_{\ell'}c'^{\ell 'k'}_{r'}) \\
	& = (\sum_{\ell\geq 0} c^{jk}_\ell c^{i\ell }_r)( \sum_{\ell'\geq 0}c'^{j'k'}_{\ell'}c'^{i'\ell' }_{r'}) \\ 
	& =\sum_{\ell,\ell'\geq 0}  c_{(\ell,\ell')}^{(j,j'),(k,k')} c_{(r,r')}^{(i,i'),(\ell,\ell')}
	\end{align*}
	which gives associativity of $r$.
	
	We now prove (i). One readily checks that if $e$ is a $\DD$-structure on $R$ then $e_1$ is a $\DD_1$-structure and $e_2$ a $\DD_2$-structure. 
	%consider a function $e:R\to \DD(R)$ given by $e(r)=\sum_{i,i'\geq 0}\dd_{i,i'}(r)\epsilon_{i}\otimes \epsilon'_{i'}$. Clearly, $e$ is a $\DD$-operator on $R$ if and only if $e_1(r):=\sum_{i\geq 0}\dd_{i,0}(r)\epsilon_i$ defines a $\DD_1$-operator on $R$ and $e_2(r):=\sum_{i'\geq 0}\dd_{0,i'}(r)\epsilon'_{i'}$ defines a $\DD_2$-operator on $R$. So let us assume these equivalent conditions hold. 
	Now assume $e$ commutes with respect to $r^\iota$. Since $c^{(i,i'),(j,j')}_{(\ell,\ell')}=c^{ij}_\ell c'^{i'j'}_{\ell'}$, and $c_{\ell}^{ij}=1$ when $i=j=0$, condition (i.1) follows. By Lemma \ref{hscommuting}, for any $0\leq i\leq m_1$ and $0\leq j'\leq m_2$ we get 
	\[\dd_{(i,0)}\dd_{(0,j')}=\sum_{\ell,\ell'\geq 0} c^{(i,0),(0,j')}_{(\ell,\ell')}\dd_{(\ell,\ell')}=\sum_{\ell,\ell'\geq 0} c^{i0}_{\ell}c'^{0j'}_{\ell'}\dd_{(\ell,\ell')}=\dd_{i,j'},\] 
	yielding (i.2), and similarly we get $\dd_{(0,j')}\dd_{(i,0)}=\dd_{i,j'}$, hence getting (i.3).
	
	Conversely, assume conditions (i.1)-(i.3) are satisfied. Then for any $0\leq i,j\leq m_1$ and $0\leq i',j'\leq m_2$ we have 
	\begin{align*}
	\dd_{(i,i')}\dd_{(j,j')} & =\dd_{1,i}\dd_{1,j}\dd_{2,i'}\dd_{2,j'} \\
	& =\sum_{\ell}c^{ij}_\ell\dd_{1,\ell}\sum_{\ell'}c'^{i'j'}_{\ell'}\dd_{2,\ell'} \\
	& =\sum_{\ell}\sum_{\ell'} c^{ij}_\ell c^{i'j'}_{\ell'}\dd_{1,\ell}\dd_{2,\ell'} \\
	& =\sum_{\ell,\ell'}c^{(i,i'),(j,j')}_{(\ell,\ell')}\dd_{(\ell,\ell')}
	\end{align*}
	which gives that $e$ commutes with respect to $r^{\iota}$ by Lemma \ref{hscommuting}.
	
	\smallskip
	The proof of (ii) is similar. Details are left to the reader.
	
	%	Suppose first that $e$ is a an $r^\iota$-commuting $\DD$-operator. In particular, $e$ is a homor

\end{proof}

\end{document}